\documentclass{amsart}

\usepackage{enumitem}
\usepackage[charter]{mathdesign}
\usepackage{amsmath}                

\usepackage{amsfonts}
\usepackage{graphicx, subfigure}
\usepackage{graphics}
\usepackage{epstopdf}
\usepackage{amsthm}

\usepackage[usenames, dvipsnames, pdftex]{color}
\usepackage[colorlinks=true,
        raiselinks=true,
        linkcolor=MidnightBlue,
        citecolor=ForestGreen,
        urlcolor=Mahogany,
        plainpages=false]{hyperref}

\usepackage{tikz}
\usepackage[all]{xy}
\usetikzlibrary{arrows,automata}
\usetikzlibrary{shapes,snakes}
\usetikzlibrary{positioning}

\usepackage{mathrsfs}

    \numberwithin{equation}{section}

    \newcommand{\cdt}{{\tt cdt}}

    \newcommand{\lcdt}{{\tt lcdt}}

    \def\true{\texttt{true}}

    \def\N{\mathbb{N}}
    \def\K{\mathbb{K}}
    \def\R{\mathbb{R}}

    \def\pr{\mathsf{P}}
    
    \def\Str{\mathsf{S}}
    \def\T{\mathsf{T}}
    \def\M{\mathsf{M}}
    
    \def\t{\mathsf{t}}
    \def\k{\mathsf{k}}
    \def\L{\mathsf{L}}
    \def\U{\mathsf{U}}
    \def\W{\mathsf{W}}
    \def\X{\mathsf{X}}
    \def\gr{\mathsf{Gr}}
    \def\proj{\mathsf{proj}}
    \def\acc{\mathsf{Acc}}
    \def\id{\mathrm{id}}

    \def\A{\mathfrak{A}}
    \def\D{\mathfrak{D}}
    \def\G{\mathfrak{G}}
    \def\l{\mathfrak{L}}
    \def\r{\mathsf{R}}
    \def\I{\mathfrak{i}}
    \def\P{\mathfrak{p}}

    \def\d{\mathrm{d}}
    \def\xx{\mathbf{x}}
    \def\qq{\mathbf{q}}
    \def\uu{\mathbf{u}}
    \def\vv{\mathbf{v}}
    \def\hh{\mathbf{h}}

    \def\b{\mathcal{B}}
    \def\p{\mathcal{P}}
    \def\a{\mathcal{A}}
    \def\TS{\mathcal{T}}
    \def\f{\mathcal{F}}
    \def\x{\mathcal{X}}
    \def\y{\mathcal{Y}}
    \def\s{\mathcal{S}}
    \def\u{\mathcal{U}}

    \def\lsa{\a_*}
    \def\usa{\a^*}

    \def\bu{\mathrm{b}\u}
    \def\bc{\mathrm{b}\c}
    \def\bb{\mathrm{b}\b}
    \def\ba{\mathrm{b}\a}

    \def\blsa{\mathrm{b}\lsa}
    \def\busa{\mathrm{b}\usa}

    \def\c{\mathcal{C}}

    \def\ve{\varepsilon}

  {
      \theoremstyle{plain}
      \newtheorem{theorem}{Theorem}
      \newtheorem{lemma}{Lemma}

      \newtheorem{problem}{Problem}
      \newtheorem{proposition}{Proposition}
      \newtheorem{assumption}{Assumption}
      \newtheorem{example}{Example}
      
      \newtheorem{definition}{Definition}
  }

\begin{document}

\title[Quantitative model-checking of cdt-MP]{
  Quantitative model-checking of \\ controlled discrete-time Markov processes
}

\author{
    Ilya Tkachev,
    Alexandru Mereacre,
    Joost-Pieter Katoen, and
    Alessandro Abate
}\thanks{
\hspace{-0.6cm}
I. Tkachev is with the Delft Center for Systems \& Control, Delft University of Technology, The Netherlands. Email: \texttt{i.tkachev@tudelft.nl}. \\
A. Mereacre is with the Department of Computer Science, University of Oxford, United Kingdom. Email: \texttt{alexandru.mereacre@cs.ox.ac.uk}. \\
J.-P. Katoen is with the Software Modeling and Verification Group, RWTH Aachen University, Germany. Email: \texttt{katoen@cs.rwth-aachen.de}. \\
A. Abate is with the Department of Computer Science, University of Oxford, United Kingdom, and with the Delft Center for Systems \& Control, Delft University of Technology, The Netherlands. Email: \texttt{alessandro.abate@cs.ox.ac.uk}.
}

\maketitle

\begin{abstract}
  This paper focuses on optimizing probabilities of events of interest defined over general controlled discrete-time Markov processes.
  It is shown that the optimization over a wide class of $\omega$-regular properties can be reduced to the solution of one of two fundamental problems:
  reachability and repeated reachability.
  We provide a comprehensive study of the former problem and an initial characterisation of the (much more involved) latter problem.
  A case study elucidates concepts and techniques.
\end{abstract}

\section{Introduction}

Stochastic control models have been widely investigated and employed in numerous applications in different areas such as finance, biology, power networks, etc. -- see \cite[Chapter 1]{hll1996} or \cite{m2008} for examples.
Under discrete time semantics,
a natural way to model probabilistic behavior allowing for the presence of control inputs is to employ the framework of controlled discrete-time Markov processes (\cdt-MP),
also known as general Markov Decision Processes (MDP) \cite{fs2002}.
In this modeling formalism,
given the current state of a system and the control action provided by an external agent,
the distribution of the next state is uniquely (deterministically) determined,
which also entails to the Markovian structure of the model.
In turn,
the choice of the control action itself may depend on the complete history of state and control observations,
and can be randomized.
The decision rule of the agent,
which assigns to the history observation a choice of the next action,
is called the policy.

A generic optimization problem over a \cdt-MP is the following:
given a performance criterion whose value is uniquely determined by a chosen policy \cite{f1983},
optimize (maximize or minimize) the value of this criterion over the given class of policies,
and determine (if possible) the policy corresponding to the optimal value.
In the literature a wide range of performance criteria has been studied --
see e.g. \cite[Section 3]{abfggm1993} for remarks on the historical development of the topic --
among them the discounted cost (DC),
the total cost (TC),
and the average cost (AC).
All these criteria present an additive structure,
which allows for the solution by means of dynamic programming (DP) \cite{b1954},
namely a backward-recursive procedure that computes the optimal control action by balancing the present value of the cost and the expected future cost caused by the choice of such an action.
The DP approach has led to a rich theory for such criteria -- see \cite{bs1978} for an overview on the DC and TC, and \cite{abfggm1993} for a survey on AC.
Unfortunately similar results for other sorts of criteria are much less comprehensive,
the focus in the literature being more on qualitative analysis,
e.g. determining which policy classes are sufficient to focus on,
and no general solution techniques have been developed to the best of our knowledge.
This in particular is the case when one wants to optimize the probability of a given event,
examples of the latter being ``the state trajectory never leaves the safe set $S$'' or ``the state trajectory eventually reaches the goal set $G$ without leaving the safe set $S$ beforehand''.
Instances of these problems have been studied in isolation \cite{mps1991,ms1996a},
however no comprehensive treatment for this general class of problems has been given.

In this work we apply methods grounded on modal logic and on automata theory for the following two purposes:
first, we develop a framework to quantitatively define a class of performance criteria of interest, encompassing the instances discussed above;
second, we solve optimization problems over such criteria in a unified way.
More specifically,
we propose to express events as formulae within a linear temporal logic (LTL) \cite[Chapter 5]{bk2008},
encompassing intuitive specifications on the model that are related to sentences in natural languages.
We further show that such formulae can be recast as automata:
simple dynamical systems endowed with a logical structure given by their acceptance conditions \cite[Chapter 4]{bk2008}.
We prove that the optimization of any given event expressed as an automaton over the original \cdt-MP model can be reduced to one of two fundamental problems,
namely
reachability or repeated reachability:
the former requires visiting a goal set at least once,
whereas the latter requires infinitely many visits to the goal set.

The reachability problem over \cdt-MP has been recently studied e.g. in \cite{APLS08b,SL10,CCL11},
however the results have either required restrictive conditions the on model or focused on special cases of the problem,
for instance when only Markov (history-independent) policies are allowed.
In contrast, here we consider the most general setting for the reachability problem,
and we provide a complete treatment of the problem under conditions on the model being as mild as possible:
this is considered to be the core of our contribution.
For example, up to our knowledge we are the first to give a comprehensive study of the unbounded-horizon reachability over \cdt-MP,
providing Lyapunov-like techniques for its solution.
In order to obtain these results,
we show that the reachability performance criterion can be expressed as a TC one over a modified \cdt-MP,
which allows us extending the rich theory for the latter criterion to the reachability case.
Unfortunately,
we are not able to give a comparable study of the repeated reachability problem,
however we extend results from gambling theory \cite{ms1996a} to characterize the DP formulation for this problem,
and propose a solution using Lyapunov-like excessive functions in the special case when the system possesses certain stability properties.

An approach to the optimal control of \cdt-MP based on LTL and automata has been developed for finite-state and finite-action models in the model-checking literature \cite[Section 10.6]{bk2008}.
Due to this reason, our contribution can be considered from two perspectives.
For readers familiar with formal methods in control \cite{t2009} and model-checking \cite{bk2008},
we extend the model-checking techniques from finite \cdt-MP to a general class of models,
whereas for readers experienced in classical stochastic optimal control we propose a novel formulation and solution of the problem of optimization of probabilities of events of interest.

The rest of the paper is organized as follows.
The model description and the problem formulation are given in Section \ref{sec:models},
which also puts forward the result on reduction of the general problem to either of two fundamental ones: reachability or repeated reachability.
Section \ref{sec:reach} is devoted to the former case,
whereas Section \ref{sec:pers} is focused on the latter instance.
We give an elucidating numerical case study in Section \ref{sec:cs},
and the paper is concluded in Section \ref{sec:concl}.
The notation, background in analysis and measure theory,
special subclasses of LTL and auxiliary results are given in the Appendix.

\section{Models and problem formulation}
\label{sec:models}

\subsection{Model syntax and semantics}
\label{ssec:model}

The models considered in this work are known as controlled discrete-time Markov processes (\cdt-MP).
A \cdt-MP is a discrete-time stochastic model with a specific transition structure:
the distribution of the next state of the process is completely determined by the current state and the current choice of the control action.
These models are alternatively known in the literature as controlled Markov models \cite{hll1996},
general Markov Decision Processes (MDP) \cite{P1994} or gambling houses \cite{ms1996}.
There are often slight variations in the their definition:
the one we give here is inspired by the Borel model introduced in \cite[Chapters 8, 9]{bs1978}.
Details on notation can be found in the Appendix.

\begin{definition}
\label{def:cdt-MP}
  A \emph{\cdt-MP} is a tuple $\D = (X,U,\K,\T)$,
  where $X$ and $U$ are non-empty Borel spaces,
  $\K$ is an analytic subset of $X\times U$,
  and $\T\in \b(X|\,X\times U)$ is a stochastic kernel.

  The \cdt-MP $\D$ is called \emph{continuous} if $U$ is a compact Borel space,
  $\K$ is a closed subset of $X\times U$ and the restriction $\T|_\K$ is a continuous kernel.
\end{definition}

Given a \cdt-MP $\D = (X,U,\K,\T)$ we say that $X$ is its state space,
$U$ is the action space, $\K_x$ are the actions that are feasible at state $x\in X$, and $\T$ is the transition kernel.
The latter induces several operators that act on functions defined over the state space.
For any $\mu\in \u(U|X)$ and any function $f\in \bu(X)$ we define
\begin{equation*}
  \T^\mu f(x) := \int_{X\times U}f(x')\T(\d x'|x,u)\mu(\d u|x).
\end{equation*}
In particular, when $\mu = \delta_u$ is a constant kernel,
where $u$ is some element of $U$,
we simply write $\T^{u}$ rather than $\T^{\delta_u}$.
Clearly, it holds that $\T^u1_A(x) = \T(A|x,u)$ for all $x\in X$,
$u\in U$ and any $A\in \b(X)$.
Furthermore, \cite[Proposition 7.46]{bs1978} implies that $T^\mu$ maps $\bu(X)$ to itself.
We also consider the following operators:
\begin{equation*}
  \T^* f(x) := \sup_{u\in U}\T^u f(x), \qquad \T_* f(x) := \inf_{u\in U}\T^u f(x).
\end{equation*}
If $f\in \blsa(X)$, then $\T^u f\in \blsa(X)$ thanks to \cite[Proposition 7.48]{bs1978}.
Furthermore, it follows from \cite[Proposition 7.47]{bs1978} that $\T_* f(x) \in \blsa(X)$ as well and,
as a result, the operator $\T_*$ maps the space $\blsa(X)$ into itself.
Similar arguments show that the operator $\T^*$ maps the space $\busa(X)$ into itself.

The semantics of the \cdt-MP $\D$ is given as follows: at any time instant $k\in \N_0$, if the state of $\D$ is $x_k\in X$ and the action $u_k\in U(x_k)$ is chosen,
then the new state $x_{k+1}$ is a random variable distributed according to the following law:
\begin{equation}\label{eq:dyn-distr}
  x_{k+1} \sim \T(\cdot|x_k,u_k).
\end{equation}
As a known example, every stochastic difference equation of the form
\begin{equation}\label{eq:dyn-eq}
  x_{k+1} = F(x_k,u_k,\xi_k),
\end{equation}
where $(\xi_k)_{k\in \N_0}$ is a sequence of iid random variables and the map $F:X\times U\times \Bbb R \to X$ is Borel measurable,
can be represented as in \eqref{eq:dyn-distr}.
In this case the kernel $\T$ can be expressed via the map $F$ as
\begin{equation*}
  \T(B|x,u) = \nu(\{\xi\in \R: F(x,u,\xi)\in B\}),
\end{equation*}
for any $B\in \b(X)$,
where $\nu$ is the distribution of $\xi_0$.
On the other hand, the converse statement also holds true,
though there is no constructive method to derive an $F$ from a given $\T$ \cite[Section 2.3]{hll1996}.
Although \eqref{eq:dyn-eq} may be more intuitive or familiar,
the representation of the dynamics as in \eqref{eq:dyn-distr} is preferred in this work.
Note also that if $F$ as in \eqref{eq:dyn-eq} is such that $F(\cdot,\xi):X\times U \to X$ is a continuous map,
then the corresponding kernel is continuous as well \cite[Example C.7]{hll1996}.

A formal definition of the evolution of a \cdt-MP is given by its paths and by the corresponding probability measures on the path space. More precisely:

\begin{definition}
\label{def:path}
  Given a \cdt-MP $\D$, its \emph{infinite path} is an infinite sequence
  \begin{equation}
  \label{eq:inf.path}
    h = (x_0,u_0,x_1,u_1,\dots),
  \end{equation}
  where $x_k\in X$ are the state coordinates and $u_k\in U$ are the action coordinates of the path, $k\in \N_0$.
  The space of all infinite paths is denoted by $H := (X\times A)^{\N_0}$ and is called the \emph{canonical sample space} of the \cdt-MP $\D$.

  For $n\in \N_0$, a \emph{finite} $n$\emph{-path} $h_n$ is a finite prefix of an infinite path ending in a state:
  \begin{equation}\label{eq:fin.path}
    h_n = (x_0,u_0,\dots,x_{n-1},u_{n-1},x_n),
  \end{equation}
  where $x_k\in X$ and $u_k\in U$.
  The space of all $n$-paths is denoted by $H_n = (X\times A)^n\times X$.
\end{definition}

Infinite paths of \cdt-MP are mostly used to introduce certain performance criteria over the model,
whereas finite $n$-paths naturally serve as the history of observation available up to instant $n$.
Due to this reason, we use notation $H$ and $H_n$ for the spaces of paths,
and below we often refer to finite paths as \emph{histories}.

Similarly to \cite{abfggm1993},
we define the state, action, and information processes on a sample space $H$.
They are denoted respectively by $(\xx_n)_{n\in \N_0}$, $(\uu_n)_{n\in \N_0}$ and $(\hh_n)_{n\in \N_0}$,
and are defined by the following projections on spaces $X$, $U$ and $H_n$:
\begin{equation*}
  \xx_n(h) := x_n,\quad \uu_n(h) := u_n,\quad \hh_n(h) := (x_0,u_0,\dots,x_{n-1},u_{n-1},x_n),\quad n\in\N_0,
\end{equation*}
for any $h \in H$ as per \eqref{eq:inf.path}.
Notice that it may happen that $\uu_k(h) \notin \K_{\xx_k(h)}$,
which reflects action coordinates that are not feasible:
this is allowed for technical reasons and later we show that the corresponding paths are of measure zero.

When dealing with stochastic processes, questions of measurability are crucial to render objects well-defined. This in particular applies to the choice of action $u_n$ at time $n$,
given the history $h_n$, and is formalized using the notion of a policy.\footnote{
  Policies are also known as strategies \cite{ms1996}, alternatively as schedulers or adversaries \cite{bk2008}. In the latter case they are used to resolve non-determinism in non-deterministic stochastic models, such as probabilistic automata \cite{sl1995}.
}

\begin{definition}
\label{def:policy}
  Given a \cdt-MP $\D$,
  a \emph{policy} is a sequence $\pi = (\pi_n)_{n\in \N_0}$ of universally measurable kernels $\pi_n \in \u(U|H_n)$,
  which is such that for any $h_n$ as in \eqref{eq:fin.path} it holds that
  \begin{equation}
  \label{eq:policy-feas}
    \pi_n(U(x_n)|h_n) = 1.
  \end{equation}
  The class of all policies of $\D$ is denoted by $\Pi$.
\end{definition}

Notice that \eqref{eq:policy-feas} implies that policies are only allowed to select actions among the currently feasible ones.
Once a policy $\pi\in \Pi$ and an initial distribution $\alpha\in \p(X)$ are fixed,
the behavior of a \cdt-MP $\D$ is completely characterized by the probability measure $\pr^\pi_\alpha\in \p(H)$ over the path space $H$.
This measure is uniquely defined by
\begin{equation}\label{eq:IT}
\begin{split}
  \int_H f\d \pr^\pi_\alpha = &\int_X\int_U\int_X\dots\int_X\int_U f(x_0,u_0,x_1,\dots,x_{n-1},u_{n-1},x_n,\dots)
  \\[2mm]
  &\times \T(\d x_n|x_{n-1},u_{n-1})\pi_{n-1}(\d u_{n-1}|x_0,u_0,\dots,x_{n-1})
  \\[2mm]
  &\times \T(\d x_{n-1}|x_{n-2},u_{n-2})\cdots \T(\d x_1|x_0,u_0)\pi_0(\d u_0|x_0)\alpha(\d x_0),
\end{split}
\end{equation}
for any bounded $\hh_n$-measurable function $f:\Omega\to \R$ \cite[Sections 8.1, 9.1]{bs1978}.
In particular, $\pr^\pi_\alpha(\K^{\N_0}) = 1$ so that (as we anticipated) the probability of paths containing non-feasible actions is equal to zero.
Moreover, by taking $f$ to be an appropriate indicator function, for any sets $B\in \b(X)$ and $C\in \b(U)$ we obtain the following equalities that hold $\pr^\pi_\alpha$-a.s.:
\begin{align}
  \pr^\pi_\alpha(\xx_0\in B) &= \alpha(B),
  \\
  \pr^\pi_\alpha(\uu_n\in C|\hh_n) &= \pi_n(C|\hh_n),
  \\
  \pr^\pi_\alpha(\xx_{n+1}\in B|\hh_n,\uu_n) &= \T(B|\xx_n,\uu_n). \label{eq:Markov}
\end{align}	
As a result, the probability measure $\pr^\pi_\alpha$ captures all the intuitive features about the behavior of the \cdt-MP $\D$ under the selected policy $\pi$ and given the initial distribution $\alpha$.
In particular, \eqref{eq:Markov} implies that the distribution of $\xx_{n+1}$ only depends on $\xx_n$ and $\uu_n$,
rather than on the whole history $\hh_n$.
Note, however, that the chosen control action $\uu_n$ itself can depend on the history rather than only on the current state $\xx_n$.
We say that $\left(H,\b(H),(\pr^\pi_\alpha)^{\pi\in \Pi}_{\alpha\in \p(X)}\right)$ is the canonical probability space for the \cdt-MP $\D$.\footnote{
  We slightly abuse the notation here since in fact this is a family of probability spaces parameterized by $\alpha\in \p(X)$ and $\pi\in\Pi$,
  rather than a single probability space.
}
Finally, as a shorthand,
the notation $\pr^\pi_x$ is used in place of $\pr^\pi_{\delta_x}$.

We conclude the discussion by highlighting the following important classes of policies over the \cdt-MP $\D$.
\begin{itemize}
  \item $\Pi_M$ -- the class of \emph{Markov} policies.
  A policy $\pi\in \Pi$ is called Markov if for any $n\in \N_0$ the measure $\pi_n(h_n)$ depends only on $x_n$ for any finite path $h_n\in H_n$ as per \eqref{eq:fin.path}.
  More precisely: for all $n\in \N_0$ and $x_n\in X$
  \begin{equation*}
    \pi_n(h',x_n) = \pi_n(h'',x_n), \quad \forall h',h''\in H_{n-1}\times U.
  \end{equation*}
  This means that a Markov policy selects an action solely based on the information about the current state,
  rather than on the whole available history.
  \item $\Pi_S\subseteq \Pi_M$ -- the class of \emph{stationary} policies.
  A Markov policy $\pi$ is called stationary if $\pi_n(x) = \pi_{n+1}(x)$ for all $n\in \N_0$ and $x\in X$.
  Thus, stationary policies are time-independent.
  \item $\Pi^D$ -- the class of \emph{deterministic} policies.
  A policy $\pi\in \Pi$ is called deterministic if $\pi_n = \delta_{f}$ for some universally measurable map $f:H_n\to A$.
  \item $\Pi^D_S \subseteq \Pi^D_M\subseteq \Pi^D$ -- classes of deterministic Markov and deterministic stationary policies.
  Such classes are defined by $\Pi^D_M = \Pi^D\cap \Pi_M$ and $\Pi^D_S = \Pi^D\cap \Pi_S$.
\end{itemize}
\smallskip

We refer to any map $f:X\to U$ satisfying $\gr[f]\subseteq \K$ as a selector (from $\K$),
whereas a stochastic kernel $\mu\in \u(U|X)$ satisfying $\mu(\K_x|x) = 1$ for all $x\in X$ is called a randomized selector.
Clearly, the existence of the former (of the latter) is equivalent to the statement that the policy class $\Pi^D_S$ ($\Pi_S$) is not empty.
Notice that $\Pi^D_S$ is the smallest among the classes of policies introduced above,
and since $\K$ is analytic,
it admits an analytically measurable selector,
namely it contains the graph of an analytically measurable map $\k:X\to U$ \cite[Proposition 7.49]{bs1978}.
As a result, $\Pi^D_S$ is not empty,
hence neither are all other classes.

\subsection{Example: a small power network}
\label{ssec:power-net}

In order to elucidate the concepts introduced above,
let us discuss the following example,
modified from the one in \cite{tmka2013}.
Figure \ref{fig:power-net} schematically depicts the setup.

\begin{figure*}[h]\label{fig:power-net}
    \centering
    \includegraphics[keepaspectratio=true,width=12cm]{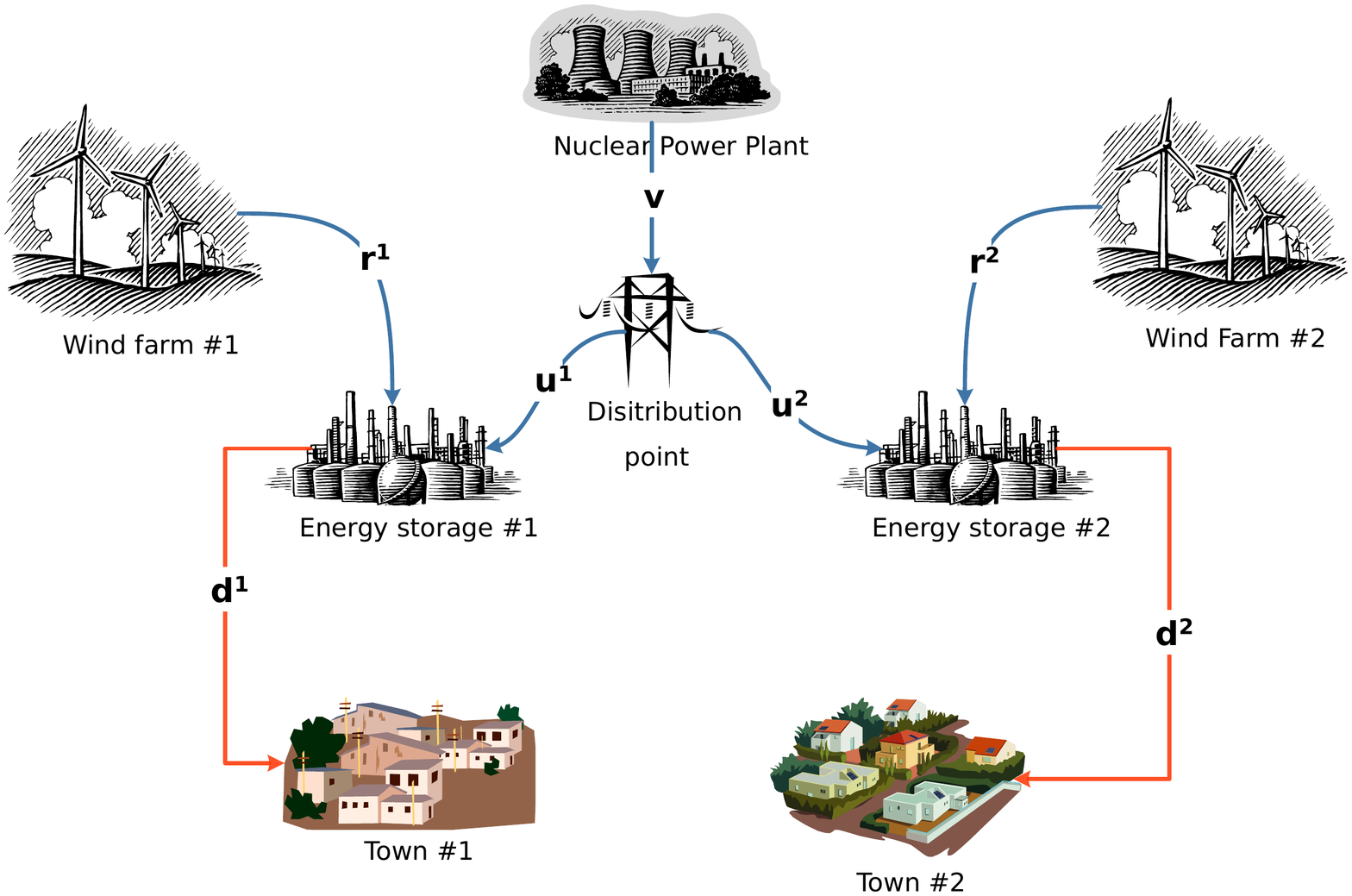}
    \caption{Case study: a power network consisting of two sub-networks.}
\end{figure*}

Consider a simple, abstract power network consisting of two aggregated consumers (e.g. small towns),
each of which benefits from a separate generator of renewable energy (e.g. a wind farm) and a separate energy storage.
Suppose that in addition there is a shared polluting power generator,
such as a nuclear power plant.
The energy flow is assumed to be stochastic,
in particular due to the production deriving from the wind farms.
The energy output of the nuclear power plant is less volatile and larger.
Within this setup one requires that the energy supply is greater than the energy demand,
or may impose some additional requirements on the energy levels.
The available control is the total load on the nuclear power plant,
as well as its distribution over the two consumers.
More precisely,
the model is given as follows:
\begin{equation}\label{eq:cs}
  \xx^i_{k+1} = c\cdot\left(\xx^i_k + \uu^i_k \cdot \vv_k\cdot p_k + r^i_k - d^i_k\right)\wedge M\vee 0,
\end{equation}
where $\xx^i_k \in [0,M]$ is the energy level in the subnetwork $i\in \{1,2\}$ at the discrete time instance $k\in \N_0$,
and $M>0$ is the maximal storage capacity.
The constant $c\in (0,1]$ is the reserve rate of the stored energy,
$\vv_k \in [v_{\min},1]$ is the load on the nuclear power plant and $\uu^i_k \in [0,1]$ is the share of energy produced by the nuclear power plant that is supplied to the subnetwork $i$,
so that $\uu^1_k + \uu^2_k \equiv 1$.
As we assume that it is not possible to switch the nuclear power plant off, $v_{\min}$ is the minimal load on the plant.
The noise is represented by a sequence of iid random variables accounting for uncertainty on the nuclear power plant actual production $p_k$,
the wind farm production $r_k = (r^1_k,r^2_k)$,
and the total local demand $d_k = (d^1_k,d^2_k)$.
Note that $r^1_k$ and $r^2_k$ are not necessarily independent (coupling can be due to weather),
and neither are the demand variables for the subnetworks $d^1_k$ and $d^2_k$.

A \cdt-MP model for the dynamics above is given considering a state space $X = [0,M]^2$ with $\xx_k  = (\xx^1_k,\xx^2_k)$,
control space $U = [0,1]^2$ with $\uu_k = (\vv_k,\uu^1_k)$,
and control actions that are always feasible (namely $\K = X\times U$),
and a transition kernel induced by the stochastic difference equation \eqref{eq:cs}.
Goals for control synthesis are discussed shortly,
whereas the analysis of the model and the synthesis problem are presented in Section \ref{sec:cs}.

\subsection{Problem formulation}
\label{ssec:prob}

The framework of \cdt-MP is often used in optimization.
In particular, one of the most prominent questions to answer is the following:
what is the maximal achievable value of a given performance measure,
and can a control policy that achieves such a value be derived?
Clearly, the answer crucially depends on the chosen criterion:
this choice is quite broad in the literature on \cdt-MP,
so let us discuss some important cases.\footnote{
  A comprehensive survey on different performance criteria,
  as well as on the general development of the theory of \cdt-MP,
  is given in \cite[Section 3]{abfggm1993}.
}

We do not consider multi-objective optimization where the performance criterion has a partial order on its co-domain (see e.g. \cite{Borkar1991}),
and instead focus on numerical criteria, namely measures taking values on $\R$.
Arguably one of the most general approaches to the definition of numerical performance criteria over \cdt-MP has been considered in \cite{f1983}.
There, the focus is on the space of \emph{strategic} measures given by $\Str := \{\pr^\pi_\alpha|\pi\in \Pi, \alpha\in \p(X)\}$,
and the criterion is simply any function $f:\Str \to \R$.
A slightly more specific class of criteria is related to the concept of the \emph{expected utility} \cite{k1977a,k1977b,k1978}.
A utility is any history-dependent random variable $J:H\to \R$ and the corresponding performance is defined to be its expected value $\M^\pi(\alpha;J):=\pr^\pi_\alpha[J]$.
Clearly, the expected utility criterion is a special case of the former,
since to any utility $J$ one can assign a function $f_J:\Str \to \bar \R$ by defining $f_J(p) := p[J]$ for any $p\in \Str$.
Research on these criteria has led to strong theoretical results,
e.g. on the characterization of classes of optimal policies.
On the other hand, the generality of the problems did not allow for specific results related to the computability of the optimal solutions.
Due to this reason, more specific performance criteria have attracted a significant interest,
in particular the \emph{discounted cost} (DC) and the \emph{average cost} (AC) criteria \cite{abfggm1993}.
Consider some universally measurable \emph{cost function} $c:\K\to \bar\R$ and define
\begin{align*}
  \mathsf{DC}_{n,\gamma}^\pi(x) := \pr^\pi_x\left[\sum_{k=0}^n \gamma^k c(\xx_k,\uu_k)\right], 
  \\
  \mathsf{AC}^\pi(x) := \limsup_{n\to\infty}\frac1n\pr^\pi_x\left[\sum_{k=0}^n c(\xx_k,\uu_k)\right] 
\end{align*}
where $\gamma \in (0,1]$ is the \emph{discounting factor} and $n\in \bar\N_0$ is the \emph{time horizon}.
The DC is clearly a special case of the expected utility criterion.
In general it is not possible to express $\mathsf{AC}$ as an expected utility,
but clearly it is still a function of strategic measures and thus belongs to the class of criteria considered in \cite{f1983}.
Furthermore, with focus on the DC,
the case $\gamma = 1$ is often referred to as the \emph{total cost} (TC) criterion or,
alternatively, the \emph{additive cost}.
These problems are extensively studied in the literature:
see e.g. \cite{bs1978,hll1996} for the DC,
and \cite{abfggm1993} for the AC.

The focus of this paper is on the probabilities associated to certain events defined over the paths of the \cdt-MP.
More precisely, let $A \in \b(H)$ be some set of desired path behaviors of the \cdt-MP,
and consider a performance criterion to be $\pr^\pi_\alpha(A)$.
Clearly, this is still a special case of the expected utility criterion,
with the utility given by $1_A$,
and thus general results apply.
However, if we focus on a certain class of events,
rather than considering all possible elements of $\b(H)$,
it is possible to obtain much stronger results in terms of characterization and of computability.
More specifically, we exploit the known approach in formal methods \cite{bk2008} to treat any event as a property (or a \emph{specification}) over paths of a \cdt-MP.
Such property is further expressed as a simple dynamical model satisfying it.
This technique has been widely employed to study  \cdt-MP models over finite state and action spaces \cite{cy1998},
leading to analytical solutions for that setup.
However, the developed methods appear to be crucially dependent on the discrete structure of finite \cdt-MP and thus are not fully applicable to the general case.
The aim of this work is to develop new techniques to tackle this problem over \emph{general} \cdt-MP.

Before we describe the class of events of interest,
let us introduce some notation for the expected utility criterion.
Given an initial distribution $\alpha\in \p(X)$,
a policy $\pi \in \Pi$ and a random variable $f\in \bb(H)$,
we denote $\M^\pi(\alpha;f):=\pr^\pi_\alpha [f]$.
In the particular case when $f = 1_A$ for some $A\in \b(H)$ or $\alpha = \delta_x$,
we simply write $\M^\pi(\alpha;A)$ or $\M^\pi(x;A)$.
The optimal expected utility functions are defined as
\begin{equation*}
  \M^*(\alpha;f):=\sup_{\pi\in \Pi} \M^\pi(\alpha;f),\qquad \M_*(\alpha;f):=\inf_{\pi\in \Pi} \M^\pi(\alpha;f).
\end{equation*}

\medskip

In order to formulate the problem,
we need to specify the class of events we focus on.
Recall the power network model from Section \ref{ssec:power-net},
and consider the following tasks:
\begin{itemize}
  \item keep the energy levels always within specified target levels;
  \item test the network as follows:
  reach an energy level above the target value over the first subnetwork,
  and while keeping it there,
  reach the same energy level over the second network,
  or do vice-versa.
  In addition,
  avoid blackouts,
  that is never allow an energy level of either of the subnetworks to reach the zero level.
\end{itemize}
The first task corresponds to a \emph{safety} problem,
which can be easily characterized using canonical probabilistic tools and the concept of the first hitting time.
On the other hand, the second task is more complicated, even in its qualitative description.
For this purpose we introduce a modal logic called Linear Temporal Logic (LTL), which is useful in the following two aspects.
First of all, it provides \emph{``a very intuitive but mathematically precise notation''} \cite[Section 5.1]{bk2008} to deal with a large class of complex and interesting events.
Secondly,
LTL
allows reducing the optimization problems for any of such events to one of the following two fundamental problems:
\emph{reachability}, requiring visiting a specified target set at least once;
or \emph{constrained repeated reachability}, requiring visiting a target set infinitely often and visiting an unsafe set only finitely often.

LTL is introduced using its \emph{grammar}, namely the set of rules determining the construction of LTL formulae.
The meaning of each formula (that is, the event corresponding to the formula) is formalized by the LTL \emph{semantics}.
It is canonical to introduce the latter not directly over the state space,
but rather using the concept of \emph{labels},
namely discrete observations of states that range over some finite set called the \emph{alphabet}.
Alternatively, one can think of assigning some distinguishable sets to the state space.
Intuitively, when one says that $x\in A$ it may be considered as an implicit assignment of the label ``$A$'' to a point $x$.

Let $\Sigma$ be a finite set,
which is referred to as the alphabet.
Elements of $\Sigma$ are called \emph{letters},
whereas finite or infinite sequences of letters are called \emph{words}.
Let us denote by $\Sigma^\omega$ the space of infinite words;
by infinite \emph{language} over $\Sigma$ we mean any collection of infinite words over the alphabet $\Sigma$.
All the languages we consider in this paper are assumed to be infinite,
i.e. we say that $\phi$ is a language to mean $\phi\subseteq \Sigma^\omega$.
If $\phi\in \b(\Sigma^\omega)$ we say that $\phi$ is a measurable language.
In particular, it follows from \cite[Proposition 2.3]{v1985} that any $\omega$-regular language\footnote{
  The definition of $\omega$-regular languages is lengthy and is omitted from this paper for the sake of clarity in presentation.
  For a formal definition see e.g. \cite[Section 4.3.1]{bk2008}
}
is measurable.
On the other hand, not any measurable language is $\omega$-regular:
clearly any singleton $\{w\}$ generated by a word $w\in \Sigma^\omega$ is measurable,
but the language $\{w\}$ may not be $\omega$-regular if $w$ is not a periodic word.
It is also easy to construct an example of a non-measurable language:
since $\Sigma^\omega$ is an uncountable Borel space,
there is a Borel isomorphism $f:\Sigma^\omega \to [0,1]$,
so for any non-Borel set $A\subseteq [0,1]$ the language $f^{-1}(A)\subseteq \Sigma^\omega$ is not measurable.
We first show how we interpret languages as events in the canonical sample space $H$,
and then introduce specific languages characterized by LTL formulae.

Consider a Borel measurable map $\L:X\to \Sigma$,
further called a \emph{labelling} map.
We call a triple $(\D,\Sigma,\L)$ a \emph{labelled} \cdt-MP (\lcdt-MP for short):
in a \lcdt-MP each state $x\in X$ is assigned to a letter $\L(x)\in \Sigma$.
As a result, to each path $h \in H$ there corresponds a unique \emph{trace} word $w\in \Sigma^\omega$,
also known as a \emph{trace} of $h$,
which is given by
\begin{equation}\label{eq:L-omega}
  \L_\omega(x_0,u_0,x_1,u_1,\dots) := (\L(x_0),\L(x_1),\dots).
\end{equation}
We consider \eqref{eq:L-omega} as the definition of the \emph{trace} map $\L_\omega:H\to\Sigma^\omega$.

\begin{proposition}\label{prop:L-meas}
  The map $\L_\omega$ is Borel measurable.
\end{proposition}

\begin{proof}
  Recall that $\b(\Sigma^\omega) =\sigma(\c)$,
  where the class $\c$ of cylinder sets is given by
  \begin{equation}
    \c = \left\{\left.\prod_{k=0}^n C_k\times \prod_{k=n+1}^\infty \Sigma\right|C_k\in \b(\Sigma),n\in \mathbb N_0\right\}.
  \end{equation}
  Next, for any cylinder set $C\in \c$,
  it holds that
  \begin{align*}
    \L^{-1}_\omega(C) &= \L_\omega^{-1}\left(\prod_{k=0}^n C_k\times \prod_{k=n+1}^\infty \Sigma\right)
    \\
    &= \prod_{k=0}^n \left(\L^{-1}(C_k)\times U\right)\times \prod_{k=n+1}^\infty (X\times U)\in \b(H).
  \end{align*}
  From \cite[Proposition 2.1]{f1999} it follows that $\L_\omega$ is Borel measurable.
\end{proof}

It follows from Proposition \ref{prop:L-meas} that given a \lcdt-MP $(\D,\Sigma,\L)$,
for each measurable language $\phi \in \b(\Sigma^\omega)$ there corresponds a unique event $\L_\omega^{-1}(\phi)\in \b(H)$
that is the set of all paths of $\D$ whose traces are elements of $\phi$.
In order to construct languages of interest in a handy and natural way,
we use LTL formulae.
The grammar of LTL over the alphabet $\Sigma$ is given by
\begin{equation}\label{eq:LTL.grammar}
  \Phi \quad ::= \quad \sigma\in \Sigma\;|\;\neg \Phi\;|\;\Phi_1\wedge\Phi_2\;|\;\X \Phi\;|\; \Phi_1\U^\infty \Phi_2.
\end{equation}
The definition \eqref{eq:LTL.grammar} shall be understood as follows:
if $\Phi_1$ and $\Phi_2$ are LTL formulae,
so are the expressions $\Phi_1\wedge \Phi_2$, $\Phi_1 \U^\infty \Phi_2$, $\neg \Phi_1$ etc.
Here $\wedge$ is the standard logical \emph{conjunction} and $\neg$ is the logical \emph{negation},
which allows us defining \emph{disjunction} as $\Phi_1\vee \Phi_2 := \neg(\neg \Phi_1\wedge \neg\Phi_2)$.
Furthermore, $\X$ and $\U^\infty$ are the \emph{neXt} and \emph{unbounded Until} temporal modalities whose meaning is clarified below.

The semantics of LTL formulae is defined using the notion of \emph{accepted language},
that is $\l(\Phi)\subseteq \Sigma^\omega$ is the collection of all infinite words over $\Sigma$ that are accepted by the formula $\Phi$.
Firstly, we define the shift on infinite words $\theta:\Sigma^\omega\to \Sigma^\omega$ by
\begin{equation*}
  \theta(w_0,w_1,w_2,\dots) = (w_1,w_2,\dots).
\end{equation*}
The semantics of LTL formulae is defined recursively as:
\begin{align*}
  w\in \l(\sigma) &\quad \iff \quad w_0 = \sigma
  \\
  w\in \l(\neg\Phi) &\quad \iff \quad w \notin \l(\Phi)
  \\
  w\in \l(\Phi_1\wedge\Phi_2) &\quad \iff \quad w \in \l(\Phi_1)\cap \l(\Phi_2)
  \\
  w \in \l(\X\Phi) &\quad \iff \quad \theta(w) \in \l(\Phi),
\end{align*}
and in addition the semantics of the $\U^\infty$ modality is as follows:
\begin{equation}\label{eq:semant.u-until}
\begin{split}
  w \in \l(\Phi_1\U^\infty \Phi_2) \quad \iff \quad &\theta^i(w) \in \l(\Phi_2) \text{ for some }i\in \N_0 \text{ and}
      \\
      &\theta^j(w)\in \l(\Phi_1)\text{ for all }0\leq j<i.
\end{split}
\end{equation}
It is useful to consider formulae describing bounded time horizon properties.
We first introduce powers of $\X$ inductively as $\X^0 \Phi := \Phi$ and $\X^n \Phi := \X (\X^{n-1}\Phi)$ for $n\geq 1$.
Using the latter notation,
it is now possible for any $n\in \N_0$ to define the formula
\begin{equation}\label{eq:synt.b-until}
  \Phi_1\U^n\Phi_2 := \bigvee_{i=0}^n \left(\bigwedge_{j=0}^{i-1}\X^j\Phi_1 \,\wedge\,\X^i\Phi_2\right),
\end{equation}
whose semantics is a finite-horizon equivalent of \eqref{eq:semant.u-until},
that is
\begin{equation*}
\begin{split}
  w \in \l(\Phi_1\U^n \Phi_2) \quad \iff \quad &\theta^i(w) \in \l(\Phi_2) \text{ for some }0\leq i\leq n \text{ and}
      \\
      &\theta^j(w)\in \l(\Phi_1)\text{ for all }0\leq j<i.
\end{split}
\end{equation*}
Note that $\U^\infty$ could be also expressed using \eqref{eq:synt.b-until},
but the countably infinite number of operations of conjunction needed are not explicitly allowed in the syntax of LTL.
We further denote $\true:=\bigvee_{\sigma\in \Sigma}\sigma$,
and introduce new temporal modalities:
\emph{eventually}, $\lozenge^n\Phi:=\true\U^n \Phi$,
and \emph{always}, $\square^n \Phi := \neg \lozenge^n\neg \Phi$,
for all $n\in \bar\N_0$.
We further simplify the notation as $\U:= \U^\infty$,
$\lozenge:=\lozenge^\infty$ and $\square:=\square^\infty$.
Note that an accepted language of any LTL formula is $\omega$-regular\footnote{
  On the other hand, there exist $\omega$-regular languages that are not accepted languages of any LTL formula~\cite{w1981}.
}
\cite{w1981},
and hence it is measurable,
so that LTL is a valid way to describe events.

Let us provide some examples of how LTL formulae can be used to describe events of interest.
We start with some basic formulae:
let us consider a \cdt-MP $\D = (X,U,\K,\T)$ and let $A,B\in \b(X)$ by two disjoint sets.
We label them as $A$ and $B$ respectively,
that is we introduce a labeling map $\L:X\to \Sigma$ where $\Sigma = \{A,B,\bot\}$ and
\begin{equation*}
  \L(x) =
  \begin{cases}
    A,&\text{ if }x\in A,
    \\
    B,&\text{ if }x\in B,
    \\
    \bot,&\text{ otherwise}.
  \end{cases}
\end{equation*}
Then the event $\{\xx_k \in A,k\geq 0\}$ can be expressed as $\square A$,
$\{\exists k\leq n: \xx_k\in B\}$ as $\lozenge^n B$,
$\{\xx_k\in B \text{ infinitely often }\}$ as $\square\lozenge B$,
$\{\exists k: \xx_j\in A, j\geq k\}$ as $\lozenge\square A$,
and finally the event $\{\exists k\leq n: \xx_k\in B \text{ and }\xx_j \in A,j<k\}$ can be expressed as $A\U^n B$.

As an additional example,
recall the power network model from Section \ref{ssec:power-net} and let $S$ be the safe set,
and $G_1$, $G_2$ be the preliminary target sets for each subnetwork,
and $G$ be the final target set.
Define the alphabet $\Sigma = \{S,G_1,G_2,G,\bot\}$,
where $\bot$ corresponds to the unsafe (failure) set,
and let $\L$ be the obvious labeling map,
e.g. $\L(x) = S$ if and only if $x\in S$.
The first task of being within the safe energy levels can be characterized by the formula $\square S$,
whereas
\begin{equation*}
  S \,\wedge\, \left(S \,\U \,\left(G_1 \,\wedge \,\left(G_1 \,\U\, G\right)\right)\right)\quad \vee \quad S \,\wedge\, \left(S \,\U \,\left(G_2 \,\wedge \,\left(G_2 \,\U\, G\right)\right)\right)
\end{equation*}
is the desired formula for the second task.
Indeed, in the first case only the word $SSSSS\dots$ is accepted,
which is produced exactly by those paths $h$ that stay in $S$ forever.
Similarly, the second formula only accepts those words that eventually have the letter $G$ following $G_1$ or $G_2$,
and never contain letter $\bot$,
so that the path representing energy levels never visits unsafe states and reaches the high energy level over the first subnetwork and then over the second,
while still keeping the first level high,
or vice-versa.

For a given a \lcdt-MP $(\D,\Sigma,\L)$ we shall consider the expected utility criterion $\M^\pi\left(x;\L_\omega^{-1}(\l(\Phi))\right)$ abbreviated by $\M^\pi(x;\Phi)$.
The main problem can be now formulated as follows:
\begin{problem}\label{prob:main}
  Given a \lcdt-MP $(\D,\Sigma,\L)$, an LTL formula $\Phi$ and a precision $\ve>0$ characterize $\M^*(x;\Phi)$ and compute its value with a given precision level $\ve$.
\end{problem}
Note that if one is able to solve Problem \ref{prob:main},
then one can also compute $\M_*(x;\Phi)$ for any LTL formula $\Phi$ thanks to the duality $\M_*(x;\Phi) = 1 - \M^*(x;\neg\Phi)$.

\subsection{Automata model-checking}
\label{ssec:aut}

Above we have formulated the main problem we are focusing on in this paper,
that requires computing
extremal probabilities of events expressed as LTL formulae over infinite paths of a \cdt-MP.
Although LTL provides a succinct way to express
events,
for algorithmic purposes an equivalent \emph{automata}-based perspective turns out to be more effective.
Automata are transition systems with inputs over a finite alphabet and simple acceptance conditions \cite[Chapter 4]{bk2008}.
An input word is accepted by the automaton if a corresponding \emph{run} of the automaton satisfies the acceptance condition.
Before we introduce these concepts formally,
let us mention that we follow the literature and only consider deterministic automata (those for which the current input and state uniquely determine the next state),
as they can be easily composed with \lcdt-MP.

\begin{definition}
  Given an alphabet $\Sigma$,
  a deterministic \emph{transition system} over $\Sigma$ is a tuple $\TS = (Q,q^s,\Sigma,\t)$ where $Q$ is a finite set,
  $q^s\in Q$,
  and $\t:Q\times \Sigma\to Q$ is some map.
  In our work all the transition systems are assumed to be deterministic\footnote{
    A non-deterministic transition system is one where $\t:Q\times \Sigma \to 2^Q$,
    that is given a current state of the system $q\in Q$ and an input letter $\sigma\in \Sigma$,
    the successor state $q' \in \t(q,\sigma)$ is not uniquely defined.
    As such, non-determinism here can be understood as set-valued dynamics,
    rather than as stochastic dynamics.
  }.
\end{definition}

Given a transition system $\TS = (Q,q^s,\Sigma,\t)$ we say that $Q$ is its state space,
$q^s$ is the initial condition,
$\Sigma$ is the input alphabet and $\t$ is the transition map.
Any word $w\in \Sigma^\omega$ induces a run $r\in Q^\omega$ of $\TS$ which is defined as follows:
$r_0 = q^s$ and $r_{k+1} = \t(r_k,w_k)$ for any $k\in \N_0$.
We can then introduce a map $\t_\omega:\Sigma^\omega \to Q^\omega$ that assigns to any input word the corresponding run.
An $\omega$-automaton is defined as follows:

\begin{definition}
  A deterministic $\omega$-automaton is a pair $\A = (\TS,\acc)$ consisting of a transition system $\TS = (Q,q^s,\Sigma,\t)$ together with an acceptance condition $\acc\in  \b(Q^\omega)$.
  We only consider automata on infinite words,
  so from now on we omit $\omega$ in ``$\omega$-automaton''.
\end{definition}

An acceptance condition of an automaton indicates which runs are accepted by the automaton ($r\in \acc$) and which are not ($r\notin \acc$).
Similarly, we can say that a word is accepted by a deterministic automaton if the corresponding run is accepted.
The literature has considered several versions of acceptance conditions for $\omega$-automata.
In the context of this work the following three are the most important:
\begin{itemize}
  \item[(DRA)] a \emph{deterministic Rabin automaton} is a tuple $\A = (Q,q^s,\Sigma,\t,(F'_i,F''_i)_{i\in I})$ where $(Q,q^s,\Sigma,\t)$ is a transition system,
   $I$ is some finite index set,
   and $F'_i,F''_i\subseteq Q$ for any $i\in I$.
   A DRA accepts a run $r\in Q^\omega$ if there exists $i\in I$ such that $r$ visits $F'_i$ an infinite number of times and $F''_i$ a finite number of times.
  \item[(DBA)] a \emph{deterministic B\"uchi automaton} is a special case of a DRA with $I$ being a singleton and $F'' = \emptyset$,
  that is:
  a DBA is a tuple $\A = (Q,q^s,\Sigma,\t,F)$ where $(Q,q^s,\Sigma,\t)$ is a transition system and $F\subseteq Q$ is a set of final states.
  A DBA accepts a run $r\in Q^\omega$ if $r$ visits $F$ an infinite number of times.
  \item[(DFA)] a \emph{deterministic finite automaton} is a special case of a DBA\footnote{
    While it is canonical to introduce a DFA on finite words \cite[Definition 4.9]{bk2008},
    we introduce it here on infinite words for the sake of consistency:
    in that way we do not have to consider both spaces of finite ($\Sigma^*$) and infinite ($\Sigma^\omega$) words over the alphabet $\Sigma$,
    and can just focus on the latter.
    It shall be clear that our definition is also consistent with the canonical one in \cite[Definition 4.9]{bk2008}:
    an infinite word $w\in \Sigma^\omega$ is accepted by a DFA if and only if there exists a finite prefix $w'\in \Sigma^*$ that is accepted by a DFA.
  }
  with all final states having self-loops ($\t(q,\sigma) = q$ for any $q\in F$, $\sigma\in \Sigma$),
  that is:
  a DFA is a tuple $\A = (Q,q^s,\Sigma,\t,F)$ where $(Q,q^s,\Sigma,\t)$ is a transition system and $F\subseteq Q$ is a set of final states.
  A DFA accepts a run $r\in Q^\omega$ if it visits $F$ at least once\footnote{
    An important version of the DFA has an $n$-horizon acceptance condition \cite[Section 2.4]{tmka2013},
    which requires the run to visit $F$ in at most $n$ steps.
    This is useful when one needs to to express formulae in bounded LTL (BLTL) --
    a fragment of LTL (for details see Section \ref{ssec:ltl.frag}).
  }.
\end{itemize}

For an automaton $\A = (\TS,\acc)$ we define the accepted language of $\A$ as the set of all infinite words that are accepted by $\A$;
we further denote this language by $\l(\A)$,
that is $\l(\A) := \t_\omega^{-1}(\acc)$.
Similar to Proposition \ref{prop:L-meas},
we can show that $\t_\omega\in \b(\Sigma^\omega)/\b(Q^\omega)$,
so that $\l(\A)$ is a measurable language as $\acc \in \b(Q^\omega)$.
Thus, for any \lcdt-MP $(\D,\Sigma,\L)$ with $\D = (X,U,\K,\T)$ the utility function $\M^\pi(\alpha;\L^{-1}_\omega(\l(\A)))$ is well-defined.
We further simplify the notation and write $\M^\pi(\alpha;\A)$.

Accepted languages of DRA are exactly $\omega$-regular languages \cite[Theorem 10.55]{bk2008},
so in particular for any LTL formula $\Phi$ there exists a DRA $\A^\Phi$ such that $\l(\Phi) = \l(\A^\Phi)$.
Furthermore, DBA (DFA) are strictly less expressive than DRA (DBA) -- for details see \cite[Chapter 4]{bk2008}.
We consider all three kinds of automata,
rather than focusing on the most expressive DRA,
due to the following reason.
We will show that for any automaton $\A$ the optimal utility $\M^*(x;\A)$ can be computed via a new optimal function $\hat \M^*((x,q^s);H\parallel \acc)$
over a newly defined \cdt-MP $\hat \D$,
which is a composition of $\D$ and $\A$.
Unfortunately, characterizing $\hat \M^*((x,q^s),H\parallel \acc)$ for DRA is rather difficult and we only provide partial results for the DBA case (Section \ref{sec:pers}),
whereas the acceptance condition of the DFA allows for a much more complete characterization (Section \ref{sec:reach}).

Before we proceed,
let us provide examples of automata for the tasks discussed over the power network model in Section \ref{ssec:power-net}.
The DBA for the first task is given in Figure \ref{fig:task1.dba}:
here if we do not label the transition (as the loop at $q^1$) it means that the transition happens for any label.
The final state is $q^0$ as indicated by a double circle.
As we have mentioned above,
the analysis of the DBA acceptance condition is more complicated than that of the DFA one,
hence even if the original LTL formula does not allow for the DFA expression,
it is worth checking whether its negation does allow for one.
For example, the DFA for the negation of the first task is given in Figure \ref{fig:task1.dfa}.

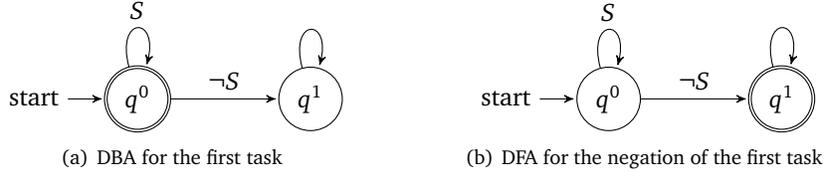
\begin{figure}[h]
\subfigure[DBA for the first task]{
   \begin{tikzpicture}[>=stealth',shorten >=1pt,auto,node distance=2.3cm]

    \node[initial,accepting,state]         (q0)                     {$q^0$};
    \node[state]                           (q1) [right of=q0]       {$q^1$};

    \path[->]
        (q0) edge [loop above]      node        {$S$}            (q0)
             edge                   node        {$\neg S$}       (q1)
        (q1) edge [loop above]      node        {}               (q1);
    \end{tikzpicture}
    \label{fig:task1.dba}
}
$\qquad\qquad$
\subfigure[DFA for the negation of the first task]{
    \begin{tikzpicture}[>=stealth',shorten >=1pt,auto,node distance=2.3cm]

    \node[initial,state]         (q0)                     {$q^0$};
    \node[accepting,state]       (q1) [right of=q0]       {$q^1$};

    \path[->]
        (q0) edge [loop above]      node        {$S$}            (q0)
             edge                   node        {$\neg S$}       (q1)
        (q1) edge [loop above]      node        {}               (q1);
    \end{tikzpicture}
    \label{fig:task1.dfa}
}
\caption{Automata representation of the first task of the case study}
\label{fig:task1}
\end{figure}

The second task has a direct DFA expression,
which is given in Figure \ref{fig:task2.dfa}.
For an overview of methods to construct an automaton given an LTL formula,
see \cite{vw1994}.

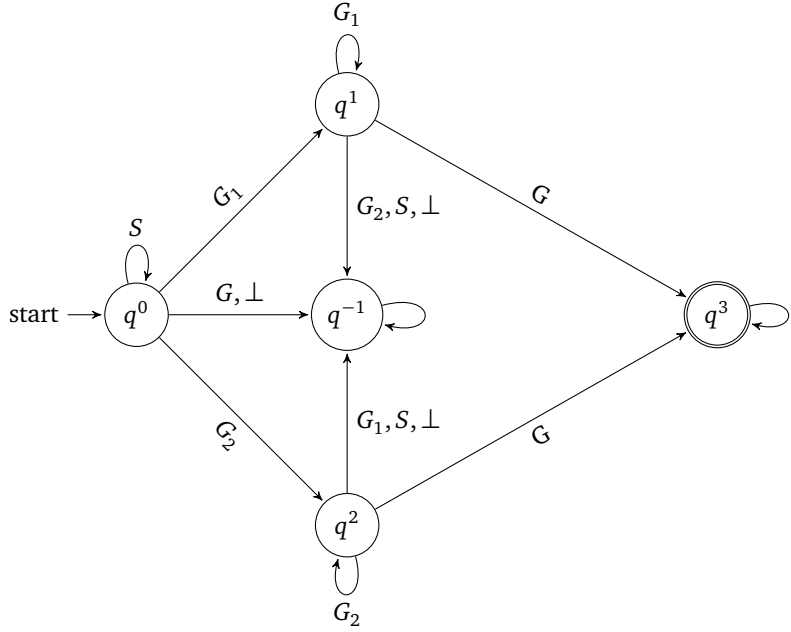
\begin{figure}[h]
\begin{tikzpicture}[>=stealth',shorten >=1pt,auto,node distance=2.8cm]
  \node[initial,state]         (q0)                     {$q^0$};
  \node[state]                 (q-1)[right of=q0]      {$q^{-1}$};
  \node[state]                 (q1) [above of=q-1]      {$q^1$};
  \node[state]                 (q2) [below of=q-1]      {$q^2$};
  \node[accepting,state]       (q3) [right=4cm of q-1]      {$q^3$};

  \path[->]
        (q0) edge [loop above]      node        {$S$}            (q0)
             edge [sloped,above]    node        {$G_1$}          (q1)
             edge [sloped,below]    node        {$G_2$}          (q2)
             edge                   node        {$G,\bot$}       (q-1)
        (q1) edge [loop above]      node        {$G_1$}          (q1)
             edge                   node        {$G_2,S,\bot$}       (q-1)
             edge [sloped,above]    node        {G}              (q3)
        (q2) edge [loop below]      node        {$G_2$}          (q2)
             edge [right]                  node        {$G_1,S,\bot$}   (q-1)
             edge [sloped,below]    node        {G}              (q3)
        (q-1)edge [loop right]       node        {}               (q-1)
        (q3) edge [loop right]       node        {}               (q3);
\end{tikzpicture}
\caption{The DFA for the second task of the case study}
\label{fig:task2.dfa}
\end{figure}

To state the main result of this section,
we need to introduce the composition between an \lcdt-MP and a transition system defined over the same alphabet.

\begin{definition}
  Given an \lcdt-MP $(\D,\Sigma,\L)$ with $\D = (X,U,\K,\T)$ and a transition system $\TS = (Q,q^s,\Sigma,\t)$,
  their \emph{composition} is a \cdt-MP $\hat \D = \D\parallel \TS =  (\hat X, U,\hat \K,\hat \T)$, where $\hat X := X\times Q$,
  $\hat \K_{(x,q)} := \K_x$ for any $x\in X$ and $q\in Q$ and
  \begin{equation*}
    \hat\T(A\times B|x,q,u) := 1_B\left(\t(q,\L(x))\right)\cdot\T(A|x,u).
  \end{equation*}
\end{definition}

Let us further discuss this notion of composition.
Consider an \lcdt-MP $(\D,\Sigma,\L)$ with $\D = (X,U,\K,\T)$ and a transition system $\TS = (Q,q^s,\Sigma,\t)$,
and let $\hat \D:= \D\parallel \TS$.
A more intuitive expression for the kernel $\hat \T$ can be given in the following way:
given a current joint state $(x_k,q_k)$ and a current control action $u_k$,
the new state is
\begin{equation*}
  \begin{cases}
    x_{k+1} &\sim \T(x_k,u_k),
    \\
    q_{k+1} &= \t(q_k,\L(x_k)).
  \end{cases}
\end{equation*}
The dynamics of the composed model should be understood as follows:
the $x$-coordinate of the new state evolves according to the law $\T$ of the original \cdt-MP $\D$,
and its label $\L(x)$ is used as an input to the transition system,
which produces the $q$-coordinate.
Let $\hat H := (\hat X\times U)^\omega$ denote the history space of $\hat \D$,
and further let $\hat \Pi$ be the class of all policies for $\hat \D$ that give rise to strategic measures $\hat\pr^{\hat\pi}_{\hat \alpha}$ for any $\hat\pi\in \hat\Pi$ and $\hat\alpha\in \p(\hat X)$.
We further let $(\hat\xx_n)_{n\in \N_0} = (\xx_n,\qq_n)_{n\in \N_0}$,
$(\uu_n)_{n\in \N_0}$ and $(\hat \hh_n)_{n\in \N_0}$ denote the state,
action and information processes on the sample space $\hat H$, respectively.

As anticipated above,
the main result of this section is as follows.
For any \lcdt-MP $(\D,\Sigma,\L)$ and for any automaton $\A = (\TS,\acc)$,
which may for example express an LTL formula $\Phi$,
it holds that $\M^*(x;\A) = \hat\M^*((x,q^s);H\parallel \acc)$,
where $\hat\M^*$ is the optimal utility functional over the composed \cdt-MP $\hat \D:=\D\parallel \TS$.
To obtain this result,
we first need to establish a policy equivalence between optimal utilities over $\D$ and $\hat \D$.
More precisely, we connect classes $\Pi$ and $\hat \Pi$ as follows.
The former class can be treated as a subclass of the latter, where policies do not depend on $q$-coordinates of $\hat h_n\in \hat H_n$,
so we let $\I:\Pi \to \hat \Pi$ denote the corresponding embedding map.
Conversely, we introduce a projection map $\P:\hat\Pi\to \Pi$ by the formula
\begin{equation*}
  (\P\hat\pi)_n(x_0,u_0,x_1,u_1,\dots,x_n) := \hat\pi_n(x_0,q_0,u_0,x_1,q_1,u_1,\dots,x_n,q_n),
\end{equation*}
where $q_0 = q^s$ and $q_{k+1} = \t(q_k,\L(x_k))$, for all $0\leq k<n$.

\begin{lemma}\label{lem:aut.equiv}
  For any $\alpha\in \p(X)$,
  and any policies $\pi\in \Pi$ and $\hat \pi\in \hat \Pi$,
  it holds that
  \begin{equation*}
    \pr^\pi_\alpha(\L_\omega(h) \in \l(\A)) = \hat\pr^{\I\pi}_{\alpha\otimes \delta_{q^s}}(H\parallel \acc),\qquad \hat\pr^{\hat\pi}_{\alpha\otimes \delta_{q^s}}(H\parallel \acc) = \pr^{\P\pi}_\alpha(\L_\omega(h) \in \l(\A)).
  \end{equation*}
\end{lemma}

\begin{proof}
  Let us introduce a map $\beta:H\to \hat H$ as $\beta := \id_H\parallel (\t_\omega\circ \L_\omega)$,
  so that given a path $h\in H$ this map returns a path $\hat h =\beta(h)\in \hat H$ which has the same $x$- and $u$-coordinates,
  and the $q$-coordinates of which are obtained using the automaton transition map.
  As a result, for any $\alpha \in \p(X)$ and any $\pi\in \Pi$ it holds that
  \begin{equation*}
    \pr^\pi_\alpha(\L_\omega(h) \in \l(\A)) = \pr^\pi_\alpha((\t_\omega\circ \l_\omega)(h)\in D) = (\beta_*\pr^\pi_\alpha)(H\parallel \acc).
  \end{equation*}
  Applying definitions of maps $\I$ and $\P$ immediately yields the desired result.
\end{proof}

Before we apply Lemma \ref{lem:aut.equiv} to characterize the optimal utility of $\D$ via that of $\hat\D$,
let us recall that DFA and DBA are not closed under negations\footnote{
  Recall that here DFAs are interpreted over infinite words.
  Usually, DFAs are interpreted over finite words,
  and in such case they are closed under negations.
},
that is if we are able to express an LTL formula $\Phi$ as a DFA or DBA $\A$,
there may not exist such an expression for $\neg \Phi$.
Due to this reason, in the next theorem we explicitly formulate both the maximization and the minimization problems,
which allow us applying the results both in cases of events expressed as DFA and DBA,
and in cases when the complement of the event can be expressed in these automata classes.
As above, let $\A = (\TS,\acc)$ be some automaton over an alphabet $\Sigma$,
$(\D,\Sigma,\L)$ be any \lcdt-MP and let $\hat\D = \D\parallel \TS$.

\begin{theorem}\label{thm:aut.equiv}
  The following equalities hold true:
  \begin{equation*}
    \M^*(\alpha;\A) = \hat \M^*(\alpha\otimes \delta_{q^s};H\parallel \acc),\qquad \M_*(\alpha;\A) = \hat \M_*(\alpha\otimes\delta_{q^s};H\parallel \acc).
  \end{equation*}
\end{theorem}

\begin{proof}
  The proof follows directly from Lemma \ref{lem:aut.equiv} and Lemma \ref{lem:equiv} (cf. Section \ref{sec:app.aux.res}).
\end{proof}

Let us discuss the importance of the result in Theorem \ref{thm:aut.equiv}.
Suppose we are given an \lcdt-MP $(\D,\Sigma,\L)$ where $\Sigma$ and $\L$ are used to distinguish the sets of interest,
and a property expressed as a DFA or a DBA $\A$ over the alphabet $\Sigma$.
Such an expression may
encode the LTL formula $\Phi$ for the desired property.
Instead of having to compute the maximal probability $\M^*(\alpha;\A)$ directly,
we can focus on an equivalent problem over $\hat \D := \D\parallel \TS$,
and focus on the property $\lozenge F$ in the case when $\A$ is a DFA,
or on the property $\square\lozenge F$ when $\A$ is a DBA.
We refer to the former property as \emph{reachability} and to the latter as \emph{repeated reachability}.
The rest of the article is focused on the solution of both problems,
so the coming results are applicable to classes of properties expressed as DFA and DBA thanks to Theorem \ref{thm:aut.equiv}.

\subsection{Comments on models and problem formulation}
\label{ssec:aut.comments}

The exposition of the model in this work is rather standard and is similar to that in \cite[Section 2.2]{hll1996}.
However the present model is more general:
for example we allow for a feasibility set $\K$ that is analytic, and for universally measurable policies.
It can be shown that whenever the initial distribution $\alpha\in \p(X)$ is fixed,
for a large class of performance criteria including all expected bounded utility cases it is sufficient to consider only analytically measurable deterministic policies depending exclusively on state coordinates of the history \cite{b1976a}.
Moreover, one can sufficiently deal with Borel measurable policies, provided they do exist.
However, if one is interested in finding a policy that is optimal or $\ve$-optimal for any initial distribution,
it is more convenient to deal with the class of universally measurable policies:
the latter is rich enough to assure the existence of policies for many interesting problems --
see e.g. the discussion in \cite[Section 1.2]{bs1978}.
This class also possesses some nice closure properties in contrast to the class of analytically measurable policies:
e.g. the composition of two universally measurable functions is again universally measurable,
but the composition of analytically measurable functions may not be analytically measurable.
Such closure properties are important to ensure the appropriate measurability of the performance criterion with respect to the initial state.
More details on this topic can be found in \cite{sb1979}.

\smallskip

It is worth mentioning that there is an alternative approach to sequential decision making in a stochastic environment,
which is known as \emph{gambling} \cite{ds1965}.
The difference with the \cdt-MP is mainly conceptual:
if the current state is $x$, instead of first making a choice of a control action $u$ and drawing a new state according to the distribution $\T(x,u)$,
in gambling the agent is allowed to choose the distribution of the new state directly,
from the set of available \emph{gambles} $\Gamma_x$\footnote{
  Note that in \cdt-MP the choice of the distribution of the successor state is ``labelled'' by actions,
  whereas in gambling models such choice is unlabelled.
  One may think of this being similar to internal and external non-determinism in probabilistic automata \cite{sl1995},
  however there is no semantic difference between \cdt-MP and gambling models,
  and in both cases non-determinism can be considered both as an internal one or as an external one.
}.
The set $\Gamma \subseteq X\times \p(X)$ is called the \emph{gambling house}.
On the methodological level,
the difference between the \cdt-MP and gambling is that the latter extensively uses stopping time-like methods to derive most of the results,
whereas the former is more focused on techniques based on DP.
Finally, the difference between \cdt-MP and gambling models is also technical.
First of all,
initially the research on gambling theory has been done in the framework of finitely-additive probability measures~\cite{ds1965}.
Later, gambling models have also been considered in the $\sigma$-additive framework,
which made it possible to compare them with \cdt-MP:
for example, \cite{b1976a} showed the equivalence between some classes of \cdt-MP and gambling models --
this result also holds for the \cdt-MP model we consider in the present paper.
Further gambling models have been used more recently, e.g. in \cite{mps1991} and \cite{ms1996a}.

Research on gambling has broadly looked into the optimization of probabilities of given events.
For example, \cite{mps1991} has obtained results for safety properties (that are clearly also applicable to the reachability),
and \cite{mps1991,ms1996a} has characterized the repeated reachability property.
Due to this reason,
although we do not use the gambling framework explicitly,
sometimes we recall the results obtained there.
For example, using the \cdt-MP framework for reachability properties seems more beneficial,
however we mostly use the results of gambling for the repeated reachability.
Another important point is that \cite[Chapter 6]{ms1996} proposes an idea to optimize the probabilities of events,
which is alternative to the one we convey in Section \ref{ssec:prob}.
More precisely, it is shown that in the case of a countable state space the functional $\M^*$ possesses some useful properties of the capacities \cite{d1981}.
In particular, \cite[Theorem (1.2), Chapter 6]{ms1996} claims that for any state $x\in X$ and any event $A\in \b(H)$ it holds that
\begin{equation}\label{eq:open.approx}
  \M^*(x;A) = \inf\left\{\M^*(x;B):B\text{ is open and }B\supseteq A\right\}.
\end{equation}
Furthermore, $\M^*$ for open events can be obtained by means of stopping times --
see \cite[Chapter 6]{ms1996} for more details.
This result may be extendable to the more general case we deal with,
where $X$ is uncountable and one is interested only in events that can be described using some finite alphabet $\Sigma$.
Unfortunately \eqref{eq:open.approx} does not provide a direct and explicit way to compute quantities of interest,
or to derive optimal policies,
so we do not pursue such direction here,
preferring instead more explicit methods based on LTL formulae and automata theory.

\smallskip

The problem of optimizing the probability of a given event (or a property) is a problem that often appears in computer science,
see e.g. a wide range of examples described in \cite[Section 10.6]{bk2008}.
Using LTL and automata theory for finite state-space \cdt-MP has a long history, part of which can be consulted in \cite[Section 10.8]{bk2008}.
However, extensions to the general state-space case have only appeared recently:
\cite{AKM2011} has provided an extension to the uncontrolled case (where trivially $U = \{u\}$ is a singleton),
whereas \cite{KamgarpourSL2013} and \cite{tmka2013} worked out the controlled case\footnote{
  The difference between the approaches in these two works is that \cite{KamgarpourSL2013} has allowed for Markov policies only,
  but clearly the policies over the composed system may depend on the state of the transition system:
  the map $\P$ can map Markov policies to history-dependent ones.
  To cope with this issue,
  extended Markov policies have been proposed in \cite{KamgarpourSL2013},
  namely policies that can depend also on an additional historical variable -- the state of the transition system,
  which is a deterministic function of the \cdt-MP state history.
}.
In particular, the latter contribution is a basis for Section \ref{sec:models} and \ref{sec:reach} of the current manuscript.

\section{Reachability}
\label{sec:reach}

\subsection{
Reachability problem: characterization}

As Theorem \ref{thm:aut.equiv} showed,
optimizing probabilities over a \cdt-MP for a large class of events of interest can be reduced to either a reachability problem,
or to a repeated reachability one.
This section is focused on the reachability problem.
For this purpose it is more convenient to consider a slightly more general setup,
called the \emph{constrained reachability} problem \cite[Section 10.1.1]{bk2008}.\footnote{
  The constrained reachability problem is also known as the \emph{reach-avoid} problem \cite{SL10}.
}
To satisfy the constrained reachability property,
the path of a \cdt-MP does not only have to reach a given goal set,
but also to stay within some safe set before hitting the goal one.
In terms of the LTL grammar,
we are going to deal with the property $S \U^n G$,
where $S$ is a safe set and $G$ is a goal set.
The \emph{(unconstrained)} reachability problem corresponds to the special case $\lozenge^n G =  \true \U^n G$.

More precisely, consider a \cdt-MP $\D = (X,U,\K,\T)$ and let $G\in \b(X)$ be the set of goal states,
and $S\in \b(X)$ be the set of safe states.
Define $D:=S^c\setminus G$ to be the corresponding set of unsafe (or dangerous) states.
For any initial distribution $\alpha$,
any policy $\pi\in \Pi$,
and any time horizon $n\in \bar\N_0$ we are thus interested in the value of $\M^\pi(\alpha,S\U^n G)$.
It is more convenient to focus on the initial distribution supported on single points and thus consider a function $\M^\pi(\cdot,S\U^n G):X\to [0,1]$,
extending the results to arbitrary initial distributions at a later stage.
Clearly, $\M^\pi(\cdot;S\U^n G) \in \bu(X)$ for any $\pi\in \Pi$ and $n\in \bar\N_0$.
Moreover, the sequence $(\M^\pi(x;S\U^n G))_{n\in \N_0}$ is non-decreasing in $n$ and furthermore for any fixed $x\in X$
\begin{equation}\label{eq:reach-lim}
  \M^\pi(x;S\U G) = \lim_{n\to\infty} \M^\pi(x;S\U^nG).
\end{equation}
Obviously, the unconstrained reachability defined in Section \ref{ssec:aut} is a special instance of the constrained reachability in case the safe set is the whole state space,
i.e. $S = X$.\footnote{
  As a side note, the constrained reachability can be also obtained from the unconstrained one by changing the dynamics of the \cdt-MP on the set $D\setminus G$~\cite[Section 3.1]{tmka2013}.
}
Note that the solution of the problem is partially known:
\begin{equation}\label{eq:reach.S.G.}
  \M^\pi(x;S\U^nG) =
  \begin{cases}
    1,& \text{ if }x\in G,
    \\
    0,& \text{ if }x\in D
  \end{cases}
\end{equation}
and, as a result,
the constrained reachability problem needs to be solved only for states in $S\setminus G$.
On the other hand, without loss of generality we can assume that sets $S$ and $G$ are disjoint:
this follows from the fact that $\M^\pi(x;S\U^nG) = \M^\pi(x;(S\setminus G)\U^n G)$.
Below this assumption is often made for the sake of notation;
this also allows us to highlight that the dynamics of \cdt-MP over the set $S$ are of the highest importance for the solution of the problem,
in contrast to the dynamics of the states in the set $G$.
As we have mentioned above,
we consider both the maximization and the minimization problems for the constrained reachability,
namely both $\M^*(x;S\U^n G)$ and $\M_*(x;S\U^n G)$.

It is known that the DP principles allow decomposing the general optimization problem into smaller and simpler subproblems \cite{b1957}.
In the literature there have been several results developing DP characterizations of the constrained reachability problem.
One of the main differences in these studies has been the choice of the structural representation of the value function $\M(\cdot;S\U^n G)$.
For example, the work in \cite{APLS08b} has considered the max cost representation for the unconstrained reachability, as
\begin{equation}\label{eq:reach.cost.max}
  \M^\pi(x;X\U^n G) = \pr^\pi_x\left[\max_{k\leq n}1_G(\xx_k)\right],
\end{equation}
and using the dual safety problem,
an alternative multiplicative cost representation
\begin{equation}\label{eq:reach.cost.mult}
  \M^\pi(x;S\U^n G) = 1 -\pr^\pi_x\left[\prod_{k=0}^n 1_{G^c}(\xx_k)\right].
\end{equation}
These results have been extended in \cite{SL10},
which has dealt with the general constrained reachability problem in the form of a sum-multiplicative cost
\begin{equation}\label{eq:reach.cost.add.mult}
  \M^\pi(x;S\U^nG) = \pr^\pi_x\left[\sum_{k=0}^n \left(\prod_{j=0}^{k-1} 1_{S\setminus G}(\xx_j)\right)1_G(\xx_k)\right].
\end{equation}
Later, \cite{CCL11} suggested a cost formulation using the notion of a first hitting time as
\begin{equation}\label{eq:reach.cost.stop}
  \M^\pi(x;S\U^nG) = \pr^\pi_x\left[\sum_{k=0}^{n \wedge \tau_G\wedge\tau_{D}} 1_G(\xx_k)\right],
\end{equation}
where $\tau_A := \inf\{k\geq 0:\xx_k \in A\}$ is the first hitting time of the set $A\in \b(X)$.
As we have mentioned in Section \ref{ssec:prob},
the TC performance criterion allows for a rich theory of DP in a general setting.
The aforementioned studies in \cite{APLS08b},
\cite{SL10} and \cite{CCL11} have recovered only a subset of these results for the reachability problem,
sometimes requiring restrictive assumptions on the structure of the model.
Here we show that the reachability problem has an equivalent TC formulation,
which allows us proving all results available for this general performance criterion.

In general it may not be possible to characterize the constrained reachability problem as a TC criterion over the original \cdt-MP $\D$.
The key idea is to consider an auxiliary \cdt-MP $\hat \D$,
constructed from the original one by adding a new Boolean variable that represents whether the path of $\D$ has left the safe set $S$ or not.
To our knowledge,
the first time such construction has been explicitly used in \cite{tmka2013}.\footnote{
  In \cite{dat2013} a similar construction was used to formulate reachability problem as a final cost problem.
}
For the sake of consistency,
here we introduce a new \cdt-MP
using the notion of the composition between the transition system and the original \cdt-MP.

\begin{figure}[ht]
    \begin{tikzpicture}[>=stealth',shorten >=1pt,auto,node distance=2.8cm]

    \node[state]         (qs)                     {$q^s$};
    \node[state]         (qf) [right of=q0]       {$q^f$};

    \path[->]
        (qs) edge [loop left]      node        {$S$}            (qs)
             edge                  node        {$D,G$}          (qf)
        (qf) edge [loop right]     node        {}               (qf);
    \end{tikzpicture}
    \caption{Transition system for the TC formulation of constrained reachability}
\label{fig:tc.ts}
\end{figure}
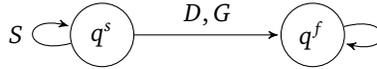

Let us consider a transition system $\TS = (Q,q^s,\Sigma,\t)$ as in Figure \ref{fig:tc.ts} with a state space $Q = \{q^s,q^f\}$,
an alphabet $\Sigma = (D,G,S)$,
and a transition function given by
\begin{equation*}
  \t(q^s,S) = q^s,\quad \t(q^s,\{D,G\}) = q^f,\quad  \t(q^f,\Sigma) = q^f.
\end{equation*}
We extend $\D$ to a \lcdt-MP $(\D,\Sigma,\L)$ trivially by letting $\L(x) = S$ for all $x\in S$,
and denote by $\hat \D := (\D,\Sigma,\L) \parallel  \TS = (\hat X,U,\hat\K,\hat T)$ the composed \cdt-MP.
We also let the corresponding canonical probability space and related state,
action and information processes to be defined as in Section \ref{ssec:aut}.
Let us explicitly write down the relation between operators $\hat\T$ and $\T$,
as they are further needed below:
\begin{equation}\label{eq:T^T}
  \begin{split}
    \hat\T^u \hat f(x,q) &= 1\{q = q^f\}\T^u \hat f(x,q^f)
    \\
    &+ 1\{q = q^s\}\left(1_S(x)\T^u \hat f(x,q^s) + 1_{S^c}(x)\T^u \hat f(x,q^f)\right),
  \end{split}
\end{equation}
which holds for any $\hat f\in \bu(\hat X)$,
$x\in X$,
$u\in U$, and $q\in Q$.
In particular, of special interest are functions $\hat f:\hat X \to \R$ that are zero off $q^s$,
namely $\hat f(\cdot,q^f) \equiv 0$:
they can be always represented in the form
\begin{equation}\label{eq:zero.off}
  \hat f(x,q) = 1\{q = q^f\}\cdot f(x)
\end{equation}
for some $f:X\to\R$.
For functions as in \eqref{eq:zero.off},
equation \eqref{eq:T^T} simplifies as
\begin{equation}\label{eq:T^T.zero}
  \hat\T^u \hat f(x,q) = 1\{q = q^s\}\T^u \hat f(x,q^s) = 1\{q = q^s\}\T^u f(x)
\end{equation}
so that $\hat \T^u$ preserves the property of being zero off $q^s$.

Let $c:\hat X\to \{0,1\}$ be a cost function $c(x,q) := 1\{q = q^s\}\cdot 1_{G}(x)$ that is zero off $q^s$,
and define the corresponding TC utility for any $n\in \bar\N_0$ as follows:
\begin{equation}\label{eq:reach.TC}
  \hat J_n:= \sum_{k=0}^n c(\xx_{k},\qq_{k}),
\end{equation}
where $\xx$ and $\qq$ are the components of the state process as in Section \ref{ssec:aut}.
The corresponding maximization and minimization problems are given by
\begin{equation}\label{eq:add.opt}
  \hat \M^*(x,q;\hat J_n) = \sup_{\hat \pi\in \hat \Pi}\hat \pr^\pi_{(x,q)}\left[\hat J_n\right], \qquad \hat \M_*(x,q;\hat J_n) = \inf_{\hat \pi\in \hat \Pi}\hat \pr^\pi_{(x,q)}\left[\hat J_n\right].
\end{equation}
In order to show the equivalence between the optimal constrained reachability problem over the \cdt-MP $\D$ and the formulation in \eqref{eq:add.opt} over the \cdt-MP $\hat \D$,
we apply the technique from Section \ref{ssec:aut}.
Let us denote by $\I:\Pi \to \hat \Pi$ the obvious embedding map,
and let the projection map $\hat\P:\hat \Pi\to \Pi$ be given by
\begin{equation}\label{eq:eta}
  (\hat\P\hat\pi)_n(x_0,u_0,\dots,x_{n-1},u_{n-1},x_n) := \hat\pi_n(x_0,q^s,u_0,\dots,x_{n-1},q^s,u_{n-1},x_n,q^s).
\end{equation}
Note that $\hat\P$ is different from the projection map $\P$ discussed in Section \ref{ssec:aut}:
in particular, later we use the fact that $\hat\P(\hat\Pi_M) \subseteq \Pi_M$,
whereas $\P$ does not necessarily preserve the Markovian property of a policy.
The following equivalence holds true:

\begin{lemma}\label{lem:add.cost}
  For any $n\in \bar\N_0$, $\pi\in \Pi$ and $\hat\pi\in \hat\Pi$, it holds that
  \begin{equation}\label{eq:thm.add.cost}
    \hat \M^{\hat\pi}(x,q^s;\hat J_n) = \M^{\hat\P\hat\pi}(x;S\U^nG), \qquad \M^\pi(x;S\U^nG) = \hat \M^{\I\pi}(x,q^s;\hat J_n).
  \end{equation}
\end{lemma}

\begin{proof}
  We prove this theorem by induction.
  First of all, both equalities in \eqref{eq:thm.add.cost} clearly hold true for $n=0$ as in this case all functions are simply $1_G(x)$.
  Furthermore, with focus on the first equality,
  we have that
  \begin{equation*}
    \hat \M^{\hat\pi}(x,q^s;\hat J_{n+1}) - \hat \M^{\hat\pi}(x,q^s;\hat J_n) = \hat\pr^{\hat\pi}_{(x,q^s)}\left[c\left(\xx_{n+1},\qq_{n+1}\right)\right].
  \end{equation*}
  As $c(\xx_{n+1},\qq_{n+1})$ is a Bernoulli random variable supported on $\{0,1\}$,
  we obtain that
  \begin{equation*}
  \begin{split}
    \hat\pr^{\hat\pi}_{(x,q^s)}\left[c\left(\xx_{n+1},\qq_{n+1}\right)\right]  &= \hat\pr^{\hat\pi}_{(x,q^s)}\left(c\left(\xx_{n+1},\qq_{n+1}\right) = 1\right)
    \\
    &= \hat\pr^{\hat\pi}_{(x,q^s)}\left(\left\{\xx_k\in S,k\leq n\right\},\left\{\xx_{n+1}\in G\right\},\left\{\qq_k = q^s,k\leq n+1\right\}\right).
  \end{split}
  \end{equation*}
  On the other hand, the increment in $n$ of the function $W^{\hat\P\hat\pi}_n$ is
  \begin{equation*}
    \M^{\hat\P\hat\pi}(x;S\U^{n+1}G) - \M^{\hat\P\hat\pi}(x;S\U^nG) = \pr^{\hat\P\hat\pi}_x\left(\left\{\xx_k\in S,k\leq n\right\},\left\{\xx_{n+1}\in G\right\}\right).
  \end{equation*}
  The fact that these probabilities are equal
  follows immediately from their integral expressions in \eqref{eq:IT} and from the definition of the projection map $\hat\P$.
  By induction we obtain the first part in \eqref{eq:thm.add.cost} for $n<\infty$,
  and the case $n=\infty$ follows by taking the limit.
  Finally, the proof of the second part of \eqref{eq:thm.add.cost} is obtained the same way,
  mutatis mutandis.
\end{proof}

Lemma \ref{lem:add.cost} leads to several important results that allow us to develop a DP framework for constrained reachability.
First of all, it clearly implies that both optimization problems are equivalent in the following sense:

\begin{theorem}\label{thm:add.cost}
  For all $n\in \bar\N_0$ and $x\in X$ we have $\hat \M^*(x,q^f;\hat J_n) = \hat \M_*(x,q^f;\hat J_n) = 0$ and
  \begin{equation}\label{eq:opt.add.cost}
    \M^*(x;S\U^n G) = \hat \M^*(x,q^s;\hat J_n),\quad \M_*(x;S\U^n G) = \hat \M_*(x,q^s;\hat J_n).
  \end{equation}
\end{theorem}

\begin{proof}
  To prove the first part,
  one has to notice that if $\qq_0 = q^f$,
  then $\qq_n = q^f$ for all $n\in \N_0$,
  hence $\hat \M^{\hat\pi}(x,q^f;\hat J_n) = 0$ for all $n\in \N_0$, $x\in X$, and $\hat\pi\in \hat \Pi$.
  Furthermore, \eqref{eq:opt.add.cost} is an immediate consequence of Lemma \ref{lem:add.cost} and Lemma \ref{lem:equiv} in the Appendix.
\end{proof}

As we have mentioned above,
Theorem \ref{thm:add.cost} allows us to extrapolate the rich theory developed for the TC criterion to the case of the constrained reachability problem.
However, most of the results for TC are developed for the minimization case \cite{bs1978,hll1996},
considering either positive or negative costs $c$.
As such,
we can directly derive the results for the minimization problem since $\M_*(x;S\U^n G) = \hat \M_*(x,q^s;\hat J_n)$,
however for the maximal constrained reachability we shall interpret
\begin{equation*}
  \M^*(x;S\U^n G) = -\hat \M_*(x,q^s;-\hat J_n),
\end{equation*}
thus characterizing both optimization problems as a minimization of some TC.
Note that for the minimization of the constrained reachability we use a positive cost $c$,
thus falling into the setting of the positive DP \cite{Blackwell1967} corresponding to \cite[Assumption (P), Chapter 9]{bs1978}.
On the other hand, for the maximization of the constrained reachability a negative cost $-c$ is used,
hence leading to the case of the negative DP \cite{Strauch1966} corresponding to \cite[Assumption (N), Chapter 9]{bs1978}.
This difference is not always important and only matters in the case $n = \infty$,
but we show below that it affects the convergence of bounded-horizon functions to the unbounded-horizon ones,
as well as the existence of optimal policies.

Let us proceed with the application of Lemma \ref{lem:add.cost} and Theorem \ref{thm:add.cost} to the characterization of the optimal constrained reachability problems.
The next results shows that it is sufficient to deal with Markov policies.

\begin{proposition}\label{prop:markov.pol}
  For any $n\in \bar\N_0$ and any policy $\pi\in \Pi$,
  there exists a Markov policy $\pi'\in \Pi_M$ such that $\M^\pi(\cdot;S\U^n G) = \M^{\pi'}(\cdot;S\U^n G)$,
  and as a consequence
  \begin{equation}\label{eq:Markov.suff}
    \M^*(x;S\U^n G) = \sup_{\pi\in \Pi_M}\M^\pi(x;S\U^n G),\quad \M_*(x;S\U^n G) = \inf_{\pi\in \Pi_M}\M^\pi(x;S\U^n G).
  \end{equation}
\end{proposition}

\begin{proof}
  Fix any state $x\in X$ and any policy $\pi \in \Pi$.
  It follows from Lemma \ref{lem:add.cost} that $\M^\pi(x;S\U^n G) = \hat \M^{\I\pi}(x,q^s;\hat J_n)$.
  On the other hand, \cite[Proposition 8.1]{bs1978} assures the existence of a Markov policy $\hat\pi'\in \hat\Pi_M$ satisfying $\hat \M^{\I\pi}(x,q^s;\hat J_n) = \hat \M^{\hat\pi'}(x,q^s;\hat J_n)$.
  From the definition of the projection map $\hat\P$ it follows that $\pi':=\hat\P\hat\pi' \in \Pi_M$ and as a result
  \begin{equation*}
    \M^{\pi'}(x;S\U^n G) = \hat \M^{\hat\pi'}(x,q^s;\hat J_n) = \hat \M^{\I\pi}(x,q^s;\hat J_n) = \M^\pi(x;S\U^n G),
  \end{equation*}
  as desired.
  In order to obtain \eqref{eq:Markov.suff} we only have to apply Lemma \ref{lem:equiv}.
\end{proof}

The results above, obtained for deterministic initial conditions, can be extended to the case of general initial distributions:
we show that a value function over an initial distribution $\alpha\in \p(X)$ can be obtained by integrating value functions over deterministic initial conditions.
Although this result is obvious whenever the policy is fixed,
it is not trivial to be shown for optimal value functions.
We show a proof for the case of the minimization problem on the unbounded time horizon,
however similar results can be obtained for the unbounded-time maximization case,
as well as for both bounded-horizon problems.

\begin{proposition}\label{prop:reach.init.distr}
  For any distribution $\alpha\in \p(X)$ it holds that
  \begin{equation}\label{eq:reach.init.distr}
    \M_*(\alpha;S\U G) = \int_X\M_*(x;S\U G)\;\alpha(\d x).
  \end{equation}
\end{proposition}

\begin{proof}
  From \cite[Propositions 9.2, 9.3, 9.5]{bs1978} it follows that
  \begin{equation*}
    \hat \M_*(\hat\alpha;\hat J_\infty) = \int_{\hat X} \hat \M_*(x,q;\hat J_\infty)\hat\alpha(\d x\times \d q)
  \end{equation*}
  for any distribution $\hat\alpha \in \p(\hat X)$.
  As a result, for any $\alpha\in \p(X)$ it holds that
  \begin{align*}
    \M_*(\alpha;S\U G) &= \inf_{\pi\in \Pi}\int_{\hat X} \hat\M^{\I \pi}(x,q;\hat J_\infty)(\alpha \otimes \delta_{q^s})(\d x\times \d q)
    \\
    &\geq \int_{\hat X} \hat\M_*(x,q;\hat J_\infty)(\alpha\otimes \delta_{q^s})(\d x\times \d q)
    \\
    &= \int_{\hat X} \hat\M_*(x,q^s;\hat J_\infty)\alpha(\d x) = \int_X\M_*(x;S\U G)\;\alpha(\d x).
  \end{align*}
  The converse inequality we get as follows:
  \begin{align*}
    \int_X\M_*(x;S\U G)\;\alpha(\d x) &= \int_{\hat X} \hat\M_*(x,q^s;\hat J_\infty)\alpha(\d x)
    \\
    &= \int_{\hat X} \hat\M_*(x,q;\hat J_\infty)(\alpha\otimes \delta_{q^s})(\d x\times \d q)
    \\
    &= \inf_{\hat\pi\in \hat\Pi}\int_{\hat X} \hat\M^{\hat\pi}(x,q;\hat J_\infty)(\alpha \otimes \delta_{q^s})(\d x\times \d q)
    \\
    &= \inf_{\hat\pi\in \hat\Pi}\int_X \M^{\hat\P\hat\pi}(x;S\U G)\alpha(\d x) \geq \M_*(\alpha;S\U G).
  \end{align*}
  Since both inequalities hold true,
  we obtain the desired result.
\end{proof}

Although in general one cannot switch the order of the minimization and of the integral,
Proposition \ref{prop:reach.init.distr} shows this can be done in the case of \eqref{eq:reach.init.distr}.
Thus, it is sufficient to deal with deterministic initial distributions:
the value function for the general one can be obtained by integrating with respect to the initial distribution of interest.

We are ready to formulate one of the most relevant outcomes of Theorem \ref{thm:add.cost}:
a DP procedure for the constrained reachability problem over a general class of policies.
For this purpose we introduce the following DP operators:
\begin{align*}
  \r^*f(x) &= 1_G(x) + 1_S(x)\cdot\T^* f(x),\qquad f\in \busa(X),
  \\
  \r_*f(x) &= 1_G(x) + 1_S(x)\cdot \T_* f(x),\qquad f\in \blsa(X).
\end{align*}
From the properties of operators $\T^*$ and $\T_*$,
it follows that $\r^*$ maps $\busa(X)$ into itself and $\r_*$ maps $\blsa(X)$ into itself.

\begin{theorem}\label{thm:reach.DP}
  It holds that $\M^*(\cdot;S\U^0 G) = \M_*(\cdot;S\U^0 G) = 1_G(\cdot)$,
  and for any $n\in \bar\N_0$
  \begin{equation*}
    \M^*(\cdot;S\U^{n+1} G) = \r^*\left[\M^*(\cdot;S\U^n G)\right],\qquad \M_*(\cdot;S\U^{n+1} G) = \r_*\left[\M_*(\cdot;S\U^n G)\right].
  \end{equation*}
  Moreover, $\M^*(\cdot;S\U G)$ and $\M_*(\cdot;S\U G)$ are the least non-negative fixpoints of the corresponding operators,
  that is if there exists a non-negative function $f\in \busa(X)$ (or $f\in \blsa(X)$) that satisfies the inequality $f\geq \r^*[f]$ (or $f\geq \r_*[f]$),
  then it holds that $f(\cdot)\geq \M^*(\cdot;S\U G)$ (or $f(\cdot)\geq \M^*(\cdot;S\U G)$).
\end{theorem}

\begin{proof}
  We provide an explicit proof for the minimization problem, and appeal to duality for the maximization case.

  First of all, the fact that $\M_*(\cdot;S\U^0 G) = 1_G(\cdot)$ follows immediately from the definition of the constrained reachability.
  Furthermore, for any $n\in \bar \N_0$ by Theorem \ref{thm:add.cost} we have that $\M_*(x;S\U^n G) = \hat \M_*(x,q^s;\hat J_n)$.
  The DP recursion for the TC  is given in \cite[Proposition 8.2, Proposition 9.8]{bs1978},
  and applied here yields the following:
  \begin{align*}
    \M_*(x;S\U^{n+1} G) &= \hat\M_*(x,q^s;\hat J_{n+1}) = \inf_{u\in \K_x}\left(c(x,q^s) + \hat\T^u \left[\hat \M_*(x,q^s;\hat J_n)\right]\right)
    \\
    & = \inf_{u\in \K_x}\left(1_G(x) + 1_S(x)\T^u \left[\hat \M_*(x,q^s;\hat J_n)\right]\right)
    \\
    &= 1_G(x) + 1_S(x)\T_*\left[\M_*(x;S\U^n G)\right] = \r_*\left[\M_*(x;S\U^n G)\right].
   \end{align*}
  This results in both the DP recursion ($n<\infty$) and in the fixpoint equation (for $n=\infty$).

  Consider now a non-negative function $f \in \blsa(X)$ satisfying $f\geq \r_*[f]$,
  and define a new function $\hat f:\hat X\to [0,\infty)$ by the formula $\hat f(x,q) := 1\{q = q^s\} \cdot f(x)$.
  Clearly, the function $\hat f$ is zero off $q^s$,
  so that we obtain:
  \begin{align*}
    \inf_{u\in \K_x}\left(c(x,q) + \hat \T^u\hat f(x,q)\right) = 1\{q = q^s\}\cdot\r_*f(x) \leq 1\{q = q^s\}\cdot f(x) = \hat f(x,q).
  \end{align*}
  As a result, \cite[Proposition 9.10 (P)]{bs1978} implies that $\hat\M_*(\cdot;J_\infty) \leq \hat f(\cdot)$ and thus
  \begin{equation*}
    \M_*(x;S\U G) = \hat\M_*(x,q^s;\hat J_\infty)\leq \hat f(x,q^s) = f(x),
  \end{equation*}
  so $\M_*(\cdot;S\U G)$ is the least fixpoint in the class of non-negative $\blsa$ functions.
\end{proof}

In view of Theorem \ref{thm:reach.DP} we can compute the value of the bounded horizon optimal constrained reachability problems backward-recursively,
starting from the indicator function $1_G$.
The computation of the fixpoint problem is more intricate and is addressed below in Section \ref{ssec:reach.comp}.
Due to this reason,
it is worth discussing the relation between the solution of the constrained reachability problem on the bounded time horizon,
and that on the unbounded time horizon.
In particular, an interesting question is whether
the latter can be in general obtained as the limit of the former,
as the time index $n$ goes to infinity.
This is one of the anticipated cases where the difference between the maximization and minimization problems becomes important:
the answer is positive in the first case and is negative in the second.

\begin{proposition}\label{prop:reach.fin2inf}
  For every state $x\in X$ it holds that
  \begin{equation}\label{eq:lim.reach.max}
    \M^*(x;S\U G) = \lim_{n\to\infty}\M^*(x;S\U^n G).
  \end{equation}
  Furthermore, for any $x\in X$ there exists a limit
  \begin{equation}\label{eq:lim.reach.min}
    f_*(x) := \lim_{n\to\infty}\M_*(x;S\U^n G)\leq \M_*(x;S\U G).
  \end{equation}
  Moreover, $\M_*(\cdot;S\U G) = f_*(\cdot)$ if and only if $f_*$ is a fixpoint of the DP operator $\r_*$.
\end{proposition}

\begin{proof}
  We start with the maximization case:
  recall that it corresponds to Assumption (N) of \cite[Chapter 9]{bs1978} since $\M^*(x;S\U^n G) = -\hat\M_*(x,q^s;-\hat J_n)$ for any $x\in X$ .
  It follows from \cite[Section 9.5]{bs1978} that the sequence $(\hat\M_*(x,q;-\hat J_n))_{n\in \N}$ has a limit for any $x\in X$ and $q\in Q$.
  Furthermore, \cite[Proposition 9.14]{bs1978} implies that this limit is $\hat\M_*(x,q;-\hat J_\infty)$, which leads to \eqref{eq:lim.reach.max}.

  For the minimization case we satisfy Assumption (P) of \cite[Chapter 9]{bs1978}.
  The discussion in \cite[Section 9.5]{bs1978} implies the existence of the point-wise limit for the sequence $(\hat\M_*(x,q;\hat J_n))_{n\in \N}$:
  we denote this limit by $\hat f_*$.
  Furthermore, it follows from \cite[Proposition 9.16]{bs1978} that $\hat f_*(\cdot) \leq \hat\M_*(\cdot;\hat J_\infty)$,
  and that the equality holds if and only if $\hat f_*$ is a fixpoint of the corresponding DP operator, i.e.
  \begin{equation}\label{eq:j.star.fp}
    \hat f_*(x,q) = c(x,q) + \hat \T_* \hat f_*(x,q).
  \end{equation}
  For the constrained reachability case,
  we now obviously have the existence of the limit
  \begin{equation*}
    f_*(x) := \lim_{n\to\infty}\M_*(x;S\U^n G) = 1\{q = q^s\}\hat f_*(x,q).
  \end{equation*}
  Clearly, $f_*(\cdot)\geq \M_*(\cdot;S\U G)$;
  if $f_*$ is a fixpoint of $\r_*$, then $\hat f_*$ satisfies \eqref{eq:j.star.fp},
  thus $\hat f_*(\cdot) = \hat\M(\cdot;\hat J_\infty)$ and hence $f_*(\cdot) = \M_*(\cdot;S\U G)$.
  Conversely, if $f_*(\cdot) = \M_*(\cdot;S\U G)$ then by Theorem \ref{thm:reach.DP} it has to be a fixpoint of the DP operator $\r_*$.
\end{proof}

The following example shows that the inequality in \eqref{eq:lim.reach.min} can be strict.\footnote{
    The example is obtained by modifying \cite[Example 1]{bs1978}.
  }

\begin{example}
\label{ex:reach.fin.inf}
  Let $X = \N_0$ and let $U = \{1/n\}_{n\in \N}\cup\{-1\}$.
  Define admissible controls as follows:
  $\K_0 = \{1/n\}_{n\in \N}$ and $\K_x = -1$ for $x\neq 0$.
  The dynamics is deterministic and is given by the following update law:
  \begin{equation*}
    \xx_{n+1} = \xx_n + 1/\uu_n.
  \end{equation*}
  Define $G := \{1\}$ to be the goal set,
  and let the safe set be its complement $S := X\setminus G$.
  Let us focus on the case when $\xx_0=0$.
  If we would like to minimize the probability of reaching $G$ over some finite horizon $n\in \N$,
  one of the optimal strategies is to choose $\uu_0 = \frac{1}{n+1}$.
  Then $\xx_1 = n+1$, $\xx _2 = n$ and $\xx_n = 2$,
  so that $G$ is not reached.
  As a result, for any finite $n\in \N_0$ we have that $\M_*(0;S \U^n G) = 0$.
  However, regardless of the chosen control action $\uu_0$,
  the set $G$ is reached by the path of the process in at most $\frac{1}{\uu_0}$ steps.
  Thus,
  \begin{equation*}
    \M_*(\cdot;S\U G) = 1\neq \lim_{n\to\infty}\M_*(0;S \U^n G).
  \end{equation*}
\end{example}

So far we have developed DP over the value functions for the constrained reachability problem.
The main tool we have used is a TC reformulation of the original performance criterion,
which makes it possible to apply the rich theory that has been developed for the TC problem.
Following similar lines as in the proofs of the theorems above,
one can reformulate for the constrained reachability problem almost any result developed for the TC criterion.
While in this paper we do not have a focus on the existence of optimal strategies,
one can easily tailor a number of results from \cite{bs1978},
as we overview next.
Recall that Assumption (P) in \cite[Chapter 9]{bs1978} corresponds to the minimization problem,
whereas Assumption (N) corresponds to the maximization one.
\begin{itemize}
  \item[(P)] \cite[Proposition 9.12]{bs1978} and its corollary provide necessary and sufficient conditions for the optimality of stationary policies,
  together with results to compute such policies.
  Moreover, \cite[Propositions 9.17, 9.18]{bs1978} and their corollaries provide various sufficient conditions for the existence of optimal stationary policies,
  for their Borel measurability,
  and for the equality in \eqref{eq:lim.reach.min}.
  \item[(N)] \cite[Proposition 9.13]{bs1978} gives necessary and sufficient conditions for the optimality of stationary policies.
  However, it does not give a way to construct a policy (such as the one available for (P)).
 This is almost the only result on the optimality of policies under Assumption (N).
\end{itemize}

\subsection{
Reachability problem: computation}
\label{ssec:reach.comp}

The TC formulation of the constrained reachability problem not only leads to results for the characterization of its solution,
but also connects to computational methods \cite{dpr2012}.
Alternatively, numerical methods with precise bounds on the error can be developed directly for the constrained reachability problem as in \cite[Section 4]{tmka2013}.
The latter methods are based on a partitioning of the state and action spaces $X$ and $U$ in order to approximate the original \lcdt-MP by a finite one.
Provided certain kinds of continuity assumptions on the kernel $\T$, such methods assure that a bounded-horizon value function can be found with any given precision if the partition is fine enough.
In the present context we are interested in extending these results to the unbounded time horizon case.

Let us recall the classical theory for the DC performance criterion.
If its discounting factor satisfies $\gamma<1$,
one falls into the setting of discounted problems for which the corresponding DP operator is contractive on some function space.
Such a property has nice consequences:
the unbounded-horizon value function is the unique fixpoint of this operator,
and it can also be efficiently approximated by means of the bounded-horizon value functions,
as it follows from the contraction mapping theorem.\footnote{
  The contraction mapping theorem is alternatively known as Banach's Fixed Point Theorem~\cite[Proposition A.1]{hl1989}.
}
This approach is clearly interesting to us because of the computational techniques developed for the bounded time horizon case.
Unfortunately, the DC formulation of the constrained reachability problem \eqref{eq:reach.TC} has a discounting factor $\gamma = 1$,
so the contractivity of the DP operators $\r^*$ and $\r_*$ cannot be established using classical techniques.
Due to this reason,
we come up with new sufficient conditions for the DP operators associated to the constrained reachability problem to be contractive:
the approach is based on the following result,
which is similar to that in \cite[Proposition A.2]{hl1989}.

\begin{lemma}\label{lem:contr}
  Let $A$ be any set,
  and let $(\f,\rho)$ be a metric space where $\f$ is any class of bounded functions $f:A\to\R$ and $\rho$ is a $\sup$-norm.
  Consider an arbitrary operator $\G:\f\to\f$ that satisfies the following two properties:
  \begin{enumerate}
    \item if $f,g\in \f$ such that $f\leq g$, then $\G f\leq \G g$,
    \item there exists $\beta \in [0,1)$ such that if $f\in \f$ and $c\geq 0$ then $\G (f+c)\leq \G f + \beta c$.
  \end{enumerate}
  Then $\G$ is a contraction on $\f$ with a modulus $\beta$.
\end{lemma}

\begin{proof}
  Let $f,g \in \f$ be arbitrary,
  then $f\leq g + \rho(f,g)$ and thus
  \begin{equation*}
    \G f \leq \G g+ \beta \rho(f,g) \implies \G f - \G g \leq \beta \rho(f,g).
  \end{equation*}
  By a symmetric argument, we obtain that
  \begin{equation*}
    \G f - \G g \leq \beta \rho(f,g) \implies |\G f - \G g| \leq \beta \rho(f,g) \implies \rho(\G f,\G g)\leq \beta\rho(f,g),
  \end{equation*}
  so that $\G$ is a contraction with a modulus $\beta$.
\end{proof}

The DP operators for the constrained reachability problem are rarely contractive
over the whole state space $X$,
so it is worth restricting attention to the safe set $S$ exclusively.
This also emphasizes the leading role of the set $S$ in the solution of the problem (in contrast to the goal set $G$), as we discussed before:
we have already mentioned that the solution of the constrained reachability problem is trivial outside of the safe set \eqref{eq:reach.S.G.},
so we can work with the restriction of value functions to the set $S$.
Consider the ``truncated'' transition operator:
\begin{equation*}
  {_S}\T^u f(x) := \int_S f(x')\T(\d x'|x,u),
\end{equation*}
which clearly maps the space $\u(S)$ into itself.
Furthermore, let us define
\begin{equation*}
  {_S}\T^* f(x) := \sup_{u\in \K_x} {_S}\T^u f(x),\qquad {_S}\T_* f(x) := \inf_{u\in \K_x}{_S}\T^u f(x).
\end{equation*}
Note that the operator ${_S}\T^*$ (${_S}\T_*$) maps the space $\busa(S)$ ($\blsa(S)$) into itself.
Moreover, for $f\in \blsa(X)$ it holds that $f|_S \in \blsa(S)$,
which follows immediately from the definition of lower-semianalytic functions and Borel measurability of $S$;
clearly, the same applies to the restrictions of functions in $\busa(X)$.
In particular, if we define
\begin{equation*}
  W_n(x):=\M^*(x;S\U^n G)|_S, \qquad w_n(x):=\M_*(x;S\U^n G)|_S
\end{equation*}
for any $x\in X$ and $n\in \N_0$,
then for these functions it holds that $W_n\in \busa(S)$ and $w_n\in\blsa(S)$.
Thus, we can rewrite the DP over the safe set $S$ as follows:
\begin{equation*}
  W_{n+1} = {_S}\r^*\left[W_n\right], \qquad w_{n+1} = {_S}\r_* \left[w_n\right]
\end{equation*}
for any $n\in \bar\N_0$,
where $W_0 = w_0 = 0$,
and the truncated DP operators are given by
\begin{align*}
  {_S}\r^* f(x) &:= \sup_{u\in \K_x}\left[\T(G|x,u) +{_S}\T^u f(x)\right], \qquad f\in \busa(S),
  \\
  {_S}\r_* f(x) &:= \inf_{u\in \K_x}\left[\T(G|x,u) +{_S}\T^u f(x)\right], \qquad f\in \blsa(S).
\end{align*}
Clearly, these operators map their domains into themselves,
so that they can be applied recursively.
Note also that in case $G = \emptyset$,
we have ${_S}\r^u = {_S}\T^u$.

In order to formulate the main result on the contractivity of the DP operators,
we are only left to introduce a very important special case of constrained reachability:
safety \cite{APLS08b}.
This can be characterized by the LTL formula $\square^n S$ and thus
\begin{equation*}
  \M^\pi(x;\square^n S) = 1 - \M^\pi(x;S\U^n S^c)
\end{equation*}
for all $x\in X$ and any $n\in \bar\N_0$.
We are interested in the restriction of the safety problem to the safe set $S$ itself,
the main focus being the characterization of contractivity.\footnote{
  Clearly, $V^\pi_n(x;S) = 0$ for any $x\in S^c$,
  so the safety problem is trivial outside the safe set.
}
We further denote
\begin{equation*}
  V_n(x) := \M^*(x;\square^n S)|_S,\qquad v_n(x) := \M_*(x;\square^n S)|_S.
\end{equation*}
The DP for safety over $S$ is hence given by
\begin{equation*}
  V_{n+1} = {_S}\T^* \left[V_n\right], \qquad v_{n+1} = {_S}\T_* \left[v_n\right], \quad n\in \bar\N_0.
\end{equation*}
with $V_0 = v_0 = 1$.
Clearly, we have that $0\leq V_n \leq 1$ for all $n\in \bar\N_0$.
Let us define
\begin{equation*}
  \beta_n(S) := \sup_{x\in S}V_n(x) = \sup_{x\in X}\M^*(x;\square^n S) \in [0,1],
\end{equation*}
\begin{equation*}
  m(S) := \inf\{n\in \N_0: \beta_n(S)<1\} \in \bar\N_0,
\end{equation*}
and note that both $\beta_n$ and $m$ are monotonic functions of $S$ with respect to set inclusion.
We are now ready to provide sufficient conditions for contractivity.

\begin{theorem}\label{thm:reach.contr}
  If $m:=m(S)<\infty$,
  then operators $({_S}\r^*)^m$ and $({_S}\r_*)^m$ are contractions with modulus $\beta:=\beta_{m}(S)$ on the spaces $\busa(S)$ and $\blsa(S)$ respectively.
  In particular, each of them has a unique fixpoint,
  and for any $n\in \N_0$ the following inequalities hold true:
  \begin{equation}\label{eq:reach.contr}
    \rho(W_\infty,W_{mn})\leq \beta^n,\quad \rho(w_\infty,w_{mn})\leq \beta^n.
  \end{equation}
  Finally, as a special case $({_S}\T^*)^m$ and $({_S}\T_*)^m$ are contractions and $V_\infty = v_\infty = 0$.
\end{theorem}

\begin{proof}
  We are going to apply Lemma \ref{lem:contr} in order to establish the contractivity property.
  Let us consider the case of ${_S}\r^*$ first,
  so in Lemma \ref{lem:contr} we put $\f = \busa(S)$.
  The condition $(1)$ of the lemma is obviously satisfied for ${_S}\r^*$ and hence for $({_S}\r^*)^n$ regardless of $n\in \N_0$.
  Furthermore, for any two functions $f,g \in \busa(S)$ we have that
  \begin{align*}
    {_S}\r^*(f(x)+g(x)) &= \sup_{u\in \K_x}\left[\T(G|x,u) +{_S}\T^uf(x) +{_S}\T^ug(x)\right]
    \\
    &\leq \sup_{u\in \K_x}\left[\T(G|x,u) +{_S}\T^uf(x)\right] + \sup_{u\in \K_x}{_S}\T^ug(x)
    \\
    &= {_S}\r^* f(x) + {_S}\T^* g(x).
  \end{align*}
  As a result,
  for any $f\in \busa(S)$ and any $c\geq 0$ it holds that
  \begin{align*}
    {_S}\r^*(f+c) \leq {_S}\r^* f + c\cdot V_1,
  \end{align*}
  and further by induction for any $n\in \N$
  \begin{align*}
    ({_S}\r^*)^n(f+c) \leq ({_S}\r^*)^n f + c\cdot V_n.
  \end{align*}
  In particular, for the case $n=m$ we obtain the following:
  \begin{align*}
    ({_S}\r^*)^m(f+c) \leq ({_S}\r^*)^m f + c\cdot V_m \leq ({_S}\r^*)^n f + c\cdot \beta.
  \end{align*}
  Hence, $({_S}\r^*)^m$ satisfies all the assumptions of Lemma \ref{lem:contr} and thus is a contraction on $\busa(S)$.
  The contractivity of $({_S}\r_*)^m$ can be shown by a similar argument,
  with the only difference being the inequality
  \begin{equation*}
    {_S}\r_*(f + g)\leq {_S}\r_* f + {_S}\T^* g,
  \end{equation*}
  rather than the one with ${_S}\T_* g$,
  and with conditions on contractivity that are state in terms of functions $V_n$ rather than $v_n$.

  After the contractivity of the operators is established,
  the uniqueness of the solutions of fixpoint equations and the bounds in \eqref{eq:reach.contr} follow immediately from the contraction mapping theorem~\cite[Proposition A.1]{hl1989}.
  Finally, the statement for operators ${_S}\T^*$ and ${_S}\T_*$ follows directly if one considers the special case $G = \emptyset$.
\end{proof}

Theorem \ref{thm:reach.contr} shows that in the case of contractive operators the unbounded-horizon value function can be approximated by bounded-horizon ones with any precision level.
However, there are some questions left:
what are the cases in which the contractivity conditions are violated,
and what would be a solution for such cases?
Let us first address the former question.
For example, whenever the conditions of Theorem \ref{thm:reach.contr} are met,
the equality holds in \eqref{eq:lim.reach.min}.
As a result, Example \ref{ex:reach.fin.inf} does not admit contractive operators since the equality does not hold there.
Some of other important examples can be given using the notion of absorbing set.

\begin{definition}\label{def:abs}
  The set $A\in \b(X)$ is called \emph{strongly absorbing} if $\T(A|x,u) = 1$ for all $x\in A$ and $u\in \K_x$.
  The set $A \in \b(X)$ is called \emph{weakly absorbing} if there exists a randomized selector $\mu \in \u(U|X)$ such that for all $x\in A$ it holds that $\mu(\K_x|x) = 1$ and that
  \begin{equation}\label{eq:abs.mu}
    \int_{\K_x}\T(A|x,u)\mu(\d u|x) = 1.
  \end{equation}
  We say that the set $A\in \b(X)$ is \emph{simple} if it does not have non-empty weakly absorbing subsets.
\end{definition}

The following notation is extensively used below:
for any $A\in \b(X)$ we define
\begin{equation*}
  \K^A := \{(x,u)\in \K:\T(A|x,u) = 1\}.
\end{equation*}
Note in particular that if the sets $A,B\in \b(X)$ are such that $B\subseteq A$,
then $\T(B|x,u) = 1$ for some $(x,u)\in \K$ implies that $\T(A|x,u) =1$,
and as a result we obtain that $\K^B\subseteq \K^A$.
The next theorem establishes some important results on the connection between weakly and strongly absorbing sets,
and on their structure.

\begin{proposition}\label{prop:w.a.s.a}
  Let $A \in \b(X)$. It holds that
  \begin{enumerate}
    \item[i.] if $A$ is strongly absorbing,
    then it is weakly absorbing,
    \item[ii.] if $A$ is weakly absorbing,
    then the randomized selector $\mu$ in \eqref{eq:abs.mu} can be equivalently replaced by a deterministic selector $f\in \u(X)/\b(U)$.
  \end{enumerate}
\end{proposition}

\begin{proof}
  To prove $i.$ we use the fact that $\T(A|x,\k(x))=1$ for any $x\in A$.
  Hence, the kernel $\mu$ as per \eqref{eq:abs.mu} can be chosen to be the deterministic as $\mu = \delta_\k$.

  With focus on $ii.$ let us fix some arbitrary $x\in A$ and show that there exists $u\in \K_x$ such that $\T(A|x,u) = 1$.
  Note that if $u\notin \K^A_x$, then $1-T(A|x,u)>0$, where it is crucial that the inequality is strict.
  To reach a contradiction we further suppose that for a $\mu$ as in \eqref{eq:abs.mu} it holds that $\mu(\K_x\setminus \K^A_x|x) >0$.
  Then:
  \begin{equation*}
    0 = \int\limits_{\K_x}(1-\T(A|x,u))\mu(\d u|x)\geq \int\limits_{\K_x\setminus \K^A_x}(1-\T(A|x,u))\mu(\d u|x)>0,
  \end{equation*}
  which obviously cannot be true.
  As a result, we obtain that $\mu(\K^A_x|x) = 1$ and in particular $\K^A_x\neq \emptyset$ for any $x\in A$.
  Hence, it holds that $\T^*1_A(x) = \sup_{u\in \K_x}\T(A|x,u) = 1$.
  The existence of a universally measurable selector $u$ from  $\K^A_x$ thus follows from \cite[Proposition 7.50 (b)]{bs1978} and the fact that $\T(A|\cdot)\in \bb(\K) \subseteq \ba^*(\K)$.
\end{proof}

Part $i.$ of Proposition \ref{prop:w.a.s.a} justifies the use of the adjectives ``weak'' and ``strong'' in Definition \ref{def:abs}.
Furthermore, in the uncontrolled case (where trivially $U = \{u\}$),
the notion of weak and strong sets coincide with that of an absorbing set \cite{mt1993}.
Intuitively, a strongly absorbing set remains absorbing under any possible control action,
whereas for a weakly absorbing set there has to exist a control policy that makes this set absorbing.
Moreover, thanks to part $ii.$ of Proposition \ref{prop:w.a.s.a},
it is sufficient to consider non-randomized controls in order to establish weak absorbance.

As promised,
absorbing sets can be used to provide examples when the contractivity of truncated DP operators is violated,
and in particular when the fixpoint equations do not have unique solutions.
Note that in the case of the unconstrained reachability $G = S^c$,
it holds that the operators ${_S}\r^*$ and ${_S}\r_*$ always admit the trivial fixpoint $1$.
However, if $S$ is not simple (that is, if it admits absorbing subsets),
then the optimal value functions are different than $1$.
For example, if a trajectory starts in an absorbing subset of $S$ then it never reaches the goal set.
More precisely:

\begin{proposition}\label{prop:abs.non.contr}
  If a set $S$ has a non-empty strongly (weakly) absorbing subset $A\subseteq S$,
  then $\M^*(x;S\U S^c) = 0$ ($\M_*(\cdot;S\U S^c) = 0$) for all $x\in A$.
  In particular, $W_\infty(x) = 0$ ($w_\infty(x) = 0$) for all $x\in A$,
  and $({_S}\r^*)^n$ ($({_S}\r_*)^n$) is not a contraction for any $n\in \N_0$.
\end{proposition}

\begin{proof}
  Let $A$ be a strongly absorbing set and fix a point $x\in A$.
  Then $\pr^\pi_x(\xx_n\in A) = 1$ for all $n\in \N_0$ regardless of the policy $\pi\in \Pi$.
  As a result, $\pr^\pi_x(\xx_n\in S^c) = 0$ for all $n\in \N_0$, so
  \begin{equation*}
    \M^\pi(x;S\U S^c) \leq \sum_{n=0}^\infty \pr^\pi_x(\xx_n\in S^c) = 0
  \end{equation*}
  for any policy $\pi\in \Pi$.
  Thus, we obtain that $\M^*(x;S\U S^c) = 0$ for any $x\in A$.
  Clearly, it follows immediately that $W_\infty(x) = 0$ for all $x\in A$.
  Suppose now that $({_S}\r^*)^n$ is contractive for some $n$.
  In such a case the solution of the fixpoint equation would be unique and hence it would imply that $W_\infty \equiv 1$, which is clearly not the case.

  Let now $A$ be a weakly absorbing set and consider a stationary policy $\pi\in \Pi_S$ with
  \begin{equation*}
    \pi_0(x) := 1_A(x)\cdot \mu(x) + 1_{A^c}(x)\cdot \delta_\k(x),
  \end{equation*}
  with $\mu$ as in \eqref{eq:abs.mu}.
  The policy $\pi$ uses the choice of the action suggested by $\mu$ whenever $x\in A$,
  and chooses some auxiliary action $\k(x)$ otherwise.
  From the definition of $\mu$ it follows that $\pr^\pi_x(\xx_n\in A) = 1$ and hence $\pr^\pi_x(\xx_n\in S^c) = 0$ for all $x\in A$, so
  \begin{equation*}
    \M_*(\cdot;S\U S^c)\leq \M^\pi(x;S \U S^c) \leq \sum_{n=0}^\infty \pr^\pi_x(\xx_n\in S^c) = 0.
  \end{equation*}
  As for ${_S}\r^*$, one can now show that $({_S}\r_*)^n$ is not a contraction for any $n\in \N_0$.
\end{proof}

In general the presence of absorbing sets is not the only reason that may violate contractivity.
For example, it is easy to see that the set $S$ in Example \ref{ex:reach.fin.inf} does not have weakly absorbing subsets,
and still the contractivity does not hold.
However, the following assumption allows to characterize precisely the relationship between absorbing sets and contractivity.

\begin{assumption}\label{as:cont.comp}
  The \cdt-MP $\D$ is continuous and the set $S$ is compact.
\end{assumption}

We are going to show that, under Assumption \ref{as:cont.comp},
the case $m(S)<\infty$ precisely coincides with the case when $S$ does not admit weakly absorbing sets.
In order to prove this fact some preparation is required:
let us define for all $n\in \N_0$ the sets
\begin{equation*}
  S_n := \left\{\M^*(\cdot,\square^n S) = 1\right\} = \left\{x\in X:\M^*(x,\square^n S) = 1\right\}.
\end{equation*}
Note that for any $x\in S$ and $\pi\in \Pi$ the sequence $(\M^\pi(x;\square^n S))_{n\in \N_0}$ is non-increasing,
as is the sequence $(\M^*(x;\square^n S))_{n\in \N_0}$.
As a result, we obtain that the sequence of sets $(S_n)_{n\in \N_0}$ is non-increasing as well: $S_{n+1} \subseteq S_n$ for all $n\in \N_0$.
Let us further denote by $S_\infty := \bigcap_{n=0}^\infty S_n$ the limit of this sequence.
We introduce the following auxiliary lemmas.

\begin{lemma}\label{lem:abs.1}
  The set $S_\infty$ is such that $\{\M^*(\cdot;\square S) = 1\}\subseteq S_\infty$.
  In particular, if $S'\subseteq S$ is a weakly absorbing subset of $S$,
  then $S'\subseteq S_\infty$.
\end{lemma}

\begin{proof}
  Let us fix any $x$ such that $\M^*(x;\square S) = 1$.
  By Proposition \ref{prop:reach.fin2inf} we have that
  \begin{equation*}
    \lim_{n\to\infty}\M^*(x;\square^n S)\geq \M^*(x;\square S) = 1.
  \end{equation*}
  Since $(\M^*(x;\square^n S))_{n\in \N_0}$ is a non-increasing sequence,
  it follows that $\M^*(x;\square^n S) = 1$ and hence $x\in S_n$ for all $n\in \N_0$.
  As a result, $x\in S_\infty$ and thus $\{\M^*(\cdot;\square S) = 1\}\subseteq S_\infty$.
  Now, if $S'\subseteq S$ is weakly absorbing,
  then $S' \subseteq \{\M^*(\cdot;\square S) = 1\}$ by Proposition \ref{prop:abs.non.contr}.
\end{proof}

\begin{lemma}\label{lem:abs.2}
  Under Assumption \ref{as:cont.comp} it holds that $\M^*(\cdot;\square^n S) \in \bc^*(X)$ for all $n\in \N_0$.
\end{lemma}

\begin{proof}
  If $n = 0$ then $\M^*(x;\square^0 S) = 1_S(\cdot)\in \bc^*(X)$ since $S$ is a closed set being a compact subset of a metrizable space.
  Also, if $\M^*(\cdot;\square^n S)\in \bc^*(X)$,
  then by continuity of the kernel $\T$ we have that $\T \M^*(\cdot;\square^n S) \in \bc^*(\K)$ and $\T^* \M^*(\cdot;\square^n S)\in \bc^*(X)$ as it follows from \cite[Proposition 7.31]{bs1978} and \cite[Proposition 7.33]{bs1978} respectively.
  Finally, $\M^*(\cdot;\square^{n+1} S) = 1_S(\cdot)\cdot \T^*\M^*(\cdot;\square^n S) \in \bc^*(X)$ by Lemma \ref{lem:semicont.ind} in the Appendix.
\end{proof}

\begin{lemma}\label{lem:abs.3}
  Under Assumption \ref{as:cont.comp}, sets $S_n$ and $\K^{S_n}_x$ are compact for all $x\in X$, $n\in \N_0$.
\end{lemma}

\begin{proof}
  Since $S_n = \{\M^*(\cdot;\square^n S)\geq 1\}$ and $\M^*(\cdot;\square^n S)$ is an upper semi-continuous function by Lemma \ref{lem:abs.2},
  we obtain that $S_n$ is a closed set.
  It is also compact as a closed subset of a compact set $S$.
  Furthermore, it holds that $\T1_{S_n} \in \bc^*(\K)$ since the set $S_n$ is closed.
  Hence, $\K^{S_n} = \{\T1_{S_n}\geq 1\}$ is a closed subset of $\K$,
  which implies that $\K^{S_n}(x)$ is a closed subset of $U$ for any $x\in X$,
  and is compact since $U$ is compact.
\end{proof}

\begin{lemma}\label{lem:abs.4}
  Under Assumption \ref{as:cont.comp}, $S_{n+1} = \{x\in S:\K^{S_n}(x)\neq \emptyset\}$ for any $n\in \N_0$,
  that is
  \begin{equation}\label{eq:S_n}
    S_{n+1} = \{x\in S:\exists u\in \K_x \text{ s.t. } \T(S_n|x,u) = 1\}.
  \end{equation}
  Moreover, $S_\infty$ is weakly absorbing and satisfies the formula $S_\infty = \{\M(\cdot;\square S) = 1\}$.
\end{lemma}

\begin{proof}
  Let us first prove \eqref{eq:S_n} for $n<\infty$.
  We first show that if $\K^{S_n}(x)\neq \emptyset$ for some $x\in S$,
  then $x\in S_{n+1}$.
  Indeed, let $u'$ be an arbitrary element of $\K^{S_n}(x)$.
  We have:
  \begin{equation*}
    \M^*(x;\square^{n+1} S) = \sup_{u\in \K_x}\int_X \M^*(x';\square^n S)\T(\d x'|x,u)\geq \int_{S_n} \M^*(x';\square^n S)\T(\d x'|x,u') = 1,
  \end{equation*}
  so that $\K^{S_n}(x)\neq\emptyset$ for $x\in S$ implies $x\in S_{n+1}$.
  Showing the converse implication is more technical.
  Let us pick $x\in S$ arbitrarily.
  Since $\T \M^*(\cdot;\square^n S)\in \bc^*(\K)$ as a function of $x$ and $u$,
  it follows that $\T \M^*(\cdot;\square^n S) \in \bc^*(\K_x)$ as a function of $u$,
  and thus the maximum in
  \begin{equation*}
    \M^*(x;\square^{n+1} S) = \sup_{u\in \K_x}\T \M^*(x,u;\square^n S)
  \end{equation*}
  is attained at some $u''\in \K_x$ as the latter set is compact.
  As a result,
  \begin{equation*}
    \int_S 1\;\T(\d x'|x,u'') = \T(S|x,u'')\leq 1 = \M^*(x;\square^{n+1} S) = \int_S \M^*(x';\square^n S)\T(\d x'|x,u''),
  \end{equation*}
  and subtracting the right-hand side from the left-hand side yeilds
  \begin{equation}\label{eq:lem.S_n.ii}
    \int_S \left[1 - \M^*(x';\square^n S)\right]\T(\d x'|x,u'') \leq 0.
  \end{equation}
  Note that the integrand in \eqref{eq:lem.S_n.ii} is non-negative,
  so
  \begin{equation*}
    \T(\{1 - \M^*(\cdot;\square^n S) = 0\}|x,u'') = \T(S_n|x,u'') = 1,
  \end{equation*}
  since otherwise the integral would be strictly positive.
  Thus, \eqref{eq:S_n} is proved.

  With focus on $S_\infty$,
  the case when $S_\infty$ is empty is trivial,
  so let us assume $S_\infty\neq\emptyset$.
  For any $x\in S_\infty$ it follows that $x\in S_n\neq\emptyset$ for all $n\in \N_0$ and hence $\K^{S_n}(x)\neq \emptyset$ for all $n\in \N_0$.
  Indeed, in case $\K^{S_n}(x) = \emptyset$ we would obtain that $x\notin S_{n+1}$ thanks to \eqref{eq:S_n} which contradicts with the fact that $x\in S_\infty$.
  Then $\K^\infty(x) := \bigcap_{n=0}^\infty \K^{S_n}(x) \neq\emptyset$ since $(\K^{S_n}(x))_{n\in N_0}$ is a non-increasing sequence of non-empty compact sets.
  For any $u'\in \K^\infty(x)$ and any $n\in \N_0$ it holds that $\T(S_n|x,u') = 1$ so that
  \begin{equation*}
    \T(S_\infty|x,u') = \T\left(\bigcap_{n=0}^\infty S_n|x,u'\right) = \lim_{n\to\infty}\T(S_n|x,u') = 1
  \end{equation*}
  and so $x\in S_\infty$ implies that $\K^{S_\infty}(x)\neq \emptyset$.
  As a result, we obtain that for all $x\in S_\infty$
  \begin{equation*}
    \T^*1_{S_\infty}(x) = \sup_{u\in \K_x}\T(S_\infty|x,u) = 1.
  \end{equation*}
  The set $S_\infty$ is compact as an intersection of compact sets,
  so that $\T(S_\infty|\cdot) \in \bc^*(\K)$,
  and thus by \cite[Proposition 7.33]{bs1978} there exists a selector $f:\b(X)/\b(U)$ such that $f(x)\in \K_x$ for all $x\in X$ and $\T(S_\infty|x,f(x)) = 1$ for all $x\in S_\infty$.
  Hence, $S_\infty$ is a weakly absorbing subset and thus $S_\infty \subseteq \{\M^*(\cdot;\square S)  =1\}$.
  Combining the latter statement with the result of \eqref{lem:abs.1},
  we obtain that $S_\infty = \{\M^*(\cdot;\square S)  =1\}$.
\end{proof}

Lemma \ref{lem:abs.4} shows that when the \cdt-MP $\D$ is continuous,
any compact set $S$ admits a largest weakly absorbing subset $S_\infty$ (which clearly may be empty).
We are now able to show that this is equivalent to the contractivity condition.

\begin{theorem}\label{thm:absor.equiv}
  Under Assumption \ref{as:cont.comp}, the following statements are equivalent:
  \begin{itemize}
    \item[i.] it holds that $m(S)<\infty$ (contractivity);
    \item[ii.] the operator ${_S}\T^*$ has a unique fixpoint (uniqueness);
    \item[iii.] it holds that $\M^*(\cdot;\square S) = 0$ (triviality);
    \item[iv.] it holds that $S_\infty = \emptyset$ (simplicity).
  \end{itemize}
\end{theorem}

\begin{proof}
  The fact that $i.\implies ii.$ has been proven in Theorem \ref{thm:reach.contr}.
  Also, ${_S}\T^* f = f$ always has a solution $f = 0$,
  so the uniqueness of a fixpoint of ${_S}\T^*$ implies $V_\infty = 0$ and thus $ii.\implies iii.$
  If $\M^*(\cdot;\square S) = 0$, then by Lemma \ref{lem:abs.4} we have $S_\infty = \{\M^*(\cdot;\square S) = 1\} = \emptyset$,
  so $iii.\implies iv.$
  Finally, if $m(S) = \infty$ then $\sup_{x\in X}\M^*(x;\square^n S) = 1$ for all $n\in N_0$.
  By Lemma \ref{lem:abs.2} each of the functions $\M^*(\cdot;\square^n S)$ is u.s.c. and thus it attains its maximum over a compact set $S$,
  so that $m(S) = \infty$ implies $S_n\neq \emptyset$ for all $n\in \N_0$.
  Moreover, from Lemma \ref{lem:abs.3} it follows that each $S_n$ is compact and hence $(S_n)_{n\in \N_0}$ is a non-increasing sequence of non-empty compact sets.
  As a result, the latter sequence has a non-empty intersection $S_\infty$ and hence $S_\infty = \emptyset$ necessarily implies $m(S)<\infty$,
  hence $iv.\implies i.$
\end{proof}

We have obtained a precise characterization of the contractivity condition $m(S)<\infty$ in terms of presence or absence of weakly absorbing subsets of the safe set.
In particular, if both Assumption \ref{as:cont.comp} and the condition $S_\infty = \emptyset$ are satisfied,
then regardless of the set $G$ we are able to approximate $\M^*(x;S\U G)$ and $\M_*(x;S\U G)$ by their bounded horizon counterparts.
Moreover, Theorem \ref{thm:absor.equiv} also justifies the following intuitive statement:
if one wants to keep the path of the process inside a set with some non-zero probability,
there has to be an ``attractor'' within such set,
which in our case appears to be the largest weak absorbing subset of $S$,
that is $S_\infty$.
If such attractor is absent,
no matter what control policy is chosen,
the path will leave the desired set almost surely.
The ``if and only if'' nature of Theorem \ref{thm:absor.equiv} also implies that for the maximal safety problem such condition is necessary.
However, it still may be the case that $S_\infty \neq\emptyset$ but ${_S}\r_*$ is a contraction.
Although such cases are interesting to study,
this goes beyond the scope of the current paper:
we are now interested in techniques that allow us reducing the unbounded horizon problem to the bounded horizon one in the situation where $S_\infty \neq\emptyset$.
These results are particularly powerful under the following assumption.

\begin{assumption}\label{as:stat.policy}
  Stationary policies are sufficient for the solution of the constrained reachability problem on the unbounded time horizon,
  that is for any $x\in X$:
  \begin{equation*}
    \M^*(x;S\U G) = \sup_{\pi\in \Pi_S}\M^\pi(x;S\U G),\qquad \M_*(x;S\U G) = \inf_{\pi\in \Pi_S}\M^\pi(x;S\U G).
  \end{equation*}
\end{assumption}

Before we provide the main result,
the following technical lemma is needed.

\begin{lemma}\label{lem:decomp}
  Let Assumption \ref{as:stat.policy} hold true,
  and let $C\in \b(X)$ be any subset of $S$.
  Then $\forall x\in X$
  \begin{align*}
    |\M^*(x;S\U G) - \M^*(x;(S\setminus C)\U G)| &\leq \chi^*(C) := \sup_{\pi\in \Pi_S}\sup_{y\in C}\M^\pi(y;S\U G),
    \\
    |\M_*(x;S\U G) - \M_*(x;(S\setminus C)\U G)| &\leq \chi_*(C) := \inf_{\pi\in \Pi_S}\sup_{y\in C}\M^\pi(y;S\U G).
  \end{align*}
\end{lemma}

\begin{proof}
  For any $\pi \in \Pi_S$ let us denote $\chi^\pi(C) := \sup_{y\in C}\M^\pi(x;S\U G)$.
  Let us fix an arbitrary policy $\pi \in \Pi_S$ and an arbitrary state $x\in X$.
  Clearly, $S\setminus C\subseteq S$ further implies that $\M^\pi(x;S\U G)\geq \M^\pi(x;(S\setminus C)\U G)$ .
  On the other hand, obviously
  \begin{equation*}
    \M^\pi(\cdot;S\U G) - \M^\pi(\cdot;(S\setminus C)\U G) \leq \chi^\pi(C)
  \end{equation*}
  which together with Lemma \ref{lem:sup.inf.bounds} immediately yields the desired result.
\end{proof}

Let us discuss how
Lemma \ref{lem:decomp} can be useful.
Suppose that Assumption \ref{as:cont.comp} holds true and that for the original problem we have that $S_\infty \neq \emptyset$,
so that $m(S) = \infty$,
and hence we cannot apply Theorem \ref{thm:reach.contr} to compute the optimal value functions.
If we find a set $C\supseteq S_\infty$ such that $m(S\setminus C)<\infty$,
then we can solve the unconstrained problem with truncated safe set $S\setminus C$.
Also, since $C$ contains $S_\infty$ we can expect that $\chi^*(C)$ and $\chi_*(C)$ are close enough to zero,
which would make the bounds in Lemma \ref{lem:decomp} useful.
To further elaborate this idea we need the notion of a {\it locally excessive} function.

\begin{definition}\label{def:loc.exc}
  A non-negative function $g\in \bb(X)$ is called \emph{locally $\mu$-excessive} for a randomized selector $\mu\in \u(U|X)$,
  if for any $x\in\{g \leq 1\}$ it holds that $\T^\mu g(x)\leq g(x)$.
  If in addition $A_\infty\subseteq \{g = 0\}$ and $\{g\leq 1\}\subseteq A$ for some $A\in \b(X)$,
  and $\{g<\ve\}$ is an open set for all $\ve>0$,
  we say that $g$ is locally $\mu$-excessive on $A$.

  A non-negative function $g\in \bb(X)$ is called locally uniformly excessive if for any $x\in \{g\leq 1\}$ and $u\in \K_x$ it holds that $(\T g)(x,u)\leq g(x)$.
  If in addition $A_\infty\subseteq \{g = 0\}$ and $\{g\leq 1\}\subseteq A$ for some $A\in \b(X)$,
  and $\{g<\ve\}$ is an open set for all $\ve>0$,
  we say that $g$ is locally uniformly excessive on $A$.
\end{definition}

\begin{theorem}\label{thm:reach.decomp}
  Let Assumptions \ref{as:cont.comp} and \ref{as:stat.policy} hold true.
  Suppose that $g^*$ is locally uniformly excessive on $S$,
  and that $g_*$ is locally $\pi'_0$-excessive for some $\pi'\in \Pi_S$.
  For any $\ve \in (0,1]$ it holds that the following inequalities are valid:
  \begin{equation*}
    \chi^*(\{g^*<\ve\})\leq \ve \qquad \chi_*(\{g^*<\ve\})\leq\ve
  \end{equation*}
  and that sets $S\setminus \{g^*<\ve\}$, $S\setminus \{g_*<\ve\}$ are simple.
\end{theorem}

\begin{proof}
  We start with the case of the maximization.
  For any policy $\pi\in \Pi_S$ we have that
  \begin{equation*}
    \M^\pi(x;X\U\{g^*>1\})\leq g^*(x)
  \end{equation*}
  whenever $x\in \{g^*\leq 1\}$,
  as it follows from \cite[Lemma 3]{ta2014}.
  Furthermore, since $\{g\leq 1\}\subseteq S$ and $G\subseteq S^c$,
  it holds that $G\subseteq \{g^*>1\}$.
  As a result,
  \begin{equation*}
    \M^\pi(\cdot;S\U G)\leq \M^\pi(\cdot;X\U \{g^*>1\}).
  \end{equation*}
  Combining both inequalities,
  we obtain that
  \begin{equation*}
    \sup_{y\in \{g^*<\ve\}}\M^\pi(y;S\U G)\leq \ve,
  \end{equation*}
  and thus after maximizing over all stationary policies we obtain that  $\chi^*(\{g^*<\ve\})\leq\ve$.

  For the case of the minimization we similarly have
  \begin{equation*}
    \sup_{y\in \{g_*<\ve\})}\M^{\pi'}(y;S\U G)\leq \ve,
  \end{equation*}
  and since $\chi_*(C) \leq \sup_{y\in C}\M^{\pi'}(y;S\U G)$ for any set $C\in \b(X)$,
  we immediately obtain that $\chi_*(\{g^*<\ve\})\leq\ve$ for all $\ve\leq 1$,
  as desired.

  Finally, the simplicity of sets $S\setminus \{g^*<\ve\}$ and $S\setminus \{g_*<\ve\}$ follows from the fact that that they are compact simple sets.
  Indeed, we have compactness thanks to the fact that $S$ is compact and sets $\{g^*<\ve\}$, $\{g^*<\ve\}$ are open.
  Moreover, the simplicity follows from the definition of functions locally excessive on $S$ which implies that $S_\infty \subseteq \{g^*<\ve\}$ and $S_\infty \subseteq \{g^*<\ve\}$.
\end{proof}

\subsection{Comments on the reachability problem}
\label{ssec:reach.comments}

Let us mention how the DP formulation has been developed for the (un)constrained reachability problem in the \cdt-MP setting.
To our knowledge,
the first work with this goal has been \cite{APLS08b},
which has considered a class of models called controlled discrete-time Stochastic Hybrid Systems (\cdt-SHS),
namely a class of \cdt-MP with a state space comprised of a collection of Borel subsets of $\R^n$.
It has treated the unconstrained reachability property $\lozenge^n G = \true \U^n G$ and the dual safety one $\square^n S = \neg \lozenge^n S^c$,
and has proposed their characterization using a maximal cost \eqref{eq:reach.cost.max} for the first problem,
and a multiplicative cost \eqref{eq:reach.cost.mult} for the second.
Within this formulations,
the DP recursion has been derived for the bounded time horizon $n<\infty$,
while restricting the attention to Markov policies.
\cite{SL10} has addressed a more general\footnote{
  Although the constrained reachability includes the unconstrained one as a special case,
  the latter can be used to solve the former if one just slightly modifies dynamic by making the set of unsafe sets $D = X\setminus (S\cup G)$ absorbing.
  Indeed, in such case $\lozenge^n G$ is equivalent to $S\U^n G$ since $G$ is never reached by a trajectory that has visited $D$ at least once~\cite[Proposition 1]{tmka2013}.
  In particular, one immediately obtains \cite[Theorem 8]{SL10} by applying \cite[Theorem 1]{APLS08b} over a modified model.
  Similarly, rendering the set $D$ absorbing allows one to recast a related terminal hitting-time reach-avoid problem \cite[Section 4]{SL10} as a special case of a terminal cost problem \cite[Section 3]{hll1996}.
}
constrained reachability problem $S \U^n G$ within a similar setting:
\cdt-SHS models,
Markov policies, and bounded time horizons:
a new sum-multiplicative cost \eqref{eq:reach.cost.add.mult} has been proposed,
leading to the DP scheme in \cite[Theorem 8]{SL10}.
In contrast to these studies,
here we have proposed a TC formulation,
which has allowed dealing with non-Markovian policies,
and to show that Markov policies are sufficient.
In particular, one obtains \cite[Theorems 1, 2]{APLS08b} and \cite[Theorem 8]{SL10} as special cases of Theorem \ref{thm:reach.DP}.
At the same time,
the TC formulation has also led to simpler proofs,
which mostly rely on known results for the TC performance criterion \cite[Chapters 8,9]{bs1978}.

The case of the unbounded time horizon problem has received some attention already in \cite[Section 3.3]{SL10} and \cite[Section V]{AAPLS06b}.
There it was suggested to use the convergence of the bounded-horizon values to the unbounded-horizon one,
which led to considering the fixpoint equations.
Although we have shown in Theorem \ref{thm:reach.DP} that fixpoint equations are indeed valid,
they can not be obtained using limiting arguments as the latter may fail as shown in Example \ref{ex:reach.fin.inf}.
An alternative approach via
a hitting time formulation \eqref{eq:reach.cost.stop} has been proposed in \cite{CCL11},
and the fixpoint equation for the maximal constrained reachability has been obtained in \cite[Theorem 2.10 (i)]{CCL11}.
However, one of the assumptions of this theorem required the first hitting time of the complement of the safe set $\tau_{S^c}$ to be almost surely finite for any Markov policy.
As a result, in the case of the unconstrained reachability this theorem assumes that the value is fixed and constant.
Finally, \cite[Theorem 2]{KamgarpourSL2013} has shown the convergence of the {\it maximal} bounded-horizon unconstrained reachability to the unbounded-horizon one,
and has showed that the latter satisfies the fixpoint equation.
In contrast to the aforementioned contributions,
Theorem \ref{thm:reach.DP} does not pose any limitations and establishes fixpoint equations for both the maximization and the minimization problems in generality,
without for example requiring any continuity assumptions that are often imposed otherwise
(cf. \cite[Assumption 1]{KamgarpourSL2013} or \cite[Assumption 2.9]{CCL11}).
In addition, Proposition \ref{prop:reach.fin2inf} provides a complete characterization of the convergence of bounded-horizon problems to the unbounded-horizon ones,
and is further supported by Example \ref{ex:reach.fin.inf}.


The approximation of the unbounded-horizon reachability problem with bounded-horizon counterparts is an extension to the controlled case of the result in \cite{ta2014}.
This extension requires no additional assumptions and (weak) continuity of the kernel $\T$ is sufficient to establish important results such as Theorems \ref{thm:absor.equiv} and \ref{thm:reach.decomp}.
At the same time,
in the proofs we have extensively used continuity assumption,
and so the equivalence in Theorem \ref{thm:absor.equiv} may fail to hold without such assumptions --
see e.g. \cite[Appendix]{ta2014}.
In particular, we acknowledge that \cite[Proposition 2]{tmka2013} is not correct:
although uniqueness of fixpoint indeed yields trivial constant solutions for the maximal and minimal unconstrained reachability in the general case,
without continuity assumptions it may happen that the solution is trivial but yet there are multiple fixpoints.
In emphasizing the role of absorbing sets,
it is crucial to use the connection between $m(S)$ and the contractivity of powers of the operators ${_S}\R^*$ and ${_S}\R_*$ in Theorem \ref{thm:reach.contr}.
In particular, as a special case we obtain \cite[Proposition 1]{KamgarpourSL2013},
which has obtained conditions for the contractivity in the special case $m(S)=1$.
The characterization of the absorbing sets,
as well as finding an appropriate $\mu$-excessive function,
is an interesting and important problem.
For example,
there seems to be a connection between weakly absorbing sets (such as $S_\infty$) and maximal controlled invariant sets in non-stochastic systems \cite{rmt2013}.
Another related concept is that of the maximal end component (MEC) \cite[Section 10.6]{bk2008},
which is used to solve both the reachability and the repeated reachability problems in the case of finite-state \cdt-MP.
Such techniques are extremely powerful and allow for the full solution of those problems,
but unfortunately the discrete structure of the finite state and control spaces is crucial,
and most of the nice properties MEC has are lost in the more general case of uncountable state spaces.

An alternative approach to the computation of the unbounded-horizon maximal reachability is in \cite[Proposition 3]{KamgarpourSL2013},
where it is proposed to recast the original fixpoint equation as a linear constrained optimization over the infinite-dimensional space $\bu(X)$,
and to apply numerical methods for its solution.
However, the uniqueness of the solution of this problem has not been addressed yet.
Other possible alternatives are the theory of Poisson's equations \cite[Chapter 7]{hll1999} and the theory of transient \cdt-MP \cite[Section 9.6]{hll1999},
both of which should be applied over the truncated operator ${_S}\T^*$.
Another interesting way to approach this problem is it impose the $\psi$-irreducibility on the model and to tailor the results in \cite[Chapter 10]{fs2002} developed for the AC performance criterion.
All those extensions,
however,
are out of the scope of the present contribution.

\section{Repeated reachability}
\label{sec:pers}

\subsection{Repeated reachability: characterization}

It follows from Theorem \ref{thm:aut.equiv} that model-checking a \cdt-MP against any property expressed as a DRA can be reduced to solving the Rabin-like conditions $\square\lozenge F' \wedge (\neg \square\lozenge F'')$ over the composition of the \cdt-MP with the underlying transition system of the DRA.
This result applies in particular to all $\omega$-regular languages and LTL formulae.
Unfortunately, we cannot provide a theory that is as comprehensive as for the reachability case (namely, for DFA or safe LTL specifications),
as it has been presented in Section \ref{sec:reach},
and only focus on some partial results.
In particular, we focus only on the case of the B\"uchi acceptance condition $\square\lozenge F$,
which is also easier to characterize by means of its dual property $\lozenge\square S$,
known as \emph{persistence}.
As mentioned in Section \ref{ssec:aut.comments},
we show how results developed in the setting of gambling theory apply to the \cdt-MP case discussed here.
Neither the repeated reachability problem nor its dual admit useful bounded-horizon counterparts,
so below we omit the symbol $\infty$ in $\lozenge$ and $\square$.

Given a \cdt-MP $\D = (X,U,\K,\T)$,
let $S\in \b(X)$ be the set of goal states.
A gambling analogue of $\D$ is given by $\G = (X,\Gamma)$, where the gambling house defined by
\begin{equation*}
  \Gamma := \proj_{X\times \p(X)}\left(\gr(\T) \cap (\K\times \p(X))\right)
\end{equation*}
is an analytical subset of $X\times \p(X)$ \cite[Section 7.6]{bs1978}.
It further follows from the equivalence between the \cdt-MP and gambling models \cite{b1976a},
that we can now invoke results in \cite{mps1991,ms1996a} to characterize the value of the repeated reachability problem.
In accordance with the mentioned work we call a function $f\in \bu(X)$ \emph{excessive}\footnote{
  Note that excessive functions are similar to locally uniformly excessive ones as per Definition \ref{def:loc.exc}.
} if $\T^*f\leq f$, \emph{deficient} if $(-f)$ is excessive,
and \emph{invariant} if its both deficient and excessive.
Clearly, invariant functions are precisely the fixpoints of the operator $\T^*$.

The next result provides a characterization of the maximal persistence probability $\M^*(\cdot;\lozenge\square S)$ and emphasizes its connection with the maximal safety probability $\M^*(\cdot;\square S)$.

\begin{theorem}\label{thm:persist.DP}
  For any set $S\in \b(X)$ it holds that $\M^*(\cdot;\lozenge\square S)\in \ba^*(X)$.
  It is also an invariant function,
  and for any excessive function $f\in \ba(X)$ satisfying the inequality $f(\cdot)\geq \T^*\left[\M^*(\cdot;\square S)\right]$ it holds that $f(\cdot)\geq \M^*(\cdot;\lozenge\square S)$.
  Moreover, the following DP-like recursions hold true:
  \begin{equation}\label{eq:persist.DP}
    \M^*(x;\lozenge\square S) = \lim_{n\to\infty}(\T^*)^n \M^*(x;\square S),
  \end{equation}
  where the limit is non-increasing point-wise,
  for all $x\in X$.
\end{theorem}

\begin{proof}
  The result follows immediately from the equivalence of the \cdt-MP and the gambling models \cite{ta2013},
  where in the latter setting the statement of the theorem is implied by \cite[Theorem 1.2]{mps1991} and \cite[Theorem 4.5, Corollary 5.5]{ms1996a}.
\end{proof}

Note that Theorem \ref{thm:persist.DP} connects the maximal safety probability $\M^*(\cdot;\square S)$ and the maximal persistence probability $\M^*(\cdot;\lozenge\square S)$.
As a result, we can use results on the former function obtained in Section \ref{sec:reach} to derive properties of the latter one.

\begin{proposition}\label{prop:persist.triv}
  For any $S\in \b(X)$: $\M^*(\cdot;\square S) = 0$ if and only if $\M^*(\cdot;\lozenge\square S) = 0$.
\end{proposition}

\begin{proof}
  Note that \eqref{eq:persist.DP} immediately implies that $\M^*(\cdot;\lozenge\square S) = 0$ is sufficient to claim that $\M^*(\cdot;\lozenge\square S) = 0$.
  On the other hand, since $\M^*(\cdot;\lozenge\square S)\geq \M^*(\cdot;\square S)$,
  thanks to Theorem \ref{thm:persist.DP} we obtain the converse implication.
\end{proof}

\subsection{Repeated reachability: computation}

Although the recursions in \eqref{eq:persist.DP} already suggest a possible computational procedure for computing the value of the maximal probability of persistence $\M^*(\cdot;\lozenge\square S)$,
the scheme requires an infinite number of iterations that are initialized at the maximal safety probability $\M^*(\cdot;\square S)$,
which in turn has to be computed in advance.
For the latter quantity we have already discussed non-trivial issues in Section \ref{sec:reach},
so the result of Theorem \ref{thm:persist.DP} is not in general practically applicable.
Instead, we propose tailoring the technique developed in Theorem \ref{thm:reach.decomp} to the problem at hand.

\begin{theorem}\label{thm:persist.decomp}
  Let Assumption \ref{as:cont.comp} hold true and further assume stationary policies are sufficient,
  that is for all $x\in X$ assume that
  \begin{equation*}
    \M^*\left(x;\lozenge\square S\right) = \sup_{\pi\in \Pi_S}\M^\pi\left(x;\lozenge\square S\right).
  \end{equation*}
  Suppose that $g$ is a locally $\pi'_0$-excessive function on $S$ for some stationary policy $\pi'\in \Pi_S$.
  Let $E\in \b(X)$ be any open set such that $\inf_{x\in E}\M^*(x;\lozenge\square S) = 0$, $E\cap \{g\leq 1\} = \emptyset$,
  $E^c$ is a compact set,
  and $(E^c)_\infty = S_\infty$.
  Then for all $x\in X$ and $\ve\in (0,1]$ it holds that
  \begin{equation}\label{eq:thm.persist.decomp}
    \left|\M^*\left(x;\lozenge\square S\right) - \M^*\left(x;A_\ve\U B_\ve\right)\right|\leq\max\left(\ve,\sup_{x\in E}\M^*(x;\lozenge\square S)\right),
  \end{equation}
  where $B_\ve := \{g\leq \ve\}$ and $A_\ve = (G_\ve\cup E)^c$.
\end{theorem}

\begin{proof}
  For any fixed stationary policy we are in the setting of \cite[Theorem 5]{ta2012b},
  so we are only left with applying Lemma \ref{lem:sup.inf.bounds}.
\end{proof}

Note that provided the requirements of Theorem \ref{thm:persist.decomp} are met,
it is possible to evaluate $\M^*\left(x;\lozenge\square S\right)$ with precise error bounds.
Indeed, in such case the set $A_\ve$ is compact and simple,
hence by Theorem \ref{thm:absor.equiv} we obtain that $m(A_\ve)<\infty$ and thus the maximal constrained reachability probability $\M^*\left(x;A_\ve\U B_\ve\right)$ can be approximated by the bounded-horizon probabilities.
Moreover, the set $E$ here has to be understood as a set where the probability of interest $\M^*\left(x;\lozenge\square S\right)$ is very small,
so if one is able to tune $E$,
then the right-hand side in \eqref{eq:thm.persist.decomp} can match any given precision.
Clearly, the assumptions in Theorem \ref{thm:persist.decomp} are rather restrictive,
and apply only to systems for which the set $S$ serves as a sort of dynamical attractor.

\subsection{Comments on the repeated reachability problem}

We have mentioned that the characterization in Theorem \ref{thm:persist.DP} is taken from the literature on gambling:
indeed we have not been able to find similar results obtained for the \cdt-MP framework.
It is interesting to see that the function $\M^*\left(x;\lozenge\square S\right)$ satisfies a fixpoint equation,
similarly to the uncontrolled case \cite{ta2012b}.
The connection between the solution of this problem
and the value of the maximal safety probability $\M^*\left(x;\square S\right)$ appears to be useful in characterizing simple instances,
as we have encountered in Proposition \ref{prop:persist.triv}.

There is range of literature in gambling on utilities with the form $J:=\limsup_{n\to\infty} c(\xx_n)$ and $J:=\liminf_{n\to\infty} c(\xx_n)$,
which turn out to be repeated reachability specifications in the case the cost is an indicator function,
namely $c(x) = 1_S(x)$.
For the $\limsup$ criterion, conditions on sufficiency of stationary policies have been obtained in \cite{s1969} and \cite{h1979},
while for the $\liminf$ case in \cite{s1983}.
A number of results valid for these criteria are summarized in \cite[Section 4]{ms1996},
in particular \cite[Theorem 9.1, Chapter 4]{ms1996}  provides a procedure to find $\M^*\left(x;\square S\right)$ using the transfinite induction algorithm over all countable ordinals,
rather than a simple recursion like in \eqref{eq:persist.DP}.
Although this book only focuses on the case when the state space is countable,
some of those results seem to allow for extensions to general Borel state spaces -- more research is needed towards this goal.
Unfortunately however,
they do not seem to lead to practical computational procedures.
To the best of our knowledge the result of Theorem \ref{thm:persist.decomp} is novel,
and is an extension of a version for uncontrolled processes in \cite{ta2012b},
where the focus was on studying the stability properties of the absorbing sets.
Alternatively, it may be worth invoking some results obtained for recurrence \cite{mt1993}:
however, such results are only strong when obtained under assumption of $\psi$-irreducibility of the transition kernel $\T$ \cite[Chapter 10]{fs2002},
which are often restrictive
and lead to results that are rarely computational.
The AC criterion also seems to be related to the $\limsup$ and $\liminf$ criteria in general,
and to the repeated reachability property in particular,
however much more research is needed to formally clarify the precise relationship.
To summarize,
on the one hand there are many results in gambling related to the repeated reachability problem,
however they do not seem to lead to practically useful computational methods.
On the other hand, in the \cdt-MP setting such criteria have not received much attention,
and although some related methods for other criteria \cite{fs2002} may be useful,
such relationship is by no means direct or clear.
The current contribution only makes an initial step towards numerical procedures for repeated reachability properties over \cdt-MP,
and much more research on the topic is needed.

\section{Case study}
\label{sec:cs}

In this section the theory developed above is applied to the example presented in Section \ref{ssec:power-net},
dealing with the control of a power network model.
The parameters are chosen as follows:
the upper bound for the energy is $M = 2$ and the reserve rate is $c = 0.93$.
The consumption of the power plant is assumed to be deterministic with a fixed $p = 0.7$,
and the minimal load is fixed to be $v_{\min} = 0.8$.
The renewable generators are assumed to produce power following a truncated Gaussian distribution with parameters $\mu = 0.1, \sigma = 0.03$ for the first subnetwork,
and $\mu = 0.05, \sigma = 0.01$ for the second one,
both with the support on $[0,2]$.
Similarly, the energy demand follows the same type of distribution,
with parameters $\mu = 0.2, \sigma = 0.05$ for the first subnetwork,
and $\mu =0.4, \sigma = 0.07$ for the second one.
As a result of this choice of parameters, in practice in the first subnetwork there is less power demand and the renewable generation is more substantial.
It is thus expected that the share of the nuclear power plant energy will be higher for the second network:
below this intuition is compared with the outputs of the numerical computations.

Let us first resort to the qualitative analysis of the two tasks formulated as automata specifications on Figures \ref{fig:task1} and \ref{fig:task2.dfa}.
We start by noticing that the cdt-MP we are dealing with is continuous as per Definition \ref{def:cdt-MP},
so that in particular Theorem \ref{thm:absor.equiv} can be applied.

With focus on the safety problem (first specification),
let the safe set be the square $S = [0.2,1.5]^2 \subset [0,M]^2$.
Clearly, this set is simple in the sense of Definition \ref{def:abs},
thus by Theorem \ref{thm:absor.equiv} the maximal safety probability over the infinite time horizon is equal to $0$ over this set.
As a consequence, let us know consider a finite horizon $n = 100$ to perform the corresponding computations.
The results are presented in Appendix \ref{app:b}.
The value function is depicted on Figure \ref{fig:cs.safety.value}:
one can see that even over a relatively long horizon of $100$ steps, the safety probability remains equal to $1$ over most of the safety set $S$.
Even though the iterations for the safety value function eventually converge to $0$ for the infinite time horizon problem,
such a convergence is clearly slow.
Regarding the optimal policy,
we have selected the one at step $n/2 = 50$ as a representative,
of which one can see on Figure \ref{fig:cs.safety.u} its $u^1$-component,
namely the fraction of the nuclear plant energy used for the first subnetwork.
In particular, whenever the energy level in first subnetwork is low whereas the one in the second high,
$u^1 = 1$ which confirms an intuition that in such situation all the nuclear power has to be used to maintain the first subnetwork.
Conversely, $u^1 = 0$ meaning $u^2 = 1$ over the set where the energy level is high in the first subnetwork and low in the second.
In addition, the set $\{u^1 = 0\}$ is larger than $\{u^1 = 1\}$ confirming our intuition that the second network is more fragile (can rely less on renewable production) and thus requires more energy from the nuclear plant.
Finally, Figure \ref{fig:cs.safety.v} which presents the $v$-component of the policy (total nuclear energy production),
and provides a justification for the intuitive idea that for low (high) energy levels $v$ is necessarily high (low).

For the reach-avoid task expressed as the DFA on Figure \ref{fig:task2.dfa} we can provide a similar analysis.
Here we choose the safe set to be $S := [0.2,1.8]^2$ and the goal sets $G_1 := (1.8,2]\times [0.2,1.8]$,
$G_2 := [0.2,1.8]\times (1.8,2]$,
and finally $G:= (1.8,2]^2$.
Again, due to simplicity of set $S$ we obtain that the finite-horizon computations converge exponentially fast to the infinite-horizon value by Theorems \ref{thm:reach.contr} and \ref{thm:aut.equiv}.
In view of this,
as in the case of safety we compute the value function for the finite horizon $n = 100$.
The results are presented in Appendix \ref{app:b}.
The value function on Figure \ref{fig:cs.dfa.value} has some intuitive properties:
it is equal to $1$ on the goal set $G$,
it is equal $0$ over the unsafe set,
and is positive elsewhere.
The optimal choice of the $v$-component of the policy is always $v = 1$, as the goal is to maximize the energy level in the two subnetworks:
due to this reason,
we are not presenting the trivial plots for the component $v$.
The behaviour of the $u^1$-component is instead more interesting:
we present it at the time step $n/2 = 50$ on Figure \ref{fig:cs.dfa.u}.
One can see that over the safe set,
when the energy level in the first subnetwork is high and in the second is low (close to the set $G_1$),
the controller increases the energy level in the first subnetwork ($u^1 \approx 1$).
At the same time, after reaching the set $G_1$ the controller pursues the new goal of maximizing the energy level in the second subnetwork and thus keeps $u^2 \in [0,0.3]$.
Symmetrically, a converse situation holds close to and over the set $G_2$.
Note that on Figure \ref{fig:cs.dfa.u} the value of $-1$ for the policy represents the points where the value does not depend on the control action chosen,
which is the case over the goal and the unsafe set.




\section{Conclusions}
\label{sec:concl}

This paper has considered an optimal control synthesis problem,
where the probability of a given event is maximized or minimized over a
controlled discrete-time Markov process (\cdt-MP) model.
Using methods from formal languages and automata theory we have proposed a characterization of events of interest using formulae in linear temporal modal logic (LTL) and derived from deterministic automata.
Furthermore, we have extended results known for finite-state \cdt-MP to general state-space models,
and have showed that the original optimal control problems can be reduced to either of two fundamental ones: reachability or repeated reachability.
For the former problem, we have provided a full characterization of the dynamic programming (DP) algorithm,
and developed a theory of approximation for the unbounded-time problem using computable bounded-horizon counterparts.
More restrictive results have been attained
for the repeated reachability problem:
we have provided a partial characterization and proposed a computational technique that can be useful for a class of stable models.
We have further discussed some questions and issues related to the repeated reachability problem:
providing a complete answer to them is a promising direction for future research.

\section*{Acknowledgements}

The authors are grateful to Bill Sudderth for helpful discussions on gambling theory.
This work has been supported by the European Commission STREP project MoVeS 257005,
by the European Commission Marie Curie grant MANTRAS 249295,
by the European Commission IAPP project AMBI 324432,
and by the NWO VENI grant 016.103.020.

\bibliographystyle{amsalpha}
\bibliography{cdtMP.bbl}

\newpage

\appendix
\section{}
\renewcommand{\thesection}{A}

\subsection{Background from analysis}

The sufficient mathematical background for this paper can be consulted as follows:
measure theory \cite[Chapters 1-3]{f1999},
topology\footnote{
  For the readers with a background in computer science \cite[Section 2, Chapter III]{PerrinP2004} may serve as an alternative reference for the introduction to topology.
} \cite[Chapter 4]{f1999} and basic probability theory \cite[Chapter 10]{f1999}.
Below we summarize facts about Borel spaces that are extensively used in the manuscript:
most of those we need are covered together with proofs in \cite[Chapter 7]{bs1978},
whereas a more detailed exposition is given \cite{p1967} and \cite{s1998a}.

Given an arbitrary set $X$ we denote by $2^X$ its powerset,
that is the collection of all subsets of $X$.
A complement of $A\subseteq X$ is denoted by $A^c = X\setminus A$.
A class $\f\subseteq 2^X$ is called an algebra if it contains the empty set and is closed under finite unions and taking the complement.
An algebra $\f$ is called a $\sigma$-algebra if in addition it is closed under countable unions.
For any class of sets $\c\subseteq 2^X$ we denote by $\sigma(\c)$ the smallest $\sigma$-algebra that contains $\c$;
in that case we say that $\sigma(\c)$ is generated by $\c$.
In particular, if $X$ is given a topology then $\b(X)$ denotes the Borel $\sigma$-algebra of $X$:
the one generated by the class of all open subsets of $X$.
Elements of $\b(X)$ are sometimes referred to as Borel sets.
Any topological space by default is assumed to be endowed with its Borel $\sigma$-algebra.
A topological space $X$ is said to be a (standard) Borel space if it is homeomorphic to a Borel subset of a complete separable metric space.
As an example, the set of real numbers $\R$ here is always assumed to endowed with a usual Euclidian topology,
so that $\R$ is a Borel space;
all subsets of $\R$ are assumed to be given their inherited subset topologies.
As another example,
any countable (finite or infinite) set is assumed to be endowed with the discrete topology,
which makes it a Borel space.

The set of natural numbers is denoted by $\N$,
and we further write $\N_0 := \N \cup \{0\}$ and $\bar\N_0 := \N_0\cup \{\infty\}$.
When dealing with $\infty$ we adopt the following convention: $\infty + 1 = \infty$.
For any set $X$ we use the notation $X^\omega$ instead of $X^{\N_0}$.
If $Y$ is any other set,
we further use the following shorthand notation:
$X^\omega \parallel Y^\omega := (X\times Y)^\omega$.
Moreover, if $D\subseteq Y^\omega$ we further denote $X\parallel D := \proj_{Y^\omega}^{-1}(D)$ where $\proj_{Y^\omega}:X^\omega\parallel Y^\omega \to Y^\omega$ is an obvious projection map.
The latter notation also extends to maps:
if $Z$ is some other set and $f:Z\to X^\omega$,
$g:Z\to Y^\omega$ are some maps,
then $h:= f\parallel g:Z\to X^\omega\parallel Y^\omega$ is the unique map such that $\proj_{X^\omega}\circ h = f$ and such that $\proj_{Y^\omega}\circ h = g$.
For any set $X$ the identity map on $X$ is given by $\id_X(x) = x$ for all $x\in X$.
For $f:X\to Y$ its graph is denoted by
\begin{equation*}
  \gr(f):=\{(x,f(x)):x\in X\}\subseteq X\times Y.
\end{equation*}

All Cartesian products of topological spaces are assumed to be endowed with the corresponding product topologies.
In particular, if $(X_k)_{k\in \N}$ is a collection of Borel spaces,
and $I\subseteq \N$ then $\b(\prod_{k\in I}X_k) = \bigotimes_{k\in I}\b(X_k)$,
i.e. a Borel $\sigma$-algebra of a countable product of Borel spaces coincides with the product of their Borel $\sigma$-algebras.

Given two measurable spaces $(X,\x)$ and $(Y,\y)$ the map $f:X\to Y$ is said to be measurable if $f^{-1}(\y) \subseteq \x$;
in that case we write $f\in \x/\y$.
If $(Y,\y) = (\R,\b(\R))$ we simplify the notation and write $f\in \x$ rather than $f\in \x/\b(\R)$.
If $\f$ is any class of functions $f:X\to \R$ we use $\mathrm b\f$ to denote a subclass of bounded functions in $\f$.
The class of all bounded functions $\mathrm b \R^X$ is assumed to be given a sup-metric $\rho(f,g):=\sup_{x\in X}|f(x) - g(x)|$ which is inherited to all its subclasses.
For any two functions $f,g\in \R^X$ we write $\{f\leq g\}:= \{x\in X:f(x)\leq g(x)\}$ and similarly for $\{f\geq g\}$ and $\{f = g\}$.
We further write $f\leq g$ if and only if $\{f\leq g\} = X$.
An important example of functions is given by an indicator function,
which for any $A\subseteq X$ is given by
\begin{equation*}
  1_A(x) :=
  \begin{cases}
    1,&\text{ if } x\in A,
    \\
    0,&\text{ if } x\notin A.
  \end{cases}
\end{equation*}

For any Borel space $X$ the collection of all probability measures on $(X,\b(X))$ is denoted by $\p(X)$.
We always assume the latter to be endowed with a topology of weak convergence,
which makes $\p(X)$ a Borel space as well and thus it can be given the Borel $\sigma$-algebra $\b(\p(X))$.
Given any $A\in \b(X)$ we define an evaluation map $e_A:\p(X) \to [0,1]$ as $e_A(p) := p(A)$ for any $p\in \p(X)$.
It appears that $\b(\p(X))$ is the smallest $\sigma$-algebra with respect to which all evaluation maps are measurable.
Given a probability measure $p\in \p(X)$ we denote the $p$-completion of $\b(X)$ by $\b_p(X)$.
The universal $\sigma$-algebra of a Borel space $X$ is defined as $\u(X) := \bigcap_{p\in \p(X)}\b_p(X)$.
For any $p\in \p(X)$ and $f\in \bu(X)$ we can define the Lebesgue integral $\int_X f\d p$ which we also write simply as $p[f]$.
If $Y$ is another Borel space,
and $g\in \u(X)/\b(X)$ then any probability measure $p\in \p(X)$ is pushed by $g$ to $g_* p\in \p(Y)$ where the pushforward of the measure is defined by $(g_* p)(A) = p(g^{-1}(A))$ for any $A\in \b(Y)$.

For any two sets $X$ and $Y$ the natural projection from their product onto $X$ is denoted by $\proj_X:X\times Y\to X$ viz. $\proj_X(x,y) = x$ for any $x\in X$ and $y\in Y$.
Furthermore, for any $D\subseteq X\times Y$ the $x$-section of $D$ is defined by
\begin{equation*}
  D_x := \{y\in Y:(x,y\in D)\}
\end{equation*}
for any $x\in X$.
If $X$ is a Borel space,
a set $A\subseteq X$ is said to be analytic if there exists $B\in \b(X\times \R)$ such that $A = \proj_X(B)$.
The collection of all analytic subsets of $X$ is denoted by $\s(X)$.
Although it contains the empty set and is closed under countable unions and intersections,
it is not closed under taking the complement,
so it is not a $\sigma$-algebra.
The analytical $\sigma$-algebra of $X$ is denoted by $\a(X) := \sigma(\s(X))$.
It further follows for any Borel space $X$ that
\begin{equation*}
  \b(X) \subseteq \s(X) \subseteq \a(X) \subseteq \u(X).
\end{equation*}
Given two Borel spaces $X$ and $Y$ we say that a map $f:X\to Y$ is Borel (analytically, universally) measurable if $f\in \b(X)/\b(Y)$ (if $f\in \a(X)/\b(Y)$, if $f\in \u(X)/\b(Y)$).
By a stochastic kernel we mean any map of the form $P:X\to \p(Y)$.
For any such kernel we write $P(A|x)$ for any $x\in X$ and $A\in \b(Y)$ instead of a more cumbersome version $P(x)(A)$.
Moreover, we write $P\in \u(Y|X)$ instead of $P\in \u(X)/\b(\p(Y))$ and similarly for $\a(Y|X)$ and $\b(Y|X)$.
It follows that $P\in \b(Y|X)$ ($\a(Y|X)$, $\u(Y|X)$) if and only if $P(A|\cdot)$ is a Borel (analytically, universally) measurable function for any $A\in \b(Y)$ \cite[Lemma 1.37]{k1997a}.
The Dirac probability measure at $x\in X$ is denoted by $\delta_x$.
Furthermore,
for any map $f:X\to Y$ we assign the correspondent kernel $\delta_f$ such that $\delta_f(x):=\delta_{\{f(x)\}}$.
It follows from \cite{cs2012} that $f$ is Borel (analytically, universally) measurable as a map if and only if $\delta_f$ is as a kernel.

A function $f:X\to \R$ is said to be lower semi-analytic if $\{f<c\}\in \s(X)$ for any $c\in R$,
and upper semi-analytic if $-f$ is lower semi-analytic.
The collection of all lower (upper) semi-analytic functions is denoted by $\a_*(X)$ ($\a^*(X)$).
A function $f:X\to \R$ is said to be lower semi-continuous if $\{f\leq c\}$ is closed in $X$ for any $c\in \R$,
and upper semi-continuous if $-f$ is lower semi-continuous.
The collection of all lower (upper) semi-continuous functions is denoted by $\c_*(X)$ ($\c^*(X)$).
The following hierarchy holds for the function classes:
\begin{equation*}
  \c^*(X),\; \c_*(X)\subseteq \b(X) \subseteq \a^*(X),\; \a_*(X) \subseteq \a(X) \subseteq \u(X).
\end{equation*}
A kernel $P\in \b(Y|X)$ is called continuous if $P:X\to \p(Y)$ is a continuous map.
Alternatively, the continuity of the kernel can be characterized as follows:
a kernel $P\in \b(Y|X)$ is continuous if and only if $\int_Y f\d P \in \bc^*(X)$ for any $f\in \bc^*(Y)$.

If $(X,\rho_X)$ and $(Y,\rho_Y)$ are metric space,
a map $f:X\to Y$ is called a contraction if there exists a constant $\beta\in [0,1)$ such that $\rho_Y(f(x'),f(x''))\leq \beta\cdot \rho_X(x',x'')$ for all points $x',x''\in X$.
The constant $\beta$ is also called a modulus of a contraction $f$.

\subsection{Important fragments of LTL}
\label{ssec:ltl.frag}

Although any LTL formula can be expressed as a DRA,
such generality is not very useful in practice.
Even when dealing with finite \cdt-MP $\D$,
expressing a given formula as a DFA $\A = (\TS,D)$ (if possible) may reduce the complexity of the automaton comparing to some DRA expressions of the formula,
as well as allows applying simpler solution methods,
which altogether leads to a smaller state space of the composition $\D\parallel \TS$ and hence to a lower computational time.
In the case when the \cdt-MP $\D$ is not finite,
in addition the solution methods are much more involved and as Sections \ref{sec:reach} and \ref{sec:pers} suggest,
solution of a bounded-horizon reachability problem simpler than the one of an unbounded horizon reachability,
which in turn is easier than the repeated reachability problem.
As a result,
e.g. although any LTL formula that encodes some bounded-horizon property can be expressed as a DRA,
it is worth analyzing the formula to check whether it allows for an automaton expression with a simpler acceptance condition.
In this section we describe how to perform such analysis,
and what are the useful fragments of LTL that allow for an expression via an automaton that is simpler than a DRA.

\medskip

The syntactically safe LTL (sLTL) \cite{kv1999} expresses safety languages.
The language $\phi\subseteq \Sigma^\omega$ is called a safety property if and only if any word $w\notin\phi$ has a finite ``bad'' prefix:
\begin{equation*}
  w\notin\phi \quad \iff \quad \exists n\in \N_0: \proj_{\Sigma^n}^{-1}\left(\proj_{\Sigma^n}(w)\right) \cap \phi = \emptyset.
\end{equation*}
The syntactically co-safe LTL (scLTL) \cite{kv1999} expresses co-safety languages,
where a co-safety language $\phi$ is the one for which any word $w\in\phi$ has a good prefix,
that is
\begin{equation*}
  w\in\phi \quad \iff \quad \exists n\in \N_0: \proj_{\Sigma^n}^{-1}\left(\proj_{\Sigma^n}(w)\right) \subseteq \phi.
\end{equation*}
Clearly $\phi$ is a safety language if and only if $\Sigma^\omega\setminus \phi$ is a co-safety one.
This comes as no surprise as safety languages are exactly closed subsets of $\Sigma^\omega$ in the product topology,
whereas co-safety languages are open \cite{as1985}.
It follows that any co-safety language can be expressed as a DFA,
and hence DFA can be used for negations of safety languages.
Here we only give a grammar of sLTL\footnote{
  The grammar of scLTL can be easily deduced from the one of sLTL;
  see also \cite[Definition 2.1]{aglb2012}.
}.
For this purpose,
in the LTL setting let us define a temporal modality \emph{Weak until} $\W^\infty$ by
\begin{equation*}
  \Phi_1 \W^\infty \Phi_2 := \Phi_1\U \Phi_2\vee \square \Phi_1.
\end{equation*}
The grammar of sLTL is given as follows:
\begin{equation*}
    \Phi \quad ::= \quad \sigma\in \Sigma\;|\;\neg \sigma\;|\;\Phi_1\wedge\Phi_2\;|\;\Phi_1\vee\Phi_2\;|\;\X \Phi\;|\; \Phi_1\W^\infty \Phi_2.
\end{equation*}
Note that in sLTL the negation can be only applied on the level of letters,
so that $\vee$ could not be expressed through $\wedge$ in general in sLTL in contrast to the LTL setting.
Moreover, in general it is not possible to express $\Phi_1 \U \Phi_2$ using sLTL grammar.
An example of an sLTL formula is $\square^n \sigma$,
and that of an csLTL formula are $\lozenge^n \sigma$ and $\sigma_1\U^n \sigma_2$where $n\in \N_0^\infty$ in all three cases.
One immediate way to see whether a given LTL formula belongs to sLTL is to write it in a negation normal form (NNF),
where the negation is presented on the level of atomic propositions by means of the following identities:
$\neg\X \Phi = \X(\neg\Phi)$,
$\neg (\Phi_1 \U \Phi_2) = \neg \Phi_1 \W^\infty \neg\Phi_2$ etc.
However, even a LTL formula corresponding to a safety language may lead to a NNF which does not belong to sLTL,
so for more elaborate methods see \cite{kv1999}.
Recent examples of applications of sLTL and of csLTL can be found in papers \cite{rmt2013} and \cite{aglb2012} respectively.

\medskip

Although sLTL and scLTL are related to the expression of formulae via DFA rather than DRA,
they still lead to the unbounded-horizon reachability problem over $\D\parallel \TS$,
even in case when the original formula encodes a bounded-horizon specification.
A useful framework to deal with the latter is given by the bounded LTL (BLTL) \cite{ta2013} which expresses bounded languages:
a language $\phi\subseteq \Sigma^\omega$ is called bounded if there exists $n\in \N_0$ such that
\begin{equation*}
  w\in \phi \quad \iff \quad \proj_{\Sigma^n}^{-1}\left(\proj_{\Sigma^n}(w)\right)\subseteq \phi.
\end{equation*}
In particular, it appears that bounded languages are exactly those that are both safety and co-safety languages \cite[Proposition 3.10, Chapter III]{PerrinP2004},
that is they are clopen subsets of $\Sigma^\omega$.
The grammar of BLTL is given as follows:
\begin{equation}
  \Phi \quad ::= \quad \sigma\in \Sigma\;|\;\neg \Phi\;|\;\Phi_1\wedge\Phi_2\;|\;\X \Phi
\end{equation}
so that it still allows for negations to be applied on all the levels,
but $\U$ is not absent.
On the other hand,
\eqref{eq:synt.b-until} implies that $\Phi_1\U^n\Phi_2$ belongs to BLTL for finite $n\in \N_0$.
It is likely that any BLTL formula allows to be expressed as a bounded-horizon version of the DFA \cite[Section 3.4]{ta2013} which accepts only those runs that visit the set of final states in at most $n$ steps,
where $n$ is specified a priori,
in the definition of the automaton.
For the applications of BLTL see e.g. \cite{jcllpz2009}.

\subsection{Auxiliary results}
\label{sec:app.aux.res}

\begin{lemma}\label{lem:equiv}
  Let $Y$, $Y'$ be arbitrary sets and let $g:Y\to\R$ and $g':Y'\to \R$ be some functions.
  Suppose that there exist maps $a:Y\to Y'$ and $a':Y'\to Y$ such that
  \begin{equation*}
    g(y) = g'(a(y)),\quad g'(y') = g(a'(y')), \quad \forall y\in Y,y'\in Y'
  \end{equation*}
  Then: $\inf_{y\in Y} g(y) = \inf_{y'\in Y'}g'(y')$ and $\sup_{y\in Y} g(y) = \sup_{y'\in Y'}g'(y')$.
\end{lemma}

\begin{proof}
  The following sequences of inequalities
  \begin{align*}
    \inf_{y\in Y} g(y) = \inf_{y\in Y} g'(a(y)) \geq \inf_{y'\in Y'}g'(y') = \inf_{y'\in Y'}g(a'(y'))\geq \inf_{y\in Y} g(y)
    \\
    \sup_{y\in Y} g(y) = \sup_{y\in Y} g'(a(y)) \leq \sup_{y'\in Y'}g'(y') = \sup_{y'\in Y'}g(a'(y'))\leq \sup_{y\in Y} g(y)
  \end{align*}
  yield the desired result.
\end{proof}

The next lemma shows that point-wise bounds also hold for the optimal values.

\begin{lemma}\label{lem:sup.inf.bounds}
  Let $Y$ be an arbitrary set and consider any two function $f,g:Y\to \R$.
  If $|f(y) - g(y)|\leq\ve$ for all $y\in Y$ then $|\sup_{y\in Y} f(y) - \sup_{y\in Y} g(y)|\leq \ve$.
\end{lemma}

\begin{proof}
  The proof is given in \cite[Appendix A.3]{hl1989}.
\end{proof}

\begin{lemma}\label{lem:semicont.ind}
  If $Y$ is a Borel space,
  the set $S$ is closed in $Y$ and the function $f\in \bc^*(X)$ is such that $f\geq 0$,
  then it holds that $1_S \cdot f\in \bc^*(X)$.
\end{lemma}

\begin{proof}
  Notice that for any $c \leq 0$ it holds that $\{1_S \cdot f\geq c\} = X$,
  whereas for $c>0$ we obtain $\{1_S \cdot f\geq c\} = S\cap \{f\geq c\}$ which is a closed set as well.
\end{proof}

\appendix
\renewcommand{\thesection}{B}
\section{}
\label{app:b}

\begin{figure}[h]
\subfigure[Safety value function plotted in 3d]{
   \includegraphics[keepaspectratio=true,width=12cm]{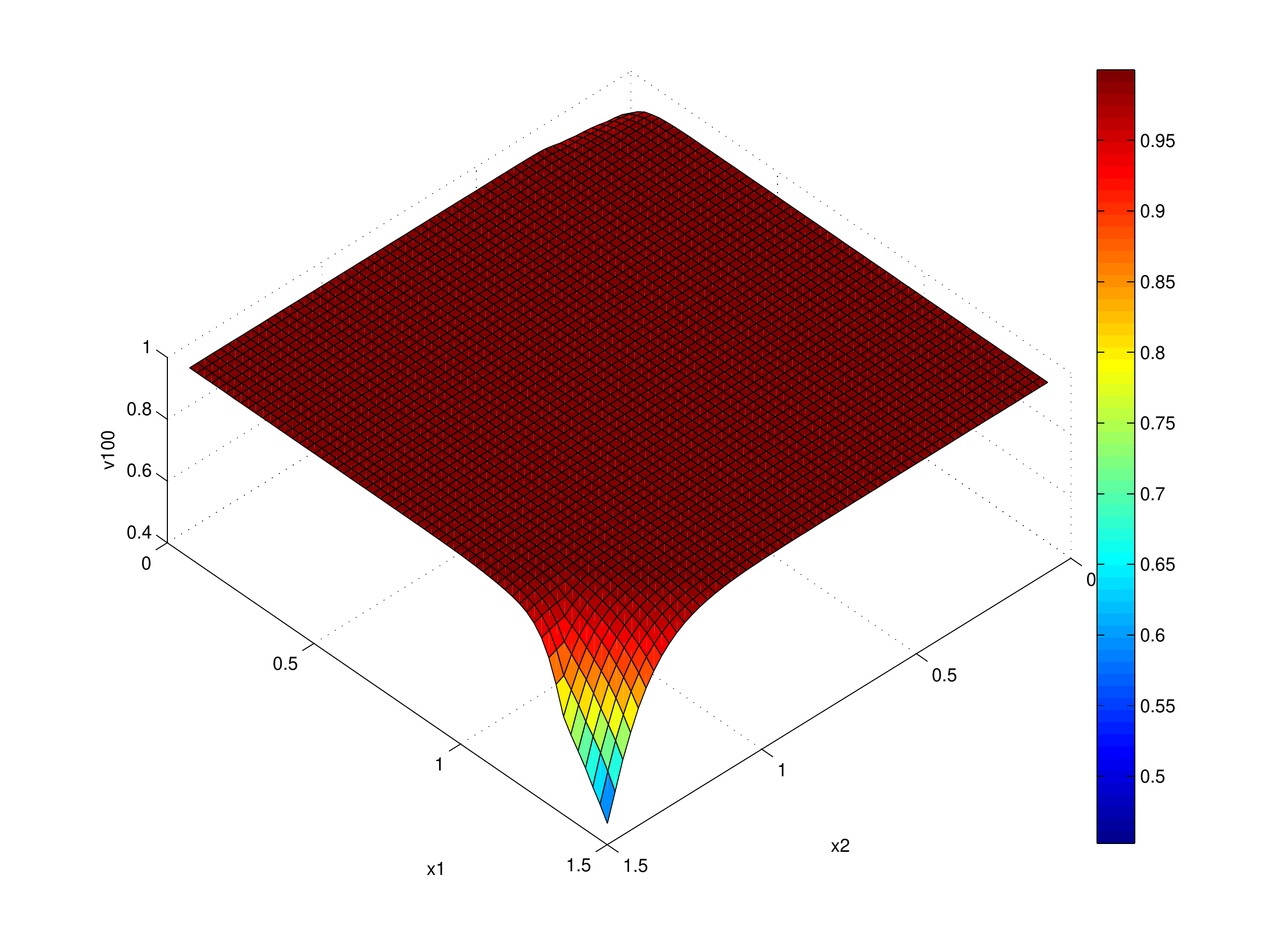}
}
\subfigure[Safety value function plotted in 2d]{
   \includegraphics[keepaspectratio=true,width=12cm]{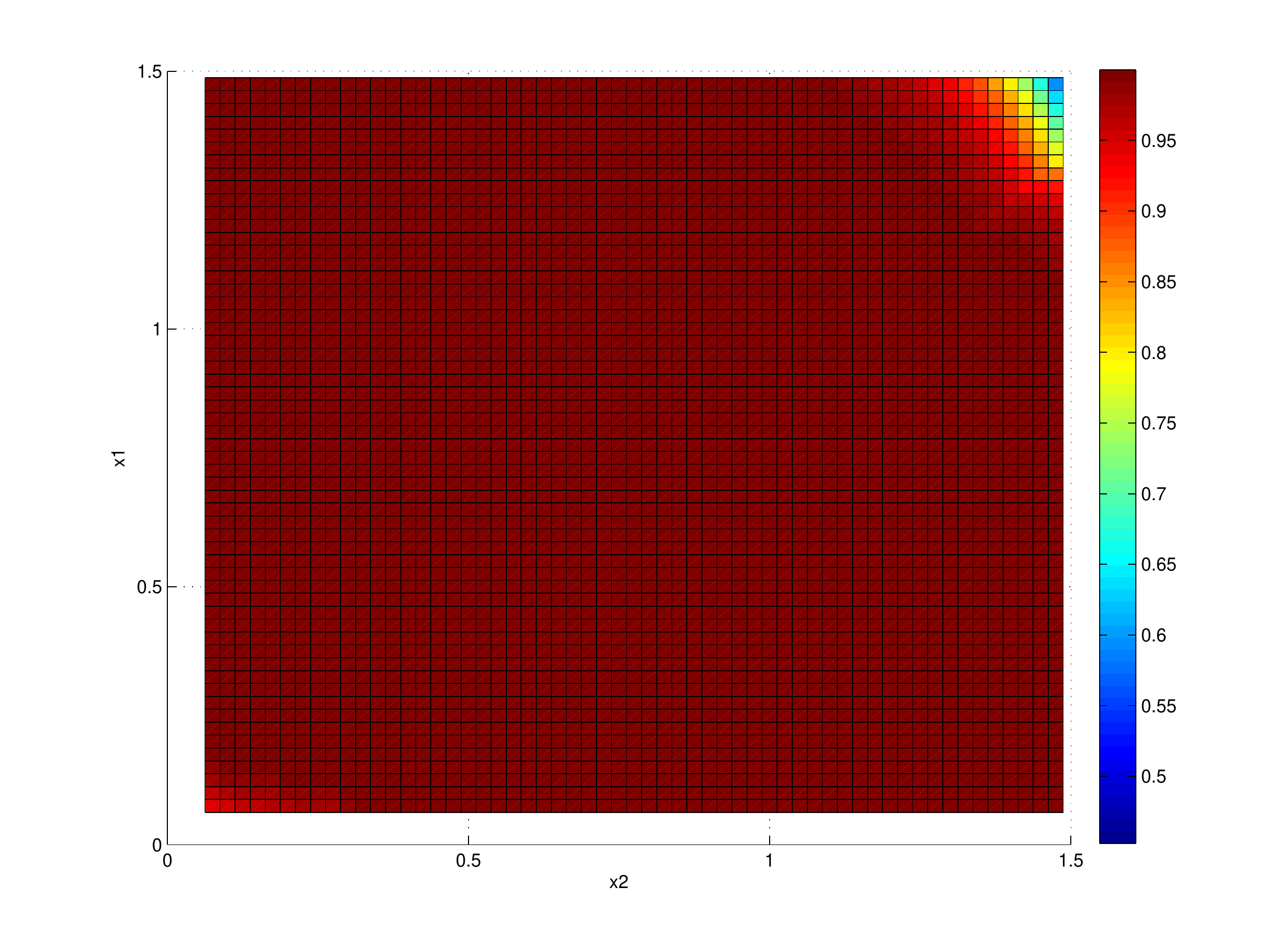}
}
\caption{Safety value function over a finite time horizon of $100$ steps.}
\label{fig:cs.safety.value}
\end{figure}

\newpage

\begin{figure}[h]
\subfigure[Optimal safety policy, component $u$ plotted in 3d]{
   \includegraphics[keepaspectratio=true,width=12cm]{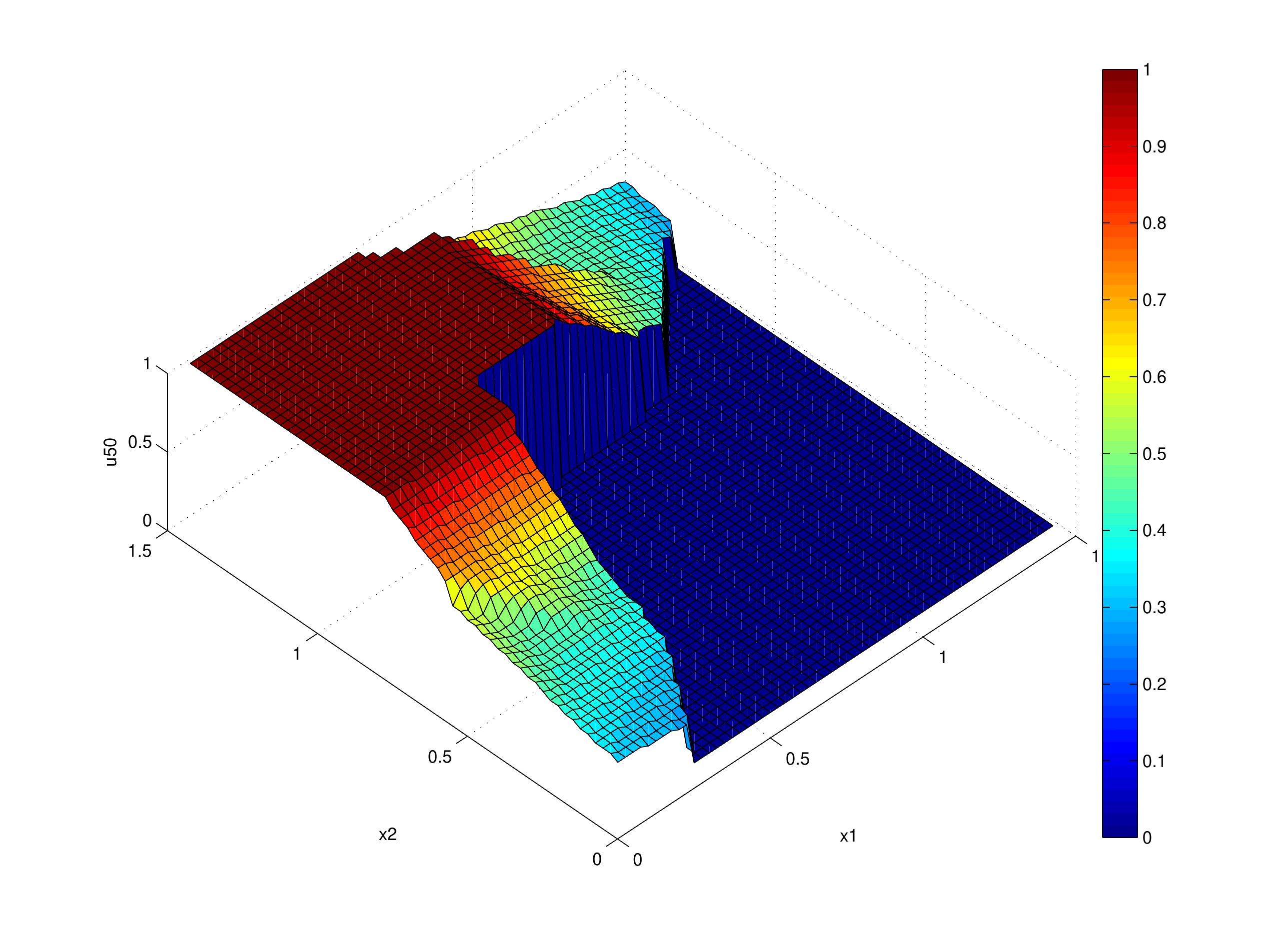}
}
\subfigure[Optimal safety policy, component $u$ plotted in 2d]{
   \includegraphics[keepaspectratio=true,width=12cm]{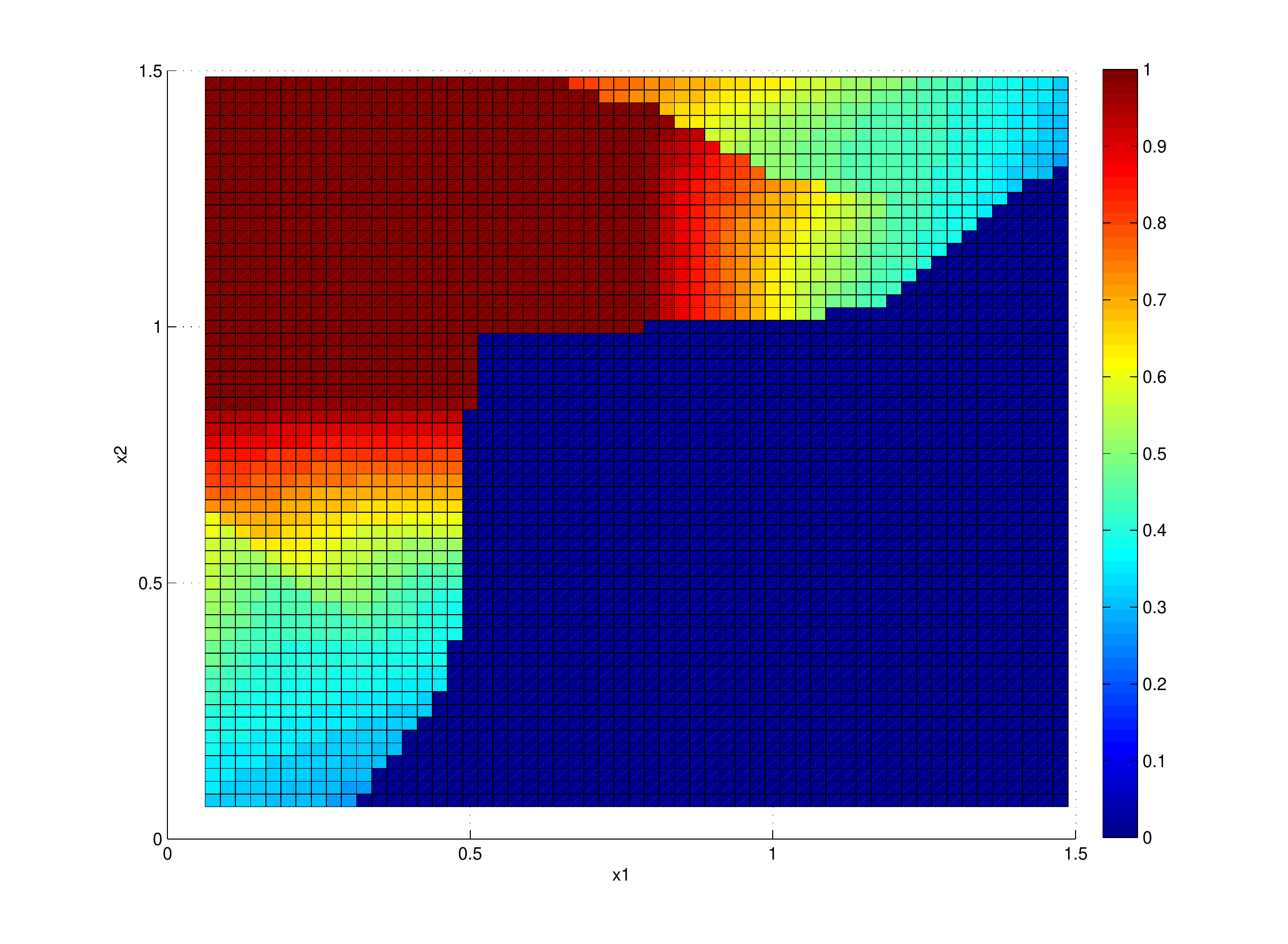}
}
\caption{Optimal safety policy, component $u$ at time step $50$.}
\label{fig:cs.safety.u}
\end{figure}

\newpage

\begin{figure}[h]
\subfigure[Optimal safety policy, component $v$ plotted in 3d]{
   \includegraphics[keepaspectratio=true,width=12cm]{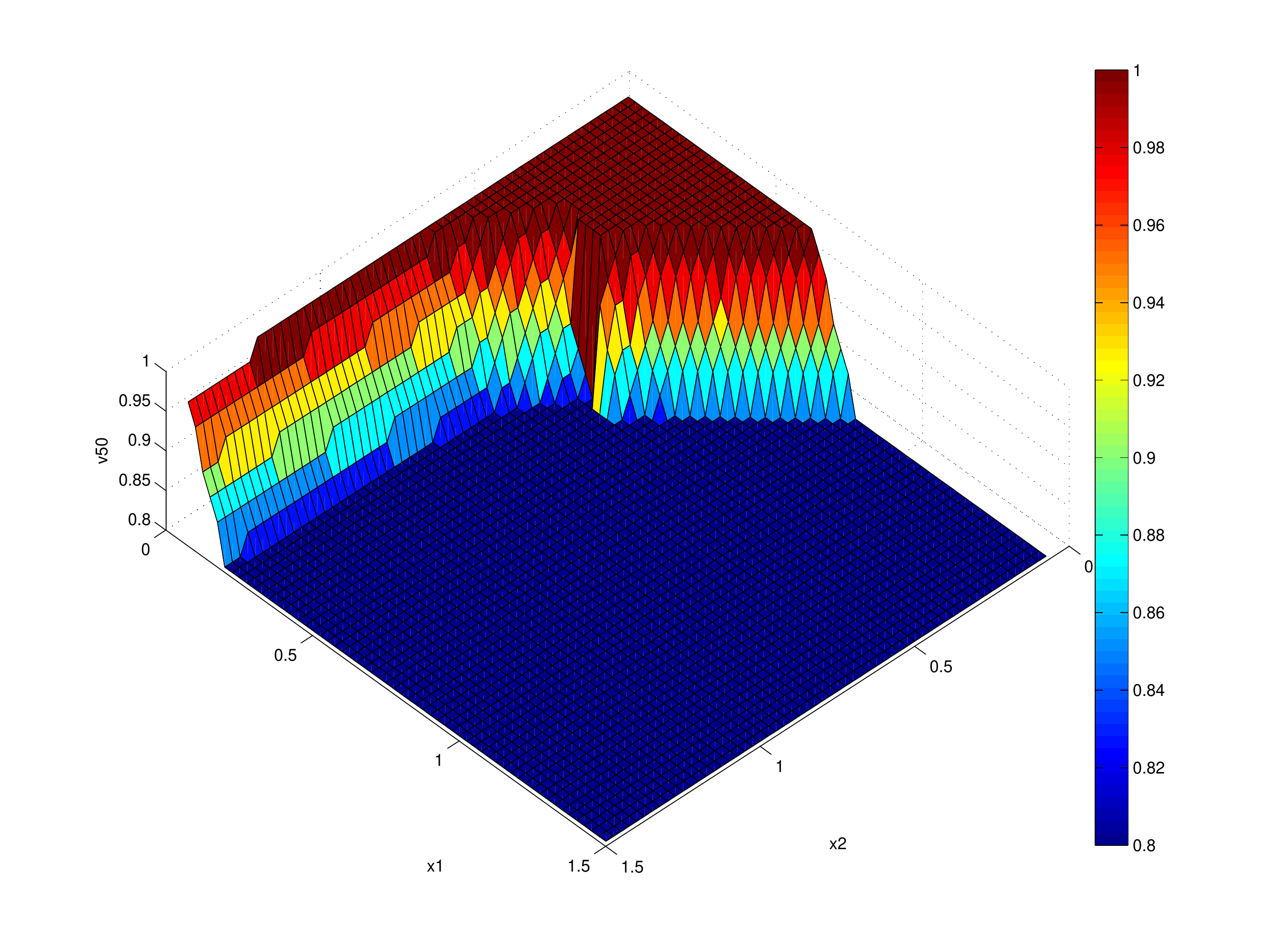}
}
\subfigure[Optimal safety policy, component $v$ plotted in 2d]{
   \includegraphics[keepaspectratio=true,width=12cm]{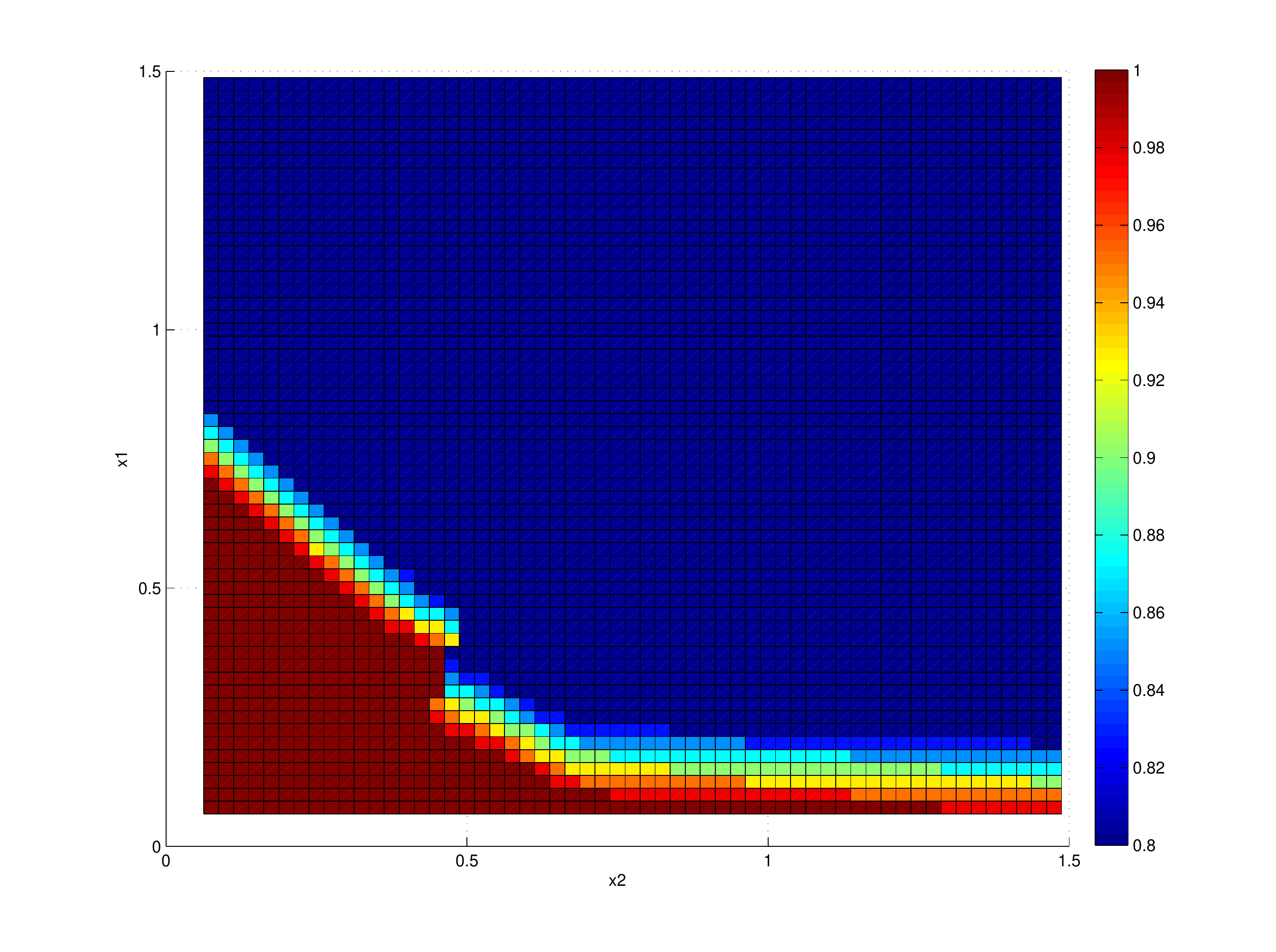}
}
\caption{Optimal safety policy, component $v$ at time step $50$.}
\label{fig:cs.safety.v}
\end{figure}

\begin{figure}[h]
\subfigure[DFA value function plotted in 3d]{
   \includegraphics[keepaspectratio=true,width=12cm]{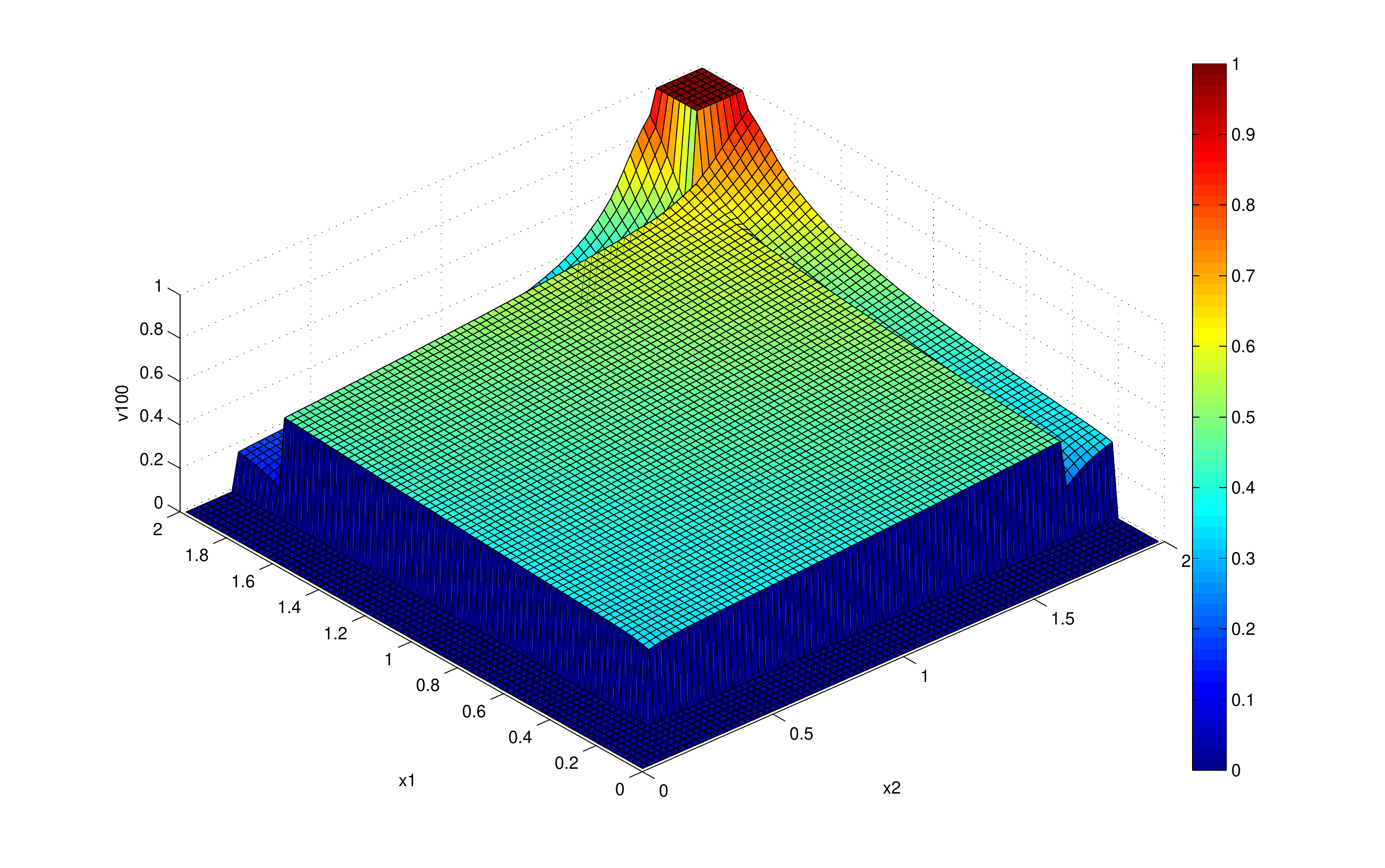}
}
\subfigure[DFA value function plotted in 2d]{
   \includegraphics[keepaspectratio=true,width=12cm]{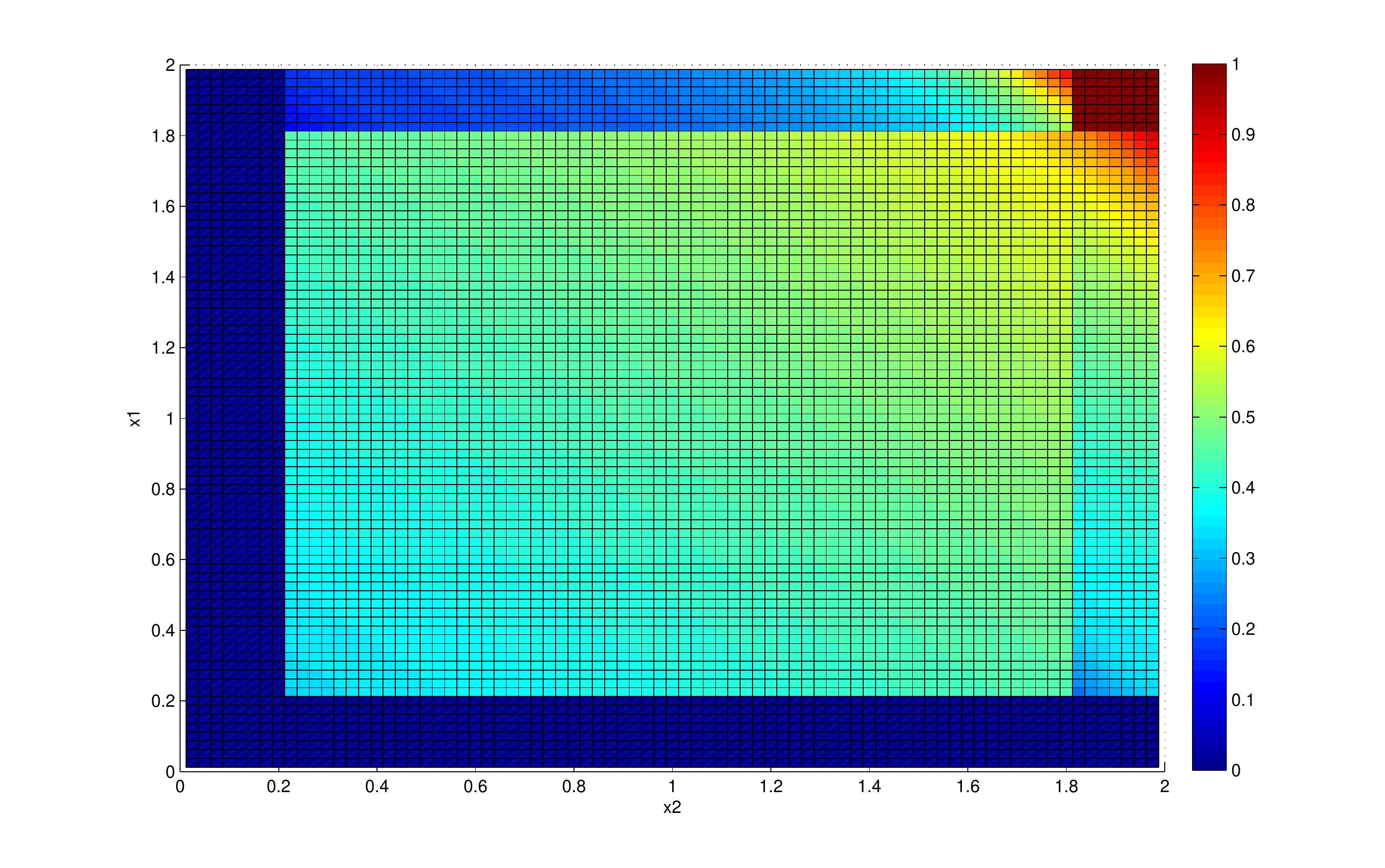}
}
\caption{DFA value function over a finite time horizon of $100$ steps.}
\label{fig:cs.dfa.value}
\end{figure}

\newpage

\begin{figure}[h]
\subfigure[Optimal DFA policy, component $u$ plotted in 3d]{
   \includegraphics[keepaspectratio=true,width=12cm]{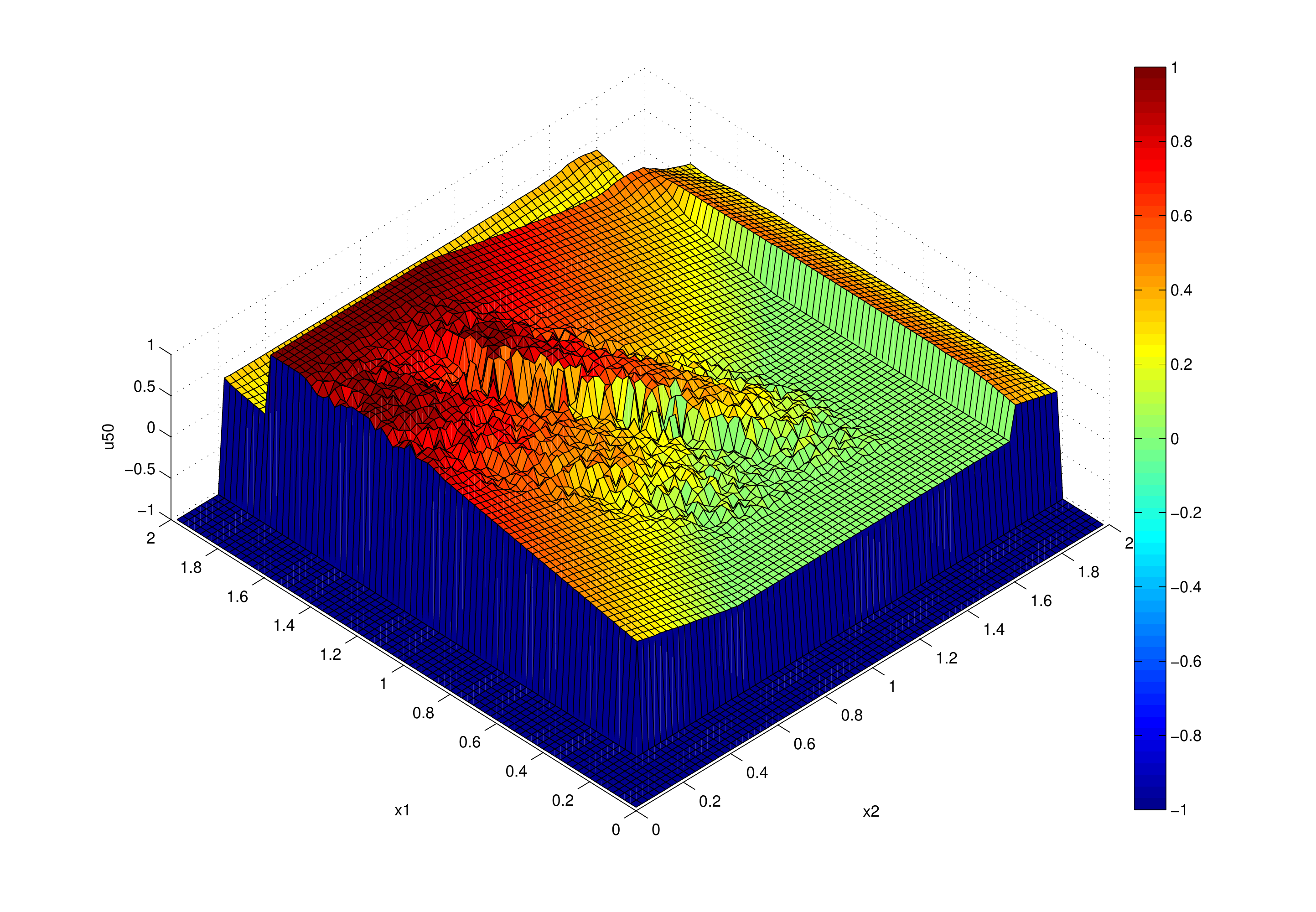}
}
\subfigure[Optimal DFA policy, component $u$ plotted in 2d]{
   \includegraphics[keepaspectratio=true,width=12cm]{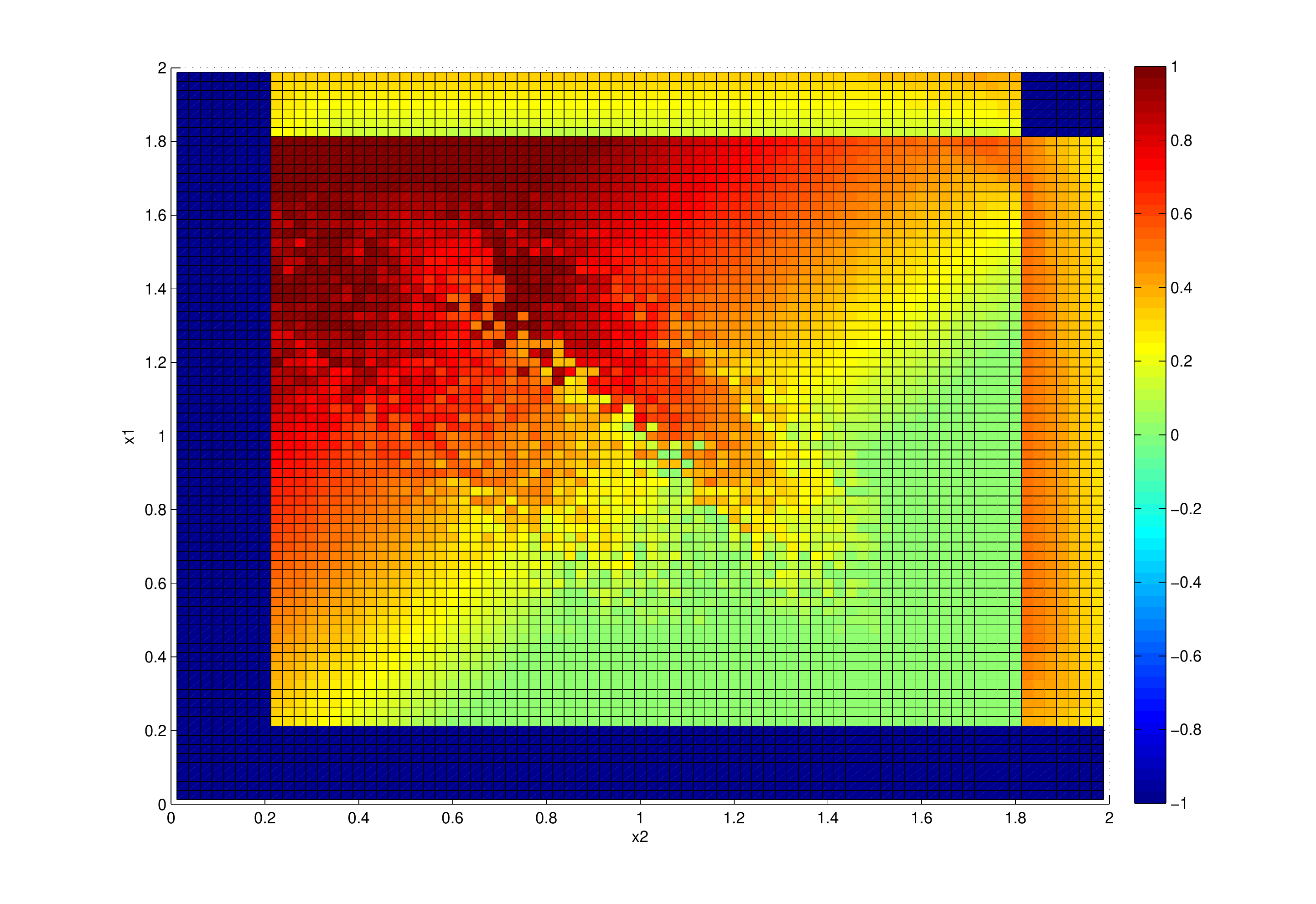}
}
\caption{Optimal DFA policy, component $u$ at time step $50$.}
\label{fig:cs.dfa.u}
\end{figure}

\end{document}